\documentclass[12pt,a4paper,openright]{report}
\usepackage[utf8]{inputenc}
\usepackage[T1]{fontenc}
\usepackage[french]{babel}
\usepackage[hmargin=2cm,vmargin=2cm]{geometry}
\usepackage{here}
\usepackage{amsfonts}
\usepackage{amssymb}
\usepackage{makeidx}
\usepackage{slashed}
\usepackage[Lenny]{fncychap}
\usepackage{multicol}
\usepackage{openbib}
\usepackage{minitoc}
\usepackage[pdftex]{graphicx,color}
\usepackage{setspace}
\usepackage{pslatex}
\usepackage{caption}
\usepackage[pdftex]{thumbpdf}
\usepackage{fancyhdr}
\usepackage{slashbox}
\usepackage{tikz}
\usepackage{shadow}
\usepackage{textcomp}
\usepackage{amsmath}
\usepackage{graphicx}
\usepackage{array} 
\usepackage{hyperref} 
\usepackage{amsmath}
\usepackage{amsthm}
\usepackage{subfig}
\usepackage{times}
\usepackage{lmodern}
\usepackage{color}
\usepackage{lmodern}
\usepackage{multirow}
\usepackage{xspace}
\usepackage{enumerate}
\usepackage{frcursive,mathrsfs}
\usepackage{algorithm2e}
\newtheorem{theorem}{Théorème}
\newtheorem{proposition}{Proposition}
\newtheorem{definition}{Définition}
\newtheorem{remarque}{Remarque}
\newtheorem{lemme}{Lemme}
\newtheorem{propriete}{Propriété}
\newtheorem{corollaire}{Corollaire}

\newtheorem{exemple}{Exemple}
\newtheorem{conclusion}{Conclusion}

\newtheorem{notation}{Notation}

\usepackage{tcolorbox}

\usepackage[]{hyperref}
\hypersetup{colorlinks=true,}

\begin{document}
\begin{titlepage}
\begin{center}
\begin{tabular}{lcl}
\multirow{3}{*}{\includegraphics[width=0.8in]{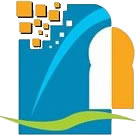}} \hspace*{0,8cm} & \textbf{UNIVERSITE SULTAN MOULAY SLIMANE } &
\multirow{3}{*}{\hspace*{0,3cm} \includegraphics[width=0.8in]{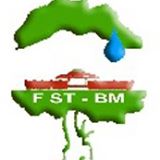}}\\ \hspace*{0,5cm} & \textbf{FACULTE DES SCIENCES ET TECHNIQUES}\\ & \textbf{BENI-MELLAL}
\end{tabular}
\end{center}

\vspace{1cm}

\begin{center}\Huge{\textbf{Rapport de projet de fin d'études}}\end{center}
\vspace{0.3cm}

\begin{center} \textit{Spécialité : Master Génie Mathématique et Applications }\\
\textit{Présenté à la Faculté des sciences et techniques de Béni Mellal}\\
\textit{Université Sultan Moulay Slimane}
\end{center}

\vspace{0,3cm}
\begin{center}\rule{14.6cm}{.12pt}\end{center}
\begin{center}\huge{\textbf{{\color{blue}{La théorié floue  et la dérivation fractionnaire : application aux équations différentielles hybrides}}}}\end{center}
\begin{center}
\rule{14.6cm}{.12pt}
\end{center}
\vspace{0,3cm}
\hspace*{1cm} \textit{Réalisé par:} \hspace*{8cm} \textit{Encadré par:  }\\
\hspace*{1cm} \textbf{- AZIZ EL-GHAZOUANI \hspace*{4.7 cm}- Pr : MHAMED EL-OMARI}\\
\hspace*{11.1 cm} \textbf{- Pr : SAID MELLIANI }\\
\hspace*{11cm}\textbf{- Pr : LALLA SAADIA CHADLI }\\

\begin{center}\textit{Soutenu le 16 juillet 2021 devant la commission d'examen composée de :} \end{center}

\vspace{0,3cm}

\begin{center}\begin{tabular}{llll}
Encadrant \hspace*{0.5cm}: & Pr :  & SAID MELLIANI & FST-Beni Mellal. \\
Encadrant: & Pr :  & LALLA SAADIA CHADLI & FST-Beni Mellal.\\
Encadrant \hspace*{0.3cm}: & Pr :  & MHAMED EL-OMARI  & FP-Beni Mellal.\\
Président: & Pr :  & KHALID HILAL & FST-Beni Mellal.\\
Examinateur: & Pr :  & ADIL ABBASSI & FST-Beni Mellal. \\
Examinateur: & Pr :  & CHAKIR ALLALOU & FST-Beni Mellal.
\end{tabular}\end{center}

\end{titlepage}

\begin{titlepage}
\tableofcontents
\end{titlepage}

\begin{titlepage}
\setcounter{page}{3}
\listoffigures
\addcontentsline{toc}{chapter}{La liste des figures}
\end{titlepage}

\begin{titlepage}
\setcounter{page}{4}
\listoftables
\addcontentsline{toc}{chapter}{La liste des tableaux}
\end{titlepage}
\setcounter{page}{5}
\chapter*{\begin{center} Dédicace\end{center}}
\addcontentsline{toc}{chapter}{Dédicace}
\begin{center}
\begin{Large}
Je dédie ce modeste travail à ceux qui m'ont encourage et soutenu moralement et
matériellement pendant les moments les plus difficiles et durant toute ma vie, et qui
me sont les plus chères sur cette planète :
\end{Large}
\end{center}
\begin{center}
{\large mon père et mon mère.}
\end{center}
\begin{center}
{\large A tous mes amies}
\end{center}
\begin{center}
{\large A tous ceux que j'aime}
\end{center}
\begin{center}
{\large A tous les étudiants de ma promotion}
\end{center}
\begin{center}
{\large Avec l'expression de tous mes sentiments de respect, je dédie ce mémoire.}
\end{center}
\vspace*{6 cm}
\hspace*{13 cm}\textbf{\textit{AZIZ EL-GHAZOUANI}}
\chapter*{\begin{center} Remerciements\end{center}}
\addcontentsline{toc}{chapter}{Remerciements}
Je tiens à témoigner ma reconnaissance à DIEU tout puissant, de m’avoir donner
le courage et la force de mener à terme ce projet. Qui m’a ouvert les portes du savoir.

Je tiens à exprimer ma profonde gratitude et sincères remerciements à mon encadrent \textbf{M. EL-OMARI} pour l’honneur qu’il m’a fait en assurant la direction et le suivi scientifique et technique, pour sa grande contribution à l’aboutissement de ce présent mémoire. Je vous remercie pour votre précieuse présence assistance, votre disponibilité et l’intérêt que vous avez manifésté pour ce modeste travail. Je vous remercie pour vos orientations et votre enthousiasme envers mon travail. Les judicieux conseils et rigueur que vous m’avez prodigué tout au long de ces années de travail m’ont permis de progresser dans mes études. Je vous remercie d’avoir cru en mes capacités et m’avoir fourni d’excelentes conditions me permettant d’aboutir à la production de ce mémoire qui n’aurait vu le jour sans votre confiance et votre générosité. Je vous remercie très chaleureusement de m’avoir continuellement encouragée, pour votre soutien scientifique et humain, pour votre gentillesse et votre hospitalité.
Je voudrais vous remercier très vivement de m’avoir fait découvrir le monde de la
recherche. Je vous suis très reconnaissant pour la confiance que vous m’avez témoigné
tout au long de la réalisation de ce mémoire. Merci pour la vivacité et la force de volonté
que vous avez su me transmettre. vous avez su me secouer aux moments ou j’en ai
vraiment besoin.

Je remercie vivement monsieur \textbf{S. MELLIANI} pour l’honneur qu’il me fait en assurant la direction et le suivi scientifique et technique de ce mémoire, Je le remercie chaleureusement pour sa participation à ma formation et de m’avoir tracer le chemin. Je le remercie
vivement autant que directeur du Laboratoire de Mathématiques Appliqué et calcul scientifiques (LMACS), pour m’avoir accuilli et mis à ma disposition tous les moyens nécessaires
pour mener ce mémoire à son terme.

Mes remerciement s’adressent aussi au Pr \textbf{L. S. CHADLI} pour l’honneur qu’il me fait en assurant la direction et le suivi scientifique et technique de ce mémoire et pour sa contribution à ma formation.

Je tiens également à exprimer mes vifs remerciements au Pr \textbf{A. ABBASSI},
pour sa contribution à ma formation et en acceptant de juger mon travail. 

Mes remerciements s’addressent aussi à monsieur \textbf{K. HILAL} pour l’honneur qu’elle me fait
en acceptant de juger et d’évaluer mon travail.

Je tiens également à exprimer mes vifs remerciements au Pr \textbf{C. ALLALOU},
pour sa contribution à ma formation et en acceptant de juger mon mémoire.

Enfin je m’incline respectueusement devant les deux êtres à qui je dois l’existence,
mon père et ma mère. Je leur exprime mes hauts et profonds signes de reconnaissances
et d’obéissance pour tous les efforts qu’ils ont fournis et tous les sacrifices qu’ils ont
généreusement faits, pour que je grandisse dans de parfaites conditions d’amour, de
satisfaction et d’épanouissement. je vous remercie mes chers parents qui sans vous je
n’arriverai pas là ou je suis actuellement, vous qui m’avez toujours guidés mes pas,
c’est donc à vous que je dédie ce fruit de mon travail.

\newpage
\hspace*{8.5cm} {\small “Les mathématiques sont une gymnastique }

\hspace*{8.5cm} {\small de l'esprit et une préparation à la philosophie.”}

\hspace*{13cm} {\small \textbf{Eric Temple Bell.}}

\hspace*{9cm} {\small “La musique est une mathématique sonore,}

\hspace*{9cm} {\small la mathématique est une musique silencieuse.”}

\hspace*{13cm} {\small  \textbf{Edouard HERRIOT.}}
\chapter*{\begin{center} Introduction \end{center}}
\addcontentsline{toc}{chapter}{Introduction}
\textbf{La théorie de dérivation fractionnaire} est un sujet presque aussi ancien que
le calcul classique tel que nous le connaissons aujourd'hui, ces origines remontent 
à la fin du $17^{\text {ème }}$ siècle, l'époque où Newton et Leibniz ont développé les fondements de calcul différentiel et intégral. En particulier, Leibniz a présenté le symbole $\frac{d^{n} f}{d t^{n}}$ pour désigner la $n^{\text {ème }}$ dérivée d'une fonction $f$. Quand il a annoncé dans une lettre à l'Hôpital
(apparemment avec l'hypothèse implicite que $n \in \mathbb{N}$ ), l'Hôpital a répondu :
\begin{figure}[h!]
\centering
\includegraphics[scale=0.9]{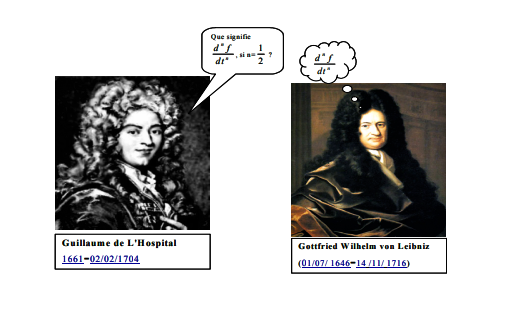}
\end{figure}

Que signifie $\frac{d^{n} f}{d t^{n}}$ si $n=\frac{1}{2} ?$

Cette lettre de l'Hôpital, écrite en 1695 , est aujourd'hui admise comme le premier
incident de ce que nous appelons la dérivation fractionnaire, et le fait que l'Hôpital a demandé spécifiquement pour $n=\frac{1}{2}$, c'est-à-dire une fraction (nombre rationnel) a en
fait donné lieu au nom de cette partie des mathématiques.

Une liste de mathématiciens qui ont fourni des contributions importantes au caclul fractionnaire jusqu'au milieu du $20^{\text {ème }}$ siècle, inclut :

P.S. Laplace (1812), J.B.J. Fourier (1822), N.H. Abel (1823-1826), J. Liouville (1832-
1873), B. Riemann (1847), H. Holmgren (1865-67), A.K. Grunwald (1867-1872), A.V.
Letnikov (1868-1872), H. Laurent (1884), P.A. Nekrassov (1888), A. Krug (1890), J.
Hadamard (1892), O. Heaviside (1892-1912) S. Pincherle (1902), G.H. Hardy et J.E.
Littlewood (1917-1928), H. Weyl (1917), P. L'evy (1923), A. Marchaud (1927), H.T.
Davis (1924-1936), A. Zygmund (1935-1945) E.R. Amour (1938-1996), A. Erd'elyi (1939-
1965), H. Kober (1940), D.V. Widder (1941), M. Riesz (1949).

Cependant, cette théorie peut être considérée comme un sujet nouveau aussi, depuis
seulement un peu plus de trente années elle a été objet de conférences spécialisées. Pour la
première conférence, le mérite est attribué à B. Ross qui a organisé la première conférence
sur les calculs fractionnaires et ses applications à l'université de New Haven en juin 1974 ,
et il a édité les débats. Pour la première monographie le mérite est attribué à K.B.
Oldham et J. Spanier, qui ont publié un livre consacré au calcul fractionnaire en 1974
après une collaboration commune, commencé en 1968 .

\textbf{Une autre théorie} se développe en parallèle de la dérivation fractionnaire telle est \textbf{la théorie des sous-ensembles flous}. C´est une théorie mathématique du domaine de
l'algèbre abstraite. Elle a été développée par Lotfi Zadeh en 1965.\\
 Le concept de nombres flous et d'opérations arithmétiques floues a été introduit
par Zadeh \cite{Ref1} et Dubois et Prade \cite{Ref3}. Depuis lors, plusieurs auteurs ont étudié les propriétés et les applications proposées des nombres flous $\cite{Ref4}$.\\
 L'une des principales applications de l'arithmétique des nombres flous est le traitement de systèmes linéaires flous et de systèmes linéaires complètement flous, et plusieurs problèmes dans divers domaines tels que l'économie, l'ingénierie et la physique se résument à la résolution d'un système linéaire d'équations.\\
 Les nombres flous sont utilisés dans les statistiques, la programmation informatique, l'ingénierie (en particulier les communications) et la science expérimentale. Le
concept prend en compte le fait que tous les phénomènes dans l'univers physique
ont un degré d'incertitude inhérente. En général, les opérations arithmétiques sur les
nombres flous peuvent être approchées soit par l'utilisation directe de la fonction d'appartenance (par le principe d'extension de Zadeh), soit par l'utilisation équivalente
de la représentation $\alpha$-coupe.\\
 Le but de \textbf{ce mémoire} est d´étudier les sous-ensembles floues, les nombres flous et leurs arithmétiques et les équations différentielles floues et leurs applications.Il contient quatre chapitres.\\
 \textbf{Dans le premier chapitre}, nous allons aborder les sous ensembles flous, ensuite nous donnons un exemple de ces sous-ensembles : il s´agit de sous-ensemble floue de $\mathbb{R}$, Il est divisé comme suit :
 \begin{itemize}
 \item[*] La section 1 : Sera résrevée aux Généralités sur les ensembles classiques,
 \item[*] La section 2 : Nous donnons des propriétées sur les sous-ensembles flous,
 \item[*] La section 3 : Nous donnons Les $ \alpha $-coupes d’un sous-ensemble flou,
 \item[*] La section 4 : Sera résrevée pour le produit cartésien de sous-ensembles flous,
 \item[*] La section 5 : Nous allons aborder les sous-ensembles flous convexes,
 \item[*] La section 6 : Nous parlons du principe d´extension de Zadeh,
 \item[*] La section 7 : Nous donnons les propriétées de sous ensemble flou de $\mathbb{R}$,
 \item[*] La section 8 : est consacrée pour les semi-groupes flou.
 \end{itemize}
 \textbf{Au deuxième chapitre} nous donnons une étude de l´arithmétique des nombtres flous.
 Plus précisément on va étudier :
 \begin{itemize}
 \item[*] La différence généralisé de HUKUHARA dans le cas des intervalles compacts dans $\mathbb{R}^{n}$,
 \item[*] Différence généralisé de HUKUHARA des nombres flous,
 \item[*] Division généralisée.
 \end{itemize}
 \textbf{Le troisième chapitre} s’étend à rappeler quelques résultats fondamentaux sur la dérivation fractionnaire, et elle traite aussi les équations différentielles floues.
  Plus précisément est divisé comme suit :
   \begin{itemize}
 \item[*] La section 1 : Sera résrevée aux outil de base,
 \item[*] La section 2 : Nous donnons la définition de l´intégration fractionnaire,
 \item[*] La section 3 : Nous allons aborder diverses Les approches des dérivées fractionnaires,
 \item[*] La section 4 : Nous donnons quelques propriétés des dérivées fractionnaires,
 \item[*] La section 5 : Est consacrée pour les équations différentielles floues.
 \end{itemize}
 \textbf{Le quatrième chapitre}  est consacré à une étude de problème fractionnaire suivant : 
$$
\left\{\begin{array}{l}
\frac{\mathrm{d}}{\mathrm{d} t}\left[\frac{u(t)}{f(t, u(t))}\right]=g(t, u(t)), \quad \text { a.e. } t \in J \\
u\left(t_{0}\right)=u_{0} \in \mathbb{R}
\end{array}\right.
$$
où $f \in C(J \times \mathbb{R}, \mathbb{R} \backslash\{0\})$ et $g \in \mathcal{C}(J \times \mathbb{R}, \mathbb{R})$,\\
Et les applications des équations différentielles floues,\\
Enfin, cette mémoire est clôturée par une \textbf{bibliographie}.
\chapter{Généralités sur les sous ensembles flous}
Les sous-ensembles flous  constituent une généralisation de notion
d'ensemble classique et ont été introduits par Lotfi Zadeh en $1965$ \cite{Ref5}.\\
Dans ce chapitre, nous donnons un aperçu sur la théorie des sous-ensembles flous
en présentant ses concepts de base ainsi que les opérations les plus couramment
utilisées. 
\section{Généralités sur les ensembles classiques}
Dans cette section, nous présentons quelques définitions et introduire la notation nécessaire, Avant d’entamer la définition de sous-ensemble flou on procède à définir l’ensemble classique.
\begin{definition}(Ensemble classique)\\
 Un ensemble classique $A$ de lensemble de référence $X$ est défni par une fonction caractéristique $\chi_{A}$ qui prend la valeur 0 pour les éléments de $X$ n'appartenant pas à $A$ et la valeur 1 pour ceux qui appartiennent à $A$
$$
\begin{array}{c}
\chi_{A}: X \rightarrow\{0,1\} \\
x \longmapsto\left\{\begin{array}{l}
1 \text { si } x \in A \\
0 \text { si } x \notin A
\end{array}\right.
\end{array}
$$
\end{definition}
\begin{exemple}
Soient $X=\mathbb{R}$ l'ensemble de référence et $A$ l'ensemble des nombres
compris entre 4 et 10 est caractérisé par la fonction caractéristique suivante
$$
\begin{array}{c}
\chi_{A}: \mathbb{R} \rightarrow\{0,1\} \\
\chi_{A}(x)=\left\{\begin{array}{l}
1 \text { si } x \in A \\
0 \text { si } x \notin A
\end{array}\right.
\end{array}
$$
\end{exemple}
\subsection{Opérations algébriques sur les ensembles}
Soit A; B deux ensemble de référentiel X,
\paragraph*{• L'inclusion :}
 on dit qu'un ensemble $A$ est inclus dans l'ensemble $B$, ou encore
que $A$ est un sous-ensemble ou une partie de $B$ si
$$
\forall x \in X,(x \in A) \Rightarrow(x \in B)
$$
On écrit alors $A \subseteq B$.
Si les relations suivantes sont satisfaites entre les deux ensembles $A$ et $B, A \subseteq B$
et $A \neq B$, alors $B$ a des éléments qui n'appartiennent pas à $A .$ Dans ce cas, $A$ est
appelé un sous-ensemble propre de $B$, et cette relation est désignée par : $A \subset B$.
\paragraph*{• L'égalité d'ensembles :}
soit deux ensembles $A$ et $B$ qui contiennent les mêmes
éléments sont dits égaux, et on écrit $A=B$
$$
A=B \Leftrightarrow A \subseteq B \text { et } B \subseteq A
$$
Dans le cas contraire on dit qu'ils sont distincts et on note $A \neq B$.

\paragraph*{• L'intersection :}
 l'intersection des ensembles $A$ et $B$ se compose des éléments communs aux deux ensembles $A$ et $B$.
$$
A \cap B=\{x \mid x \in A \text { et } x \in B\}
$$
L'intersection peut être généralisée entre les ensembles dans une famille d'ensembles
$$
\cap_{i \in I} A_{i}=\left\{x \mid x \in A_{i}, \forall i \in I\right\},
$$
où $\left\{A_{i} \mid i \in I\right\}$ est une famille d'ensembles.
L'union des ensembles $A$ et $B$ est l'ensemble des éléments qui appartiennent à $A$ ou à $B$
$$
A \cup B=\{x \mid x \in A \text { ou } x \in B\}
$$
\paragraph*{• L'union : } pourrait être définie entre plusieurs ensembles. Par exemple, la réunion
des ensembles de la famille suivante peuvent être définis comme suit.
$$
\cup_{i \in I} A_{i}=\{x \mid x \text { pour certains } i \in I\},
$$
où la famille d'ensembles est $\left\{A_{i} \mid i \in I\right\}$.
\paragraph*{• La différence : }  La différence de $B$ et $A$ noté $B-A$ ou $B \backslash A$ est constituée des
éléments qui sont en $B$, mais pas dans $A .$ C'est -à- dire :
$$
B-A=\{x \mid x \in B, x \notin A\}
$$
\paragraph*{• Le complément : }
 Soit $A$ un sous-ensemble de l'ensemble référentiel $X .$ Alors le
complément de $A$, noté $A^{c}$, est l'ensemble des éléments qui appartiennent à $X$
mais qui n'appartiennent pas à $A .$ De façon plus concise nous écrivons
$$
A^{c}=X-A=\{x \in X \text { tel que } x \notin A\}
$$
\begin{exemple}
Soit un référentiel $X=\{1,2,3, \ldots, 10\} .$ et soient $A, B$ deux sous-
ensembles de $X$ donnée par :
$$
A=\{2,3,4,5,6,7\}, B=\{4,6,8\}
$$
Alors on obtient :
$$
\begin{array}{l}
A \cap B=\{4,6\} \\
A \cup B=\{2,3,4,5,6,7,8\}, \\
A^{c}=\{1,8,9,10\}, \\
B^{c}=\{1,2,3,5,7,9,10\}, \\
B-A=\{8\}
\end{array}
$$
\end{exemple}

\paragraph*{• Propriétés des opérations sur les ensembles classiques :\\}
Dans ce tableau on va cité les différents propriétés des opérations sur les ensembles classiques :
\vspace*{3cm}
\begin{table}[!h]
$$
\begin{array}{|l|l|}
\hline \text { Involution } & \left(A^{c}\right)^{c}=A \\
\hline \text { Commutativité } & A \cup B=B \cup A \\
& A \cap B=B \cap A \\
\hline \text { Associativité } & A \cup(B \cup C)=(A \cup B) \cup C \\
& A \cap(B \cap C)=(A \cap B) \cap C \\
\hline \text { Distributivité } & A \cup(B \cap C)=(A \cup B) \cap(A \cup C) \\
& A \cap(B \cup C)=(A \cap B) \cup(A \cap C) \\
\hline \text { Idempotence } & A \cup A=A \\
& A \cap A=A \\
\hline \text { Absorption } & A \cup(A \cap B)=A \\
& A \cap(B \cup C)=A \\
\hline \text { Absorption avec } \varnothing \text { et } X & A \cup X=X \\
& A \cap \varnothing=\varnothing \\
\hline \text { Identité } & \begin{array}{l}
A \cup \varnothing=A \\
A \cap X=A
\end{array} \\
\hline \text { Loi de De Morgan } & \begin{array}{l}
(A \cup B)^{c}=A^{c} \cap B^{c} \\
(A \cap B)^{c}=A^{c} \cup B^{c}
\end{array} \\
\hline \text { Absorption de complément } & \begin{array}{l}
A \cup\left(A^{c} \cap B\right)=A \cup B \\
A \cap\left(A^{c} \cup B\right)=A \cap B
\end{array} \\
\hline \text { Loi de non contradiction } & A \cap A^{c}=\varnothing \\
\hline \text { Tiers exclu } & A \cup A^{c}=X \\
\hline
\end{array}
$$
\caption{Propriétés des opérations sur les ensembles classiques}
\end{table}
\section{Sous-ensembles flous}
Le concept de sous-ensemble flou constitue un assouplissement de celui de sous-ensemble d'un ensemble donné. Il n'existe pas d'ensembe flou au sens propre, tous les ensembles considérés étant classiques et bien définis. On utilise toutefois souvent le terme d'ensemble flou au lieu de sous-ensemble flou, par abus de langage, conformément à la traduction du terme original de "fuzzy set", que l'on oppose au «crisp set» désignant un sous-ensemble non flou. Pour un langage mathématique acceptable,nous utilisons indifféremment le terme sous-ensemble flou et ensemble flou.
\begin{definition}
Soit $X$ un ensemble de référence(classique), un sous-ensemble
flou $A$ de $X$ est défini par une fonction d'appartenance qui associe  chaque élément
$x$ de $X$, le degré $\mu_{A}(x)$, compris entre 0 et 1, avec lequel $x$ appartient à $A$
$$
\mu_{A}: X \rightarrow[0,1]
$$
\end{definition}
Ainsi, un sous-ensemble flou est toujours (et seulement) une fonction de $X$ dans $[0,1] .$ Par exemple, le graphe d'une fonction d'appartenance possible pour l'ensemble $B$ est donné dans la figure $2.2 .$
\begin{figure}[h!]
\begin{center}
\includegraphics[scale=0.6]{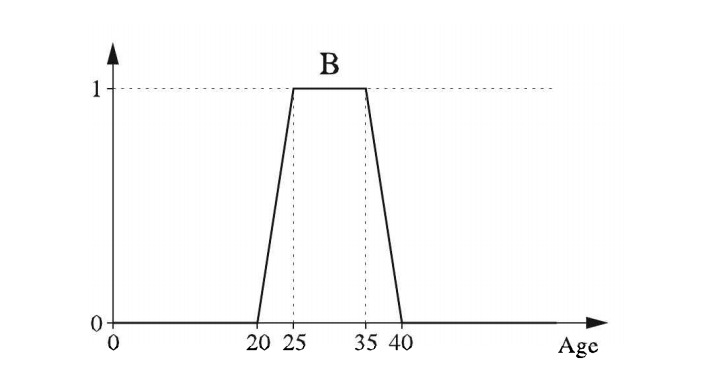}
\caption{La fonction d´apparteneance de l´ensemble B .}
\end{center}
\end{figure}
\begin{remarque}
Dans le cas particulier où $\mu_{A}$ ne prend que des valeurs égales à
0 ou 1 , le sous-ensemble flou $A$ est un sous-ensemble classique de $X$. Un sous-ensemble
classique est donc un cas pariculier d'un sous-ensemble flou.
\end{remarque}

Les cas extrêmes de sous-ensembles flous de $\mathrm{X}$ sont respectivement X lui-même, associé à une fonction d'appartenance $\mu_{X}$ prenant la valeur 1 pour tous les éléments de $X$, et l'ensemble vide $\varnothing$, associé à une fonction d'appartenance nulle sur tout $X$.
\begin{notation}
Soit $X$ l'ensemble de référence.\\
1. Un sous-ensemble flou $A=\left\{\left(x, \mu_{A}(x)\right), x \in X\right\}$ de $X$, noté par $A=\sum_{x \in A} \mu_{A}(x) / x$ si $X$ est dénombrable et par $A=\int_{x} \mu_{A}(x) / x$, si $X$ et non dénombrable.\\
2. L'ensemble de tous les sous-ensembles flous de $X$ noté par $F(X)$.
\end{notation}
\begin{exemple}
Soit l'univers $X$ des âges de 10 à 70 ans, on écrit $X=\{10,20,30,40,50,60,70\}$.
On peut définir les notions de "jeune" et de "vieux" respectivement par les ensembles flous $A$ et $B$ suivants :
$$
\begin{array}{l}
A=\{<10,1>,<20,0.8>,<30,0.6>,<40,0.2>,<50,0.1>,<60,0>,<70,0>\} \\
B=\{<10,0,>,<20,0.1,>,<30,0.3><40,0.5>,<50,0.7>,<60,0.9>,<70,1>\}.
\end{array}
$$
\end{exemple}
\subsection{Caractéristiques des sous ensembles flous}
Un sous-ensemble flou est complètement défini par la donnée de sa fonction d'appartenance. A partir d'une telle fonction, un certain nombre de caractéristiques du sous-ensemble flou peuvent être étudiées.
\paragraph*{• Noyau :}
Le noyau d'un sous-ensemble flou $A$ de $X$, noté $N o y(A)$, est l'ensemble de tous les éléments qui lui appartiennent totalement. Formellement:
$$
\operatorname{Noy}(A)=\left\{x \in X \mid \mu_{A}(x)=1\right\}
$$
\paragraph*{• Support :}
Le support d'un sous-ensemble flou $A$ de $X$, noté $\operatorname{Supp}(A)$, est l'ensemble de tous les éléments qui lui appartiennent au moins un petit peu. Formellement:
$$
\operatorname{Supp}(A)=\left\{x \in X \mid \mu_{A}(x)>0\right\}
$$
\paragraph*{• Hauteur :}
La hauteur d'un sous-ensemble flou $A$ de $X$, notée $h(A)$, est la valeur maximale atteinte sur le support de $A .$ Formellement :
$$
h(A)=\sup _{x \in X} \mu_{A}(x)
$$
\begin{definition}
On dira alors qu'un sous-ensemble flou est normalisé si sa hauteur $h(A)$ est égale à $1 .$
\end{definition}
\begin{figure}[h!]
\begin{center}
\includegraphics[scale=0.6]{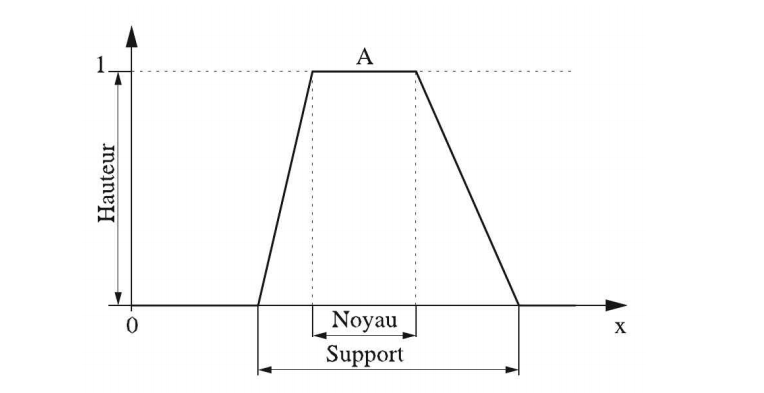}
\caption{Noyau,support et hauteur d´un sous ensemble flou.}
\end{center}
\end{figure}

\begin{exemple}
 Le noyau $\operatorname{Noy}(\boldsymbol{B})$ du sous-ensemble flou $B$ de la figure $2.1$ est l'intervalle $[25,35]$, son support $\operatorname{Supp}(B)$ est l'intervalle $] 20,40[$ et sa hauteur $h(B)=1$.

\end{exemple}
\paragraph*{• Cardinalité :}
La cardinalité d'un sous-ensemble flou $A$ de $X$, notée $|A|$, est le nombre d'éléments appartenant à $A$ pondéré par leur degré d'appartenance. Formellement, pour $A$ fini:
$$
|A|=\sum_{x \in X} \mu_{A}(x).
$$

\subsection{Opérations sur les sous-ensembles flous}
Le concept de sous-ensemble flou de l'ensemble $X$ étant une généralisation de la
notion de sous-ensemble classique de $X$, ces opérations sont choisies de façon à être
équivalentes aux opérations classiques de la théorie des ensembles lorsque les fonctions
d'appartenance ne prennent que les valeurs 0 ou $1 .$\\
Etant donné deux sous-ensembles flous $A$ et $B$ de $X$. Pour plus de détail voir $\cite{Ref6}$
\paragraph*{• Egalité :}
Deux sous-ensembles flous $A$ et $B$ de $X$ sont dits égaux s'ils ont des fonctions d'appartenance égales en tout point de $X .$ Formellement, $A=B$ si et seulement si :
$$
\forall x \in X, \mu_{A}(x)=\mu_{B}(x).
$$
\paragraph*{• Inclusion :}
Soient $A$ et $B$ deux sous-ensembles flous de $X .$ Si pour n'importe quel élément $x$ de $X, x$ appartient toujours moins à $A$ qu'à $B$, alors on dit que $A$ est inclus dans $B$ $(A \subseteq B) .$ Formellement, $A \subseteq B$ si et seulement si :
$$
\forall x \in X, \mu_{A}(x) \leq \mu_{B}(x).
$$
\paragraph*{• Union :}
L'union de deux sous-ensembles flous $A$ et $B$ de $X$ est le sous-ensemble flou constitué des éléments de $X$ affectés du plus grand des degrés avec lesquels ils appartiennent à $A$ et $B$. Formellement, $A \cup B$ est donné par:
$$
\mu_{A \cup B}(x)=\max \left(\mu_{A}(x), \mu_{B}(x)\right).
$$
\paragraph*{• Intersection}
L'intersection de deux sous-ensembles flous $A$ et $B$ de $X$ est le sous-ensemble flou constitué des éléments de $X$ affectés du plus petit des degrés avec lesquels ils appartiennent à $A$ et $B .$ Formellement, $A \cap B$ est donné par :
$$
\mu_{A \cap B}(x)=\min \left(\mu_{A}(x), \mu_{B}(x)\right)
$$
\paragraph*{• Propriété de l´union et d´intersection}
La plupart des propriétés des opérations sur les sous-ensembles classiques restent valides pour les opérations sur les sous-ensembles flous, c'est-à-dire que si $A$ et $B$ sont des sous-ensembles flous de $X$, on a :\\
- associativité de $\cap$ et $\cup$,\\
- commutativité de $\cap$ et $\cup$,\\
- $ A \cap X=A, A \cap \emptyset=\emptyset$,\\
- $ A \cup X=X, A \cup \emptyset=A$,\\
- $ A \cap B \subseteq A \subseteq A \cup B$\\
- $ A \cap\left(B \cup B^{\prime}\right)=(A \cap B) \cup\left(A \cap B^{\prime}\right)$\\
- $ A \cup\left(B \cap B^{\prime}\right)=(A \cup B) \cap\left(A \cup B^{\prime}\right)$\\
- $|A|+|B|=|A \cap B|+|A \cup B| .$\\

\paragraph*{• Complément :}
Le complément d'un sous-ensemble flou $A$ de $X$ est noté $\bar{A}$. Il est défini à partir de la fonction d'appartenance de $A$ par :
$$
\forall x \in X, \mu_{\bar{A}}(x)=1-\mu_{A}(x).
$$
\paragraph*{• Propriété de la complétion}
En général, contrairement aux sous-ensembles classiques, la propriété de noncontradiction n'est pas satisfaite par les sous-ensembles flous $(A \cap \bar{A} \neq \emptyset) .$ De même, la propriété du tiers exclus n'est pas satisfaite non plus $(A \cup \bar{A} \neq X) .$ Les autres propriétés sont conservées:\\
$-\overline{A \cap B}=\bar{A} \cup \bar{B}, \overline{A \cup B}=\bar{A} \cap \bar{B}$ (lois de De Morgan),\\
$-\overline{\bar{A}}=A$\\
$-\bar{\emptyset}=X$\\
$-\bar{X}=\emptyset$\\
$-|A|+|\bar{A}|=|X|$ si $X$ est fini.
\paragraph*{• L'addition $A+B$ : }
$$
A+B=\left\{\left\langle x, \mu_{A}(x)+\mu_{B}(x)-\mu_{A}(x) \cdot \mu_{B}(x)\right\rangle \mid x \in X\right\}
$$
\paragraph*{• La multiplication $A \cdot B$ : }
$$
A \cdot B=\left\{\left\langle x, \mu_{A}(x) \cdot \mu_{B}(x)\right\rangle \mid x \in X\right\}
$$
\paragraph*{• La différence :}
 la différence de deux ensembles flous $A$ et $B$ dans $X$, est l'ensemble
flou $A-B=A \cap B^{c}$ dont la fonction d'appartenance est
$$
\mu_{A-B}(x)=\min \left\{\mu_{A}(x), 1-\mu_{B}(x)\right\}
$$
\begin{remarque}
Contrairement aux ensembles classiques, il vérifie généralement
$A^{c} \cap A \neq \varnothing$ et $A^{c} \cup A \neq X$, c'est-à-dire qu'il ne vérifie pas les propriétés classiques de la non-contradiction et du tiers exclu. Les autres propriétés de la théorie des ensembles classiques sont cependant satisfaites.
\end{remarque}
\section{Les $\alpha-$coupes d’un sous-ensemble flou}
Il peut être utile de décrire un sous-ensemble flou en se référant à des sous-ensembles ordinaires. Une façon de réaliser une approximation d'un sous-ensemble flou consiste à fixer un seuil inférieur sur les degrés d'appartenance. Le sous-ensemble ordinaire $A_{\alpha}$ de $X$ associé à $A$ pour le seuil $\alpha$ est l'ensemble des éléments qui appartiennent à $A$ avec un degré au moins égal à $\alpha .$ On dit que $A_{\alpha}$ est $l^{\prime} \alpha$-coupe de $A .$ Formellement :
$$
A_{\alpha}=\left\{x \in X \mid \mu_{A}(x) \geq \alpha\right\}.
$$
et $A_{\alpha}$ est un sous-ensemble ordinaire de fonction caractéristique :
$$
\chi_{A_{\alpha}}(x)=\left\{\begin{array}{ll}
1 & \text { si } \mu_{A}(x) \geq \alpha \\
0 & \text { sinon }
\end{array}\right.
$$
\begin{definition}
Pour tout niveau $\alpha$ de $[0,1]$, on définit la $\alpha$-coupe stricte de $A$
comme le sous-ensemble
$$
A^{\alpha}=\left\{x \in X / \mu_{A}(x)>\alpha\right\}
$$
\end{definition}
\begin{exemple}
Soit A un ensemble flou donnée par : \\
$A=\{<x, 0.3>,<y, 1.0>,<z, 0.6>\}$,\\
$A_{0.6}=\left\{x \in X \mid \mu_{A}(x) \geq 0.6\right\}=\{y, z\}$,\\
$A_{0.7}=\left\{x \in X \mid \mu_{A}(x) \geq 0.7\right\}=\{y\}$,\\
$A^{0.6}=\left\{x \in X \mid \mu_{A}(x)>0.6\right\}=\{y\}$
\end{exemple}
\begin{proposition}
Soient A, B deux sous-ensembles flous alors :\\
1. Si $A \subset B$ alors $A_{\alpha} \subset B_{\alpha}$,\\
2. $(A \cap B)_{\alpha}=A_{\alpha} \cap B_{\alpha}$,\\
3. $(A \cup B)_{\alpha}=A_{\alpha} \cup B_{\alpha}$.
\end{proposition}
\begin{theorem}(Théorème de décomposition)\\
Tout sous-ensemble flou $A$ de l'ensemble de référence $X$ est défini à partir de ses $\alpha$ -coupes par :
$$
\forall x \in X \mid \mu_{A}(x)=\sup _{\alpha \in[0,1]}\left(\alpha \cdot \chi_{A_{\alpha}}(x)\right)
$$
où "sup" indique la borne supérieure des valeures possibles et $\chi_{A_{\alpha}}$ est la fonction caractéristique de $A_{\alpha} .$
\end{theorem}
\begin{proof}
Soit $x \in X$, supposons $\mu_{A}(x)=\beta\left(\beta \in[0,1], x \in A .\right.$ Donc 
$$
\mu_{A}(x)=\beta \cdot \chi_{A_{\beta}}(x) \text { et } \mu_{A}(x) \leq \sup _{\alpha \in[0,1]}\left(\alpha \cdot \chi_{A_{\alpha}}(x)\right)
$$
Réciproquement, soit $x \in X$ pour tout niveau $\beta$ on a
$$
\left\{\begin{array}{lll}
\chi_{A_{\beta}}(x)=0 & \mathrm{si} & \mu_{\alpha}(x)<\beta \\
\chi_{A_{\beta}}(x)=1 & \mathrm{si} & \mu_{\alpha}(x) \geq \beta
\end{array}\right.
$$
Donc,
$$
\left\{\begin{array}{r}
\beta \cdot \chi_{A_{\beta}}(x)=\beta \text { si } \mu_{\alpha}(x) \geq \beta \\
\beta \cdot \chi_{A_{\beta}}(x)=0 \quad \text { si } \quad \mu_{\alpha}(x)<\beta
\end{array}\right.
$$
Donc dans les deux cas, $\beta \cdot \chi_{A_{\beta}}(x) \leq \mu_{A}(x) .$ D'où, $\sup _{\alpha \in[0,1]} \alpha \cdot \chi_{A_{\alpha}}(x) \leq \mu_{A}(x)$.
Par conséquent, pour tout $x \in X, \mu_{A}(x)=\sup _{\alpha \in[0,1]}\left(\alpha \cdot \chi_{A_{\alpha}}(x)\right)$.
\end{proof}

\section{Produit cartésien de sous-ensembles flous}
Lorsque la situation nécessite plusieurs ensembles de référence $X_{1}, \ldots, X_{r}$, on est amené à considérer un univers global $X$ composé des différents ensembles de référence $:$ le produit cartésien $X=X_{1} \times, \ldots, \times X_{r} .$ Les éléments de $X$ sont des
$$
r \text { -uplets }\left(x_{1}, \ldots, x_{r}\right) \text { avec } x_{1} \in X_{1}, \ldots, x_{r} \in X_{r}
$$
Le produit cartésien de $r$ sous-ensembles flous $A_{1}, \ldots A_{r}$, définis respectivement sur $X_{1}, \ldots, X_{r}$, est le sous-ensemble flou $A$ de $X$ de fonction d'appartenance :
$$
\forall x=\left(x_{1}, \ldots, x_{r}\right), \mu_{A}(x)=\min \left(\mu_{A_{1}}\left(x_{1}\right), \ldots, \mu_{A_{r}}\left(x_{r}\right)\right)
$$
\begin{exemple}
Soient $X_{1}=\{x, y, z\}, X_{2}=\{\alpha, \beta\}$ et soient $A, B$ deux sous-ensembles
flous réspèctivement définis sur $X_{1}, X_{2}$ donnée par :\\
$
\begin{array}{l}
A=\{<x, 0.7>,<y, 0.5>,<z, 1.0>\} \\
B=\{<\alpha, 0.7>,<\beta, 0.4>\}, \\
\end{array}
$\\
Alors on obtient :
$$
A \times B=\left\{\begin{array}{l}
<(x, \alpha), 0.7>,<(x, \beta), 0.4>,<(y, \alpha), 0.5> \\
<(y, \beta), 0.4>,<(z, \alpha), 0.7>,<(z, \beta), 0.4>
\end{array}\right\}
$$
\end{exemple}
\section{Sous-ensembles flous convexes}
\begin{definition}
Un sous-ensemble flou $A$ de l'ensemble $X$ des nombre réels est
convexe si, pour tout couple d'élément $a$ et $b$ de $X$, et pour tout nombre $\lambda$ de $[0,1]$,
la fonction d'appartenance de $A$ vérifie
$$
\mu_{A}(\lambda a+(1-\lambda) b) \geq \min \left(\mu_{A}(a), \mu_{A}(b)\right)
$$
\end{definition}
\begin{exemple}
Soit le sous-ensemble flou $A$ de $\mathbb{R}$ tel que
$$
\mu_{A}(x)=\left\{\begin{array}{lr}
0 & \text { si } x \leq 10 \\
\left(1+(x-10)^{-2}\right)^{-1} & \text { si } x>10
\end{array}\right.
$$
Alors, A est un sous-ensemble flou convexes de $\mathbb{R}$.
\end{exemple}
\begin{propriete}
Un sous-ensemble flou $A$ de $\mathbb{R}$ est convexe si toutes ses $\alpha$-coupes
$A_{\alpha}$ sont convexes, c'est-à-dire si, pour tout couple d'éléments a et $b$ de $A_{\alpha}$ et pour
tout nombre $\lambda$ de $[0,1], x=\lambda a+(1-\lambda) b$ appartient aussi à $A_{\alpha} .$
\end{propriete}
\begin{theorem}
Si $A$ et $B$ sont deux sous-ensembles flous convexes de $\mathbb{R}$, leur intersection et convexe.
\end{theorem}
\begin{proof}
Soit $C=A \cap B$, alors
$$
\mu_{C}(\lambda a+(1-\lambda) b)=\min \left(\mu_{A}(\lambda a+(1-\lambda) b), \mu_{B}(\lambda a+(1-\lambda) b)\right)
$$
Maintenant, puisque $A$ et $B$ sont convexes
$$
\begin{aligned}
\mu_{A}(\lambda a+(1-\lambda) b) & \geq \min \left(\mu_{A}(a), \mu_{A}(b)\right) \\
\mu_{B}(\lambda a+(1-\lambda) b) & \geq \min \left(\mu_{B}(a), \mu_{B}(b)\right),
\end{aligned}
$$
et donc
$$
\mu_{C}(\lambda a+(1-\lambda) b) \geq \min \left(\operatorname { m i n } \left(\mu_{A}(a), \mu_{A}(b), \min \left(\mu_{B}(a), \mu_{B}(b)\right)\right.\right.
$$
ou équivalent
$$
\mu_{C}(\lambda a+(1-\lambda) b) \geq \min \left(\operatorname { m i n } \left(\mu_{A}(a), \mu_{B}(a), \min \left(\mu_{A}(b), \mu_{B}(b)\right)\right.\right.
$$
et ainsi
$$
\mu_{C}(\lambda a+(1-\lambda) b) \geq \min \left(\mu_{C}(a), \mu_{C}(b)\right).
$$
\end{proof}
\section{Principe d´extension de Zadeh}
Un principe d'extension est utilisé pour étendre une fonction mathématique classique aux ensembles flous. Etant donné une fonction $\phi$ définie sur un univers classique $X($ par exemple, $X=\mathbb{R})$, l'idée de base d'un principe d'extension est de permettre l'utilisation de cette fonction pour des sous-ensembles flous de $X$. Connaissant un sous-ensemble flou $A$ de l'univers $X$ et une application $\phi: X \longrightarrow Y$, on veut pouvoir construire l'image de $A$ par $\phi .$ C'est donc un principe fondamental pour utiliser les fonctions mathématiques classiques pour des valeurs imprécises et qui autorise donc la prise en compte d'un tel type de données dans des mécanismes de calculs élaborés.\\
Le principe d'extension le plus connu a été introduit par \cite{Ref1} et il est très utilisé, en particulier en arithmétique floue . Pour plus de détails sur ce principe, voir aussi \cite{Ref2}.
\begin{definition}
Etant donné un sous-ensemble flou A de X et une application $\phi$ de X vers $Y$, le principe d'extension permet de définir un sous-ensemble flou $B$ de $Y$ associé à A par l'intermédiaire de $\phi$ :
\begin{equation}
\forall y \in Y, \mu_{B}(y)=\left\{\begin{array}{ll}
\sup _{\{x \in X \mid y=\phi(x)\}} \mu_{A}(x) & \text { si } \phi^{-1}(y) \neq \varnothing \\
0 & \text { si } \phi^{-1}(y)=\varnothing
\end{array}\right.
\end{equation}
\end{definition}
Le sous-ensemble flou $B$ est l'image du sous-ensemble flou $A$ par la fonction $\phi$
\section{Sous ensemble flou de $\mathbb{R}$}
On note $ E^{1} $ la classe de fonction définie comme suit:
$$ E^{1}=\{u: \mathbb{R} \rightarrow[0,1], \quad u \quad \texttt{ satisfait (1-4) ci-dessous } \} $$
1) $u$ est normal, c'est-à-dire qu'il y a un $x_{0} \in \mathbb{R}$ tel que $u\left(x_{0}\right)=$ 1;\\
2) $u$ est un ensemble convexe flou;\\
3) $u$ est semi-continue supérieure;\\
4) $[u]^{0}=c l\{x \in \mathbb{R} \mid u(x)>0\}$ est compact.

Alors $E^{1}=\mathbb{R}_{\mathcal{F}}$ est appelé l'espace des nombres flous. Évidemment, $\mathbb{R} \subset \mathbb{R}_{\mathcal{F}} .$ Pour $0<\alpha \leq 1$ dénotons $[u]^{\alpha}=\{x \in \mathbb{R} \mid u(x) \geq \alpha\}$, alors de (1) à (4) il s'ensuit que le $\alpha$-coupe définit $[u]^{\alpha} \in P_{K}(\mathbb{R})$ pour tout $0 \leq \alpha \leq 1$ est un intervalle fermé borné que l'on note $[u]^{\alpha}=\left[u_{1}^ {\alpha}, u_{2}^{\alpha}\right]$.

Où $P_{K}(\mathbb{R})$ désigne la famille de tous les sous-ensembles convexes compacts non vides de $\mathbb{R}$ et définit l'addition et la multiplication scalaire dans $P_{K}(\mathbb{R} )$ comme d'habitude. La propriété des nombres flous est que les $\alpha$-coupe $[u]^{\alpha}$ sont des ensembles fermés pour tout $\alpha \in[0,1]$.

Par l'extension principale de Zadeh nous avons:
$$
\begin{array}{c}
(u+v)^{\alpha}=u^{\alpha}+v^{\alpha} \\
(\lambda u)^{\alpha}=\lambda u^{\alpha}
\end{array}
$$
 Pour tous $u, v \in E^{1}$ et $\lambda \in \mathbb{R}$ La distance entre deux éléments de $E^{1}$ est donné par :
$$
\begin{aligned}
d(u, v) &=\sup _{\alpha \in(0,1]} d_{H}\left(u^{\alpha}, v^{\alpha}\right) \\
&=\sup _{\alpha \in(0,1]} \max \{|\bar{u}(\alpha)-\bar{v}(\alpha)|,|\underline{u}(\alpha)-\underline{v}(\alpha)|\}
\end{aligned}
$$
\begin{theorem}
Soit $u \in E^{1}$ et notons $C_{\alpha}=[u]^{\alpha}$ pour $\alpha \in[0,1]$. Alors :

1. $C_{\alpha}$ est un ensemble convexe compact non vide dans $\mathbb{R}$ pour chaque $\alpha \in[0,1],$

2. $C_{\beta} \subseteq C_{\alpha}$ pour $0<\alpha \leq \beta \leq 1,$

3. $C_{\alpha}=\bigcap_{i=1}^{\infty} C_{\alpha_{i}}$, pour toute suite non décroissante $\alpha_{i} \rightarrow \alpha$ sur $[0 ,1].$
\end{theorem}
\begin{proposition}
L'espace métrique $\left(E^{1}, d\right)$ est complet, séparable et localement compact et les propriétés suivantes pour la métrique $ d $ sont valides:\\
1)$  d(u+v, u+w)=d(u, v) ;$\\
2) $d(\lambda u, \lambda v)=|\lambda| d(u, v) ;$\\
3) $ d(u+v, w+z) \leq d(u, w)+d(v, z) $
\end{proposition}
\begin{remarque} 
L'espace $\left(E^{1}, d\right)$ est un espace normé linéaire avec $\|u\|=d(u, 0)$.
\end{remarque}
\begin{definition}
Soit $u, v \in E^{1} .$ nous mettons $u^{\alpha}=[\underline{u}(\alpha), \bar{u}(\alpha)]$
et $v^{\alpha}=[\underline{v}(\alpha), \bar{v}(\alpha)]$. On définit la multiplication de $ u $ et $ v $ par:
$$
\begin{aligned}
(u v)^{\alpha}=&[\min \{\underline{u}(\alpha) \underline{v}(\alpha), \underline{u}(\alpha) \bar{v}(\alpha), \bar{u}(\alpha) \underline{v}(\alpha), \bar{u}(\alpha) \bar{u}(\alpha)\}\\
&\max \{\underline{u}(\alpha) \underline{v}(\alpha), \underline{u}(\alpha) \bar{v}(\alpha), \bar{u}(\alpha) \underline{v}(\alpha), \bar{u}(\alpha) \bar{u}(\alpha)\}]
\end{aligned}
$$
\end{definition}
\begin{definition}
La différence généralisée de Hukuhara de deux nombres flous $ u, v \in E^{1} $ est définie comme suit:
$$
u-_{g} v=w \Leftrightarrow\left\{\begin{array}{ll}
u= & v+w \\
\text { or } & v=u+(-1) w
\end{array}\right.
$$
\end{definition}
En terme de $\alpha$-niveaux que nous avons :
$$
\left(u-_{g} v\right)^{\alpha}=[\min \{\underline{u}(\alpha)-\underline{v}(\alpha), \bar{u}(\alpha)-\bar{v}(\alpha)\},\max \{\underline{u}(\alpha)-\underline{v}(\alpha), \bar{u}(\alpha)-\bar{v}(\alpha)\}].$$
et les conditions d'existence de $w=u-_{g} v \in E^{1}$ sont \\
\begin{center}
cas (i) $\left\{\begin{array}{l}\underline{w}(\alpha)=\underline{u}(\alpha)-\underline{v}(\alpha) \text { and } \bar{w}(\alpha)=\bar{u}(\alpha)-\bar{v}(\alpha) \\ \text { with } \underline{w}(\alpha) \text { increasing, } \bar{w}(\alpha) \text { decreasing, } \underline{w}(\alpha) \leq \bar{w}(\alpha)\end{array}\right.$\\
cas(ii) $\left\{\begin{array}{l}\underline{w}(\alpha)=\bar{u}(\alpha)-\bar{v}(\alpha) \text { and } \bar{w}(\alpha)=\underline{u}(\alpha)-\underline{v}(\alpha) \\ \text { with } \underline{w}(\alpha) \text { increasing, } \bar{w}(\alpha) \text { decreasing, } \underline{w}(\alpha) \leq \bar{w}(\alpha)\end{array}\right.$\\
\end{center}
pour tout $\alpha \in[0,1]$.\\
Dans le reste de cet article, nous supposons que $u-_{g}$ $v \in E^{1}$
\begin{proposition}
$$
\left\|u-_{g} v\right\|=d(u, v)
$$
\end{proposition}
Puisque $\|.\|$ est une norme sur $E^{1}$ et par la proposition $ 3 $ nous avons :
\begin{proposition}
$$
\left\|\lambda u-_{g} \mu u\right\|=|\lambda-\mu|\|u\|
$$
\end{proposition}
\section{Semi-groupe flou }
Cette section est consacrée à la notion de semigroupe flou et aux propriétés liées à ce concept.
\subsection{Domain de semi-groupe flou }
Dans cette sous-section, nous présenterons le domaine du semigroupe flou d'opérateur en utilisant la différence généralisée de Hukuhara.\\
 Nous avons commencé par le théorème suivant.
\begin{theorem}
 Il existe un véritable espace de Banach $ X $ tel que $ E ^ {1} $ puisse être incorporé comme un cône convexe $ C $ avec le sommet $0$ dans $ X $. En outre, les conditions suivantes sont vérifiées:\\
(i) l'incorporation $ j $ est isométrique,\\
(ii) l'addition en $ X $ induit l'addition en $E^{1}$,\\
(iii) la multiplication par un nombre réel non négatif en $ X $ induit l'opération correspondante en $E^{1}$,\\
(iv) $C-C=\left\{a-b / a, b \in E^{1}\right\}$ est dense dans $X$,\\
(v) $C$ est fermé.
\end{theorem}
 Nous donnons également la définition suivante de la carte linéaire floue:
\begin{definition}
Une application $p: E^{1} \rightarrow E^{1}$ est dit une application linéaire si:\\
1) $p(x+y)=p(x)+p(y), \quad \forall x, y \in E^{1}$;\\
2) $p(\lambda x)=\lambda p(x), \quad \forall \lambda \in \mathbb{R}$.
\end{definition}
\begin{definition}
 Une famille $\{T(t), t \geq 0\}$ d'opérateurs linéaires flous de $E^{1}$ en lui-même est un semi-groupe flou fortement continu si\\
(i) $T(0)=i$, l´application identité sur $E^{1}$,\\
(ii) $T(t+s)=T(t) T(s)$ pour tout $t, s \geq 0$,\\
(iii) La fonction $g:\left[0, \infty\left[\rightarrow E^{1}\right.\right.$, definé par $g(t)=T(t) x$ est continue en $t=0$ pour tout $x \in E^{1}$ c'est à dire
$$
\lim _{t \rightarrow 0^{+}} T(t) x=x.
$$
\end{definition}
Nous nous inspirons du cas «net», nous obtenons la définition suivante.
\begin{definition}
Soit $\{T(t), t \geq 0\}$ un semi-groupe flou fortement continu sur $E^{1}$ et $x \in E^{1}$. nous mettons
$$
A x=\lim _{h \rightarrow 0^{+}} \frac{T(h) x-_{g} x}{h}
$$
chaque fois que cette limite existe dans l'espace métrique $\left(E^{1}, d\right)$. Alors l'opérateur $ A $ défini sur :
$$
D(A)=\left\{x \in E^{1}: \lim _{h \rightarrow 0^{+}} \frac{T(h) x-_{g} x}{h} \text { exists }\right\} \subset E^{1}
$$
est appelé le générateur infinitésimal du semi-groupe flou $\{T(t), t \geq 0\} .$
\end{definition}
\begin{theorem}
 Soit $T(t)$ un semi-groupe flou fortement continu. Il existe une constante $\omega \geq 0$ et $M \geq 1$ telle que
$$
d(T(t) x, \tilde{0}) \leq M e^{\omega t}, \quad \forall x \in E^{1}.
$$
\end{theorem}
\begin{proof}
 Soit $x \in E^{1}$. On a :
 $$j T(t) x j^{-1}=T_{1}(t) j x$$
  puisque : 
  $$\left\|T_{1}(t)\right\| \leq M e^{\omega t},$$
   alors :
   $$d(T(t) x, \tilde{0}) \leq M e^{\omega t}$$
 \end{proof}
 \begin{corollaire}
Si $T(t)$ est un semi-groupe flou d'opérateurs alors pour chaque $x \in E^{1}, t \rightarrow T(t) x$ est une fonction continue de $\mathbb{R}^{+}$ vers $E^{1}$.
\end{corollaire}
\begin{proof}
Soit $t, h \geq 0$, on a
$$
\begin{aligned}
d(T(t+h) x, T(t) x) & \leq d(T(t) T(h) x, T(t) x) \\
& \leq d(T(t) x, \widetilde{0}) d(T(h) x, x) \\
& \leq M e^{\omega t} d(T(h) x, x)
\end{aligned}
$$
Mais $M e^{\omega t} d(T(h) x, x) \rightarrow 0$, comme $h \rightarrow 0$.
\end{proof}
Parmi les conséquences de la différentiabilité générale, le lemme suivant.
\begin{lemme}
 Soit $A$  le générateur d'un semi-groupe flou $\{T(t), t \geq 0\}$ sur $E^{1}$, alors pour tout $x \in E^{1}$ tel que $T(t) x \in$ $D(A)$ pour tout $t \geq 0$, l´application $t \rightarrow T(t) x$ est differentiable et
$$
{ }_{g H} \frac{d}{d t}(T(t) x)=A T(t) x, \quad \forall t \geq 0
$$
\end{lemme}

\begin{proof}
 Par définition de A, nous avons :
 $$ A x=\lim _{t \rightarrow 0} \frac{T(t) x-_{g}^{x}}{t}$$
1) Si $t \rightarrow T(t)$ est $i$-diff alors
$$
\frac{d}{d t} T(t)=A.
$$
2) Si $t \rightarrow T(t)$ est $i i$-diff alors
$$
\frac{d}{d t} T(t)=(-1) A.
$$
Donc
$$
g H \frac{d}{d t} T(t)=A
$$
\end{proof}
\begin{lemme} Soit $A: E^{1} \rightarrow E^{1}$ et $A_{1}=j A j^{-1}$ :
$C \rightarrow C$ deux opérateurs (non linéaires). $ A $ est le générateur infinitésimal d'un semi-groupe flou $\{T(t), t \geq 0\}$ sur $E^{1}$ si et seulement si $A_{1}$ est le générateur infinitésimal du semi-groupe  $\left\{T_{1}(t), t \geq 0\right\}$ défini sur l'ensemble fermé convexe $ C $ par :
$$
T_{1}(t)=j \bar{T}(t) j^{-1} \text { for } t \geq 0.
$$
\end{lemme}
\begin{exemple}
Nous définissons sur $E^{1}$ la famille d'opérateur $\{T(t), t \geq 0\}$ par :
$$
T(t) x=e^{k t} x, \quad k \in \mathbb{R} .
$$
Pour $k \geq 0,\{T(t), t \geq 0\}$ est un semi-groupe flou fortement continu sur $E^{1}$, et l'opérateur linéaire $ A $ défini par $ A x = k x $ est le générateur infinitésimal de ce demi-groupe flou.
\end{exemple}
\chapter{ Arithmétique des nombres flous}
Dans cette chapitre, nous analysons la possibilité de généraliser la différence de Hukuhara pour les ensembles compacts et convexes, y compris le cas particulier des intervalles unidimensionnels et multidimensionnels. Nous suggèrons une définition de l'opération de différence $ \Theta_{g}, $ appelée difference-$ g H $ , qui étend la différence bien connue de Hukuhara, et nous caractérisons son existence (pour les ensembles convexes compacts) en termes de fonctions de support. La différence de gH proposée existe toujours dans le cas de l'intervalle unidimensionnel et une condition simple nécessaire et suffisante pour son existence dans le cas des intervalles multidimensionnels est donnée. La différence gH est également suggérée pour les nombres flous et nous donnons des conditions simples pour son existence et les régles pour la déterminer.

Nous adoptons une approche similaire pour définir une généralisation de la division $ \div_{g}, $ appelée $ g$-division, des intervalles unidimensionnels (pour lesquels il existe toujours) et pour les nombres flous et nous montrons les propriétés et les conditions d'existence.

Les travaux en cours utilisent l'opérateur de différence gH proposé pour définir la différentiabilité généralisée pour les fonctions d'intervalle et de valeurs floues et pour la théorie et les algorithmes numériques dans les équations différentielles d'intervalle et floues.

Dans cette chapitre on étudie l´arithmétique des nombres flous \cite{Ref11}.
\section{Définitions et propriétés}
Considérons un espace vectoriel métrique $X$ avec la topologie induite et en particulier l'espace $X = \mathbb{R}^{n},$ $ n\geq 1$, de vecteurs réels équipés d'opérations standards d'addition et de multiplication scalaire.\\
notons $K(X)$ et $K_{C}(X)$ les espaces des ensembles convexes compacts et compacts non vides de X.
\begin{definition}
étant donné deux sous-ensembles $A, B \subseteq X $ et $ k \in R$, addition de Minkowski et les multiplications scalaires sont définies par :
$$ A + B = \{a + b | a \in A, b \in B \} $$
Et 
$$ k A = \{ka | a \in A \}. $$
\end{definition}
Il est bien connu que l'addition est associative et commutative et à élément neutre $ \{0 \}.$\\
Si $k = -1$, la multiplication scalaire donne l'opposé :
\begin{tcolorbox}{}
$$-A = (-1) A = \{-a | a \in A \}$$
\end{tcolorbox}
Mais, en général, $A + (-A) \neq {0}$, i.e l'opposé de A est pas l'inverse de A dans l'addition de Minkowski (sauf si $A = \{a\}$ est un singleton).
\begin{definition}
La différence de Minkowski est :
$$ A - B = A + (-1) B = \{a - b | a \in A, b \in B\}.$$
\end{definition}
Une première implication de ce fait est que, en général, même s'il est vrai que :
\begin{tcolorbox}
$$ (A + C = B + C) \Leftrightarrow A = B,$$ 
\end{tcolorbox}
La simplification d'addition / soustraction n'est pas valide, c'est-à-dire :
\begin{tcolorbox}
$$ (A + B) - B \neq A.$$
\end{tcolorbox}
Pour surmonter partiellement cette situation, Hukuhara a ajouté la différence H suivante:
\begin{tcolorbox}
\begin{equation}
A \Theta B = C \Leftrightarrow A = B + C
\label{6}
\end{equation}
\end{tcolorbox}
Et une propriété importante de $ \color{purple}{\Theta} $ est que :
\begin{tcolorbox}
$$ A\Theta A = {0}, \forall A \in K (X) $$
\end{tcolorbox}
Et
\begin{tcolorbox}
$$ (A + B)\Theta B = A, \forall A, B \in K (X).$$
\end{tcolorbox}

La différence H est unique, mais une condition nécessaire pour que A $\color{purple}{\Theta}$ B existe est que A contienne une translation $\{c\} + B$ de B.\\
 En général,
 \begin{tcolorbox}
  $$A-B\neq A\Theta B.$$
  \end{tcolorbox}
  D'un point de vue algébrique, la différence de deux ensembles A et B peut être interprétée à la fois en termes d'addition comme dans (\ref{6}) ou en termes d'addition négative, i.e.
  \begin{tcolorbox}
  \begin{equation}
 A\Theta B = C \Leftrightarrow B = A + (-1)C, 
 \label{7}
  \end{equation}
  \end{tcolorbox}
  Où $(-1)C$  est l'ensemble opposé de C.\\
  
  Les conditions (\ref{6}) et (\ref{7}) sont compatibles entre elles et cela suggère une généralisation de la différence de Hukuhara:
 \begin{definition}
 Soit $A, B \in K(X)$; nous définissons la différence Hukuhara généralisée de A et B comme $C \in K(X)$ telle que :
 \begin{equation}
A \Theta_{g} B=C \Longleftrightarrow\left\{\begin{array}{ll}
\text{(i) } A=B+C \\
\text{or} \\
\text{(ii) } B=A+(-1) C
\end{array}\right.
\label{8}
\end{equation}
 \end{definition}
 \begin{proposition}(Unicité de $ A \Theta_{g} B $)\\
  Si $C = A\Theta_{g}B$ existe, il est unique et si aussi $ A\Theta B $ existe alors $ A \Theta_{g} B =A\Theta B. $
 \end{proposition}
 \begin{proof}
Si $C=A \Theta_{g} B$ existe dans le cas (i), on obtient $C=A \Theta B$ ce qui est unique. Supposons que le cas (ii) soit satisfait pour $C$ et $D$, c'est à dire. $B=A+(-1) C$ et $B=A+(-1) D$; alors $A+(-1) C=A+(-1) D \Longrightarrow(-1) C=(-1) D=$ $\Rightarrow C=D$. Si le cas (i) est satisfait pour $ C $ et le cas (ii) est satisfait pour $ D $, c'est-à-dire $A=B+C$ et $B=A+(-1) D$, alors $B=B+C+(-1) D \Longrightarrow\{0\}=C-D$ et cela n'est possible que si $C=D=\{c\}$ est un singleton.
 \end{proof}
La différence généralisée de Hukuhara $A \Theta_{g} B $ sera appelée la différence gH de A et B.
\begin{remarque}
Une condition  necéssaire pour $ A \Theta_{g} B $ exister est que soit A contient une translation de B (comme pour $ A \Theta B $) ou B contient une translation de A. En fait, pour tout $c \in C$, on a $ B + {c} \subseteq A $ de (i) or $ A + {-c} \subseteq B $ de (ii).
\end{remarque}
\begin{remarque}
Il est possible que $ A = B + C $ et $ B = A + (-1)C $ tenir simultanément; dans ce cas, A et B se traduisent l'un dans l'autre et C est un singleton. En fait, $ A = B + C $ implique $ B + {c} \subseteq A \forall c \in C $ et $ B = A + (-1)C $ implique
$ A - {c} \subseteq B \forall c \in C $ i.e. $ A \subseteq B + {c} $; il s'ensuit que $ A = B + {c} $ et $ B = A + {-c}.$ En revanche, si $ c^{\prime}, c^{\prime\prime} \in C $ alors $ A = B + {c^{\prime}} = B + {c^{\prime\prime}}$ et cela nécessite $c^{\prime} = c^{\prime\prime}.$
\end{remarque}
 \begin{remarque}
 Si $ A \Theta_{g} B $ existe, alors $ B \Theta_{g} A $ existe et $ B \Theta_{g} A =- (A \Theta_{g} B ) . $
 \end{remarque}
 
 \begin{proposition}
La différence gH $ \Theta_{g} $ a les propriétés suivantes:
\begin{enumerate}
\item[1)] $A\Theta_{g} A = \{0\}.$
\item[2)] $\left\{  \begin{array}{lll}
 (a) (A + B)\Theta_{g} B = A ,\\
 (b) A\Theta_{g}(A - B) = B,\\
 (c) A\Theta_{g}(A + B) = -B.
\end{array}\right.$
\item[3)] $A\Theta_{g} B $ existe si et seulement si $ B \Theta_{g} A $ et $ (-B) \Theta_{g} (-A) $ existe et $A\Theta_{g} B  = (-B)\Theta_{g} (-A) = -(B\Theta_{g} A).$
\item[4)] En générale, $ B \Theta_{g} A  = A\Theta_{g} B$ n´implique pas $ A = B $; mais $ A\Theta_{g} B = B\Theta_{g} A = C $ si et seulement si $ C = -C $ et, en particulier, $ C = \{0\} $ si et seulement si $ A = B.$
\item[5)] Si $B\Theta_{g} A $ existe alors chaque $ A + (B\Theta_{g} A) = B $ ou $ B - (B\Theta_{g} A) = A $ et les deux égalités sont valables si et seulement si $ B\Theta_{g} A $ est un singleton.
\item[6)] Si $ B\Theta_{g} A = C  $ existe, alors pour tout $ D \in K(X)$ either  $ (B + D)\Theta_{g} A = C + D $ ou $ B\Theta_{g}(A + D) = C - D.$
\end{enumerate}
\end{proposition}

\begin{proof}
La propriété 1 est immédiate. Prouver $2(\mathrm{a})$ si $C=(A+B) \Theta_{g} B$ alors soit $A+B=C+B$ ou  $B=(A+B)+(-1) C=$ $B+(A+(-1) C) ;$ dans le premier cas, il s'ensuit que $C=A$, dans le second cas $A+(-1) C=\{0\}$ et $A$ et $C$ sont des singleton donc $A=C .$ Avec un argument similaire, $2(\mathrm{~b})$ et $2(\mathrm{c})$ peut être prouvé. Pour prouver la première partie de $(3)$ soit $C=A \Theta_{g} B$ selon le cas (i), c'est-à-dire $A=B+C$, alors $A=B-(-C)$ et $B \odot_{g} A=-C$ selon le cas (ii); si à la place $C=A \Theta_{g} B$ selon le cas (ii), c'est-à-dire $B=A-C$, alors $B=A+(-C)$ et $B \Theta_{g} A=-C$ selon le cas (i); d'autre part, si $A=B+C$ ou $B=A-C$, alors $-A=-B+(-C)$ ou $-B=-A+C$ et cela signifie $(-B) \Theta_{g}(-A)=C .$ Pour voir la première partie de (4) considérons par exemple le cas unidimensionnel $A=\left[a^{-}, a^{+}\right], B=\left[b^{-}, b^{+}\right] ;$ égalité $A-B=B-A$ est valide si $a^{-}+a^{+}=b^{-}+b^{+}$ et cela ne nécessite pas $A=B$ (à moins que $A$ et $B$ sont des singletons). Pour la deuxième partie de (4), à partir de $\left(A \Theta_{g} B\right)=\left(B \Theta_{g} A\right)=C$,considérant les quatre combinaisons dérivées de (3), l'un des quatre cas suivants est valide: $(A=B+C$ et $B=A+C)$ ou $(A=B+C$ et $A=B-C)$ ou $(B=A+(-1) C$ et $B=A+C)$ ou $(B=A+(-1) C$ et $A=B+(-1) C) ;$ dans chacun d'eux nous déduisons $C=-C$ et $C=\{0\}$ si et seulement si $ A = B $. Pour voir (5), considérez que si $\left(B \odot_{g} A\right)$ existe au sens de (i) la première égalité est valide et si elle existe au sens de (ii) le second est valide. Pour prouver (6), si $B=A+C$ alors $B+D=A+C+D$ et $(B+D) \Theta_{g} A=C+D$ selon le cas (i); si $ A = B-C $ alors $ A + D = B-C + D = B- (C-D) $ et $B \odot_{g}(A+D)=C-D$ selon le cas (ii).
\end{proof}

\begin{remarque}
 L'équivalence $(A\Theta B) = C \Leftrightarrow (A\Theta C) = B $ est valable uniquement pour $ \Theta $, ou pour  $ \Theta_{g} $ dans le cas (i).
 En fait, si $ A\Theta_{g} B = C $ dans le sens (ii), alors $ B = A-C $   et cela n'implique pas $ A = B+C $ ni $ C = A - B $ à moins que $ B = B+C -C $ (i.e. $ C = \{\widehat{c}\} $ ) ou $ B = B + A - A$  (i.e. $ A =\{\widehat{a}\} $ ). Notez également que, en général, $A+\left(B \Theta_{g} A\right) \neq A$ et $A-\left(A \Theta_{g} B\right) \neq B$
\end{remarque}
Pour les ensembles $   A, B \in K(X)$ sur un espace normé $(X,\parallel .\parallel)$ la distance de Hausdorff est définie comme d'habitude par:
$$ H(A, B) = \max \{d_{H} (A, B), d_{H} (B, A) \}, $$
où
$$ d_{H} (A, B) = \sup_{a\in A}\inf_{b\in B}\parallel a - b\parallel $$ 
et
 $$ d_{H} (B, A) = \sup_{b\in B}\inf_{a\in A}\parallel a - b\parallel $$
 
nous dénotons $\parallel A\parallel_{H} = H(A, \{0\}) = \sup_{a\in A}\parallel a\parallel .$\\
Si $ X = \mathbb{R}^{n} $, $ n \geq 1 $ est l'espace vectoriel réel à n dimensions avec produit interne $ <x, y>$ et norme correspondante $\parallel  x \parallel =\sqrt{<x, x>} $, nous désignons par $ K^{n} $ et $ K^{n}_{c} $ les espaces d'ensembles convexes compacts et compacts (non vides) de $ \mathbb{R}^{n} $, respectivement.\\
Si $ A \subseteq \mathbb{R}^{n} $ et $ S^{n-1} = \{p |p \in \mathbb{R}^{n},\parallel p\parallel = 1\}$ est la sphère unitaire, la fonction de support associée à A est
$$
\begin{aligned}
s_{A}    &: \mathbb{R}^{n} \longrightarrow \mathbb{R} \\
& p \longrightarrow s_{A}(p) =\sup \{\langle p, a\rangle \mid a \in A\}.
\end{aligned}
$$
Si $ A \neq \varnothing  $ est compact, alors $ s_{A}(p) \in \mathbb{R},\forall p \in S^{n-1}$. Les propriétés suivantes sont bien connues:
\begin{enumerate}
\item[1)] Toute fonction $s: \mathbb{R}^{n} \longrightarrow \mathbb{R}$ qui est continue (ou, plus généralement, semi-continue supérieure), positivement homogène $s(t p)=t s(p), \forall t \geq 0, \forall p \in \mathbb{R}^{n}$ et sous-additif $s\left(p^{\prime}+p^{\prime \prime}\right) \leq s\left(p^{\prime}\right)+s\left(p^{\prime \prime}\right), \forall p^{\prime}, p^{\prime \prime} \in \mathbb{R}^{n}$ est une fonction de support d'un ensemble convexe compact; la restriction $\widehat{s}$ de $s$ à $\mathcal{S}^{n-1}$ est telle que $\widehat{s}(p /\|p\|)=(1 /\|p\|) s(p), \forall p \in \mathbb{R}^{n}, p \neq 0$ et nous pouvons considérer $s$ limité à $\mathcal{S}^{n-1}$. Il s'ensuit également que $s: \mathcal{S}^{n-1} \longrightarrow \mathbb{R}$ est une fonction convexe.
\item[2)] Si $ A \in K^{n}_{C} $ est un ensemble convexe compact, alors il se caractérise par sa fonction de support et
$$
A=\left\{x \in \mathbb{R}^{n} \mid\langle p, x\rangle \leq s_{A}(p), \forall p \in \mathbb{R}^{n}\right\}=\left\{x \in \mathbb{R}^{n} \mid\langle p, x\rangle \leq s_{A}(p), \forall p \in \mathcal{S}^{n-1}\right\}
$$
\item[3)] Pour $A, B \in \mathcal{K}_{C}^{n}$ et $\forall p \in \mathcal{S}^{n-1}$ on a  $s_{\{0\}}(p)=0$ et
$$
\begin{aligned}
A & \subseteq B \Longrightarrow s_{A}(p) \leq s_{B}(p) ; \quad A=B \Longleftrightarrow s_{A}=s_{B} \\
s_{k A}(p) &=k s_{A}(p), \forall k \geq 0 ; \quad s_{k A+h B}(p)=s_{k A}(p)+s_{h B}(p), \quad \forall k, h \geq 0
\end{aligned}
$$
et en particulier
$$ s_{A+B}(p) = s_{A}(p) + s_{B}(p)$$
\item[4)] Si $s_{A}$ est la fonction de support de $A \in \mathcal{K}_{C}^{n}$ et $s_{-A}$ est la fonction de support de $-A \in \mathcal{K}_{C}^{n}$, alors $\forall p \in \mathcal{S}^{n-1}, s_{-A}(p)=$ $s_{A}(-p),$
\item[5)] Si $v$ est une mesure sur $\mathbb{R}^{n}$ tel que $v\left(\mathcal{S}^{n-1}\right)=\int_{\mathcal{S}^{n-1}} v(d p)=1,$ une distance est définie par
$$
\rho_{2}(A, B)=\left\|s_{A}-s_{B}\right\|=\left(n \int_{\mathcal{S}^{n-1}}\left[s_{A}(p)-s_{B}(p)\right]^{2} v(d p)\right)^{1 / 2}
$$
La distance $\rho_{2}(\cdot, \cdot)$ induit la norme sur $\mathcal{K}_{C}^{n}$ Défini par $\|A\|=\rho_{2}(A,\{0\})$.
\item[6)] Le point Steiner de $A \in \mathcal{K}_{C}^{n}$ est défini par $\sigma_{A}=n \int_{\mathcal{S}^{n-1}} p s_{A}(p) v(d p)$ et $\sigma_{A} \in A$.
\end{enumerate}

Nous pouvons exprimer la différence Hukuhara généralisée des ensembles convexes compacts $A, B \in \mathcal{K}_{C}^{n}$ par l'utilisation des fonctions de support. Considérer $A, B, C \in \mathcal{K}_{C}^{n}$ avec $C=A \Theta_{g} B$ tel que défini dans (\ref{8}); soit $s_{A}, s_{B}, s_{C}$ et $s_{(-1) C}$ être les fonctions de support de $A, B, C,$ et $(-1) C,$ respectivement. Dans le cas (i) nous avons $s_{A}=s_{B}+s_{C}$ et dans le cas (ii) nous avons $s_{B}=s_{A}+s_{(-1) C} .$ Alors, $\forall p \in \mathcal{S}^{n-1}$
$$
s_{C}(p)=\left\langle\begin{array}{ll}
s_{A}(p)-s_{B}(p) & \text { in case (i), } \\
s_{B}(-p)-s_{A}(-p) & \text { in case (ii) }
\end{array}\right.
$$
i.e.
\begin{equation}
s_{C}(p)=\left\langle\begin{array}{ll}
s_{A}(p)-s_{B}(p) & \text { in case (i), } \\
s_{(-1) B}(p)-s_{(-1) A}(p) & \text { in case (ii). }
\end{array}\right.
\label{9}
\end{equation}
Maintenant, $s_{C}$ dans (\ref{9}) est une fonction de support correcte si elle est continue (semi-continue supérieure), positivement homogène et sous-additive et cela nécessite que, dans les cas correspondants (i) et (ii), $s_{A}-s_{B}$ et/ou $s_{-B}-s_{-A}$ être des fonctions de support, en supposant que $s_{A}$ et $s_{B}$ sont des fonctions supports.

Considérons $s_{1}=s_{A}-s_{B}$ et $s_{2}=s_{B}-s_{A}$. Continuité de $s_{1}$ et $s_{2}$ est évidente. Pour voir leur homogénéité positive soit $t \geq 0 ;$ on a $s_{1}(t p)=s_{A}(t p)-s_{B}(t p)=t s_{A}(p)-t s_{B}(p)=t s_{1}(p)$ et de même pour $s_{2} .$ Mais $s_{1}$ et/ou $s_{2}$ peut ne pas être sous-additif et les quatre cas suivants, liés à la définition de la différence de gH, sont possibles.

\begin{proposition}
Soit $s_{A}$ et $s_{B}$ être les fonctions de support de $ A, B \in K^{n}_{C}$ et considérons $ s_{1} = s_{A} - s_{B}, s_{2} = s_{B} - s_{A}$; les quatre cas suivants s'appliquent:
\begin{enumerate}
\item[1)] 1. Si $ s_{1} $ et $ s_{2} $ sont tous les deux sous-additifs, alors $ A \Theta_{g} B $ existe; (i) et (ii) sont satisfaits simultanément et $ A \Theta_{g} B = \{c\}.$ 
\item[2)] Si $ s_{1} $ est sous-additif et $ s_{2} $ n'est pas, alors $ C =  A \Theta_{g} B $ existe, (i) est satisfait et $ s_{C} = s_{A} - s_{B} $.
\item[3)] Si $ s_{1} $ n'est pas sous-additif et $ s_{2} $ est, alors $ C =  A \Theta_{g} B $ existe, (ii)  est satisfait et $ s_{C} = s_{-B} - s_{-A} $.
\item[4)] Si $ s_{1} $ et $ s_{2} $ ne sont pas tous les deux sous-additifs, alors $ A \Theta_{g} B $ n'existe pas.
\end{enumerate}
\end{proposition}
\begin{proof}
Dans le cas 1 sous-additivité de $s_{1}$ et $s_{2}$ signifie que, $\forall p^{\prime}, p^{\prime \prime} \in \mathcal{S}^{n-1}$
$$
\begin{array}{l}
s_{1}: s_{A}\left(p^{\prime}+p^{\prime \prime}\right)-s_{B}\left(p^{\prime}+p^{\prime \prime}\right) \leq s_{A}\left(p^{\prime}\right)+s_{A}\left(p^{\prime \prime}\right)-s_{B}\left(p^{\prime}\right)-s_{B}\left(p^{\prime \prime}\right) \\
s_{2}: s_{B}\left(p^{\prime}+p^{\prime \prime}\right)-s_{A}\left(p^{\prime}+p^{\prime \prime}\right) \leq s_{B}\left(p^{\prime}\right)+s_{B}\left(p^{\prime \prime}\right)-s_{A}\left(p^{\prime}\right)-s_{A}\left(p^{\prime \prime}\right)
\end{array}
$$
il s'ensuit que
$$
s_{A}\left(p^{\prime}+p^{\prime \prime}\right)-s_{A}\left(p^{\prime}\right)-s_{A}\left(p^{\prime \prime}\right) \leq s_{B}\left(p^{\prime}+p^{\prime \prime}\right)-s_{B}\left(p^{\prime}\right)-s_{B}\left(p^{\prime \prime}\right)
$$
$$
s_{B}\left(p^{\prime}+p^{\prime \prime}\right)-s_{B}\left(p^{\prime}\right)-s_{B}\left(p^{\prime \prime}\right) \leq s_{A}\left(p^{\prime}+p^{\prime \prime}\right)-s_{A}\left(p^{\prime}\right)-s_{A}\left(p^{\prime \prime}\right)
$$
de sorte que l'égalité tient :
$$
s_{B}\left(p^{\prime}+p^{\prime \prime}\right)-s_{A}\left(p^{\prime}+p^{\prime \prime}\right)=s_{B}\left(p^{\prime}\right)+s_{B}\left(p^{\prime \prime}\right)-s_{A}\left(p^{\prime}\right)-s_{A}\left(p^{\prime \prime}\right) .
$$
Prons $p^{\prime}=-p^{\prime \prime}=p$ on a, $\forall p \in \mathcal{S}^{n-1}, s_{B}(p)+s_{B}(-p)=s_{A}(p)+s_{A}(-p)$ c'est-à-dire $_{B}(p)+s_{-B}(p)=s_{A}(p)+s_{-A}(p)$
i.e. $s_{B-B}(p)=s_{A-A}(p)$ et $B-B=A-A\left(A\right.$ Et $B$ se traduire les uns dans les autres); il s'ensuit que $\exists c \in \mathbb{R}^{n}$ tel que $A=B+\{c\}$ et $B=A+\{-c\}$ de sorte que $A \Theta_{g} B=\{c\}$.

Dans le cas 2, nous avons cet être $s_{1}$ une fonction de support il caractérise un ensemble non vide $C \in \mathcal{K}_{C}^{n}$ et $s_{C}(p)=s_{1}(p)=$ $s_{A}(p)-s_{B}(p), \forall p \in \mathcal{S}^{n-1} ;$ alors $s_{A}=s_{B}+s_{C}=s_{B+C}$ et $A=B+C$ à partir de laquelle (i) est satisfait.

Dans le cas 3, nous avons cela $s_{2}$ la fonction de support d'un ensemble non vide $D \in \mathcal{K}_{C}^{n}$ et $s_{D}(p)=s_{B}(p)-s_{A}(p), \forall p \in \mathcal{S}^{n-1}$ de sorte que $s_{B}=s_{A}+s_{D}=s_{A+D}$ et $B=A+D .$ Définissons $C=(-1) D$ (ou  $\left.D=(-1) C\right)$ on obtient $C \in \mathcal{K}_{C}^{n}$ avec $s_{C}(p)=s_{-D}(p)=s_{D}(-p)=s_{B}(-p)-s_{A}(-p)=s_{-B}(p)-s_{-A}(p)$ et (ii) est satisfaite.

Dans le cas 4, il n'y a pas $C \in \mathcal{K}_{C}^{n}$ tel que $A=B+C$ (autrement $s_{1}=s_{A}-s_{B}$ est une fonction de support) et il n'y a pas $D \in \mathcal{K}_{C}^{n}$ tel que $B=A+D$ (autrement $s_{2}=s_{B}-s_{A}$ est une fonction de support); il s'ensuit que (i) et (ii) ne peuvent être satisfaits et que $ A \Theta_{g} B $ n'existe pas.
\end{proof}
\begin{proposition}
Si $C=A \Theta_{g} B$ existe, alors $\left\|A \Theta_{g} B\right\|=\rho_{2}(A, B) ;$ il s'ensuit que $$ \left\|A \Theta_{g} B\right\|=0 \Longleftrightarrow A=B$$.
\end{proposition}
\begin{proof}
En fait $\rho_{2}(A, B)=\left\|s_{A}-s_{B}\right\|$ et, si $A \Theta_{g} B$ existe, alors soit $s_{C}=s_{A}-s_{B}$ ou $s_{C}=s_{-B}-s_{-A} ;$ mais $\left\|s_{A}-s_{B}\right\|=\left\|s_{-A}-s_{-B}\right\|$ as, modification de la variable $p$ dans $-q$ et rappelant que $s_{-A}(p)=s_{A}(-p)$, on a :
$$
\begin{aligned}
\left\|s_{-A}-s_{-B}\right\| &=\int_{\mathcal{S}^{n-1}}\left[s_{-A}(p)-s_{-B}(p)\right]^{2} v(d p)=\int_{\mathcal{S}^{n-1}}\left[s_{A}(-p)-s_{B}(-p)\right]^{2} v(d p) \\
&=\int_{\mathcal{S}^{n-1}}\left[s_{A}(q)-s_{B}(q)\right]^{2} v(-d q)=\left\|s_{A}-s_{B}\right\|
\end{aligned}
$$
la dernière propriété découle du fait que $\left\|A \Theta_{g} B\right\|=0$ implique $\rho_{2}(A, B)=0$ de sorte que $ A = B; $ d'autre part, pour $A=B, A \odot_{g} A=\{0\}$
\end{proof}
Une propriété intéressante relie le point de Steiner de $ A \Theta_{g} B $ aux points de Steiner de A et B.
\begin{proposition}
 Si $C=A \Theta_{g} B$ existe, soit $\sigma_{A}, \sigma_{B}$ et $\sigma_{C}$ être les points Steiner de $A, B$ et $C,$ respectivement; alors $\sigma_{C}=\sigma_{A}-\sigma_{B}$
 \end{proposition}
 \begin{proof}
 Pour les points Steiner, nous avons
$$
\sigma_{A}=n \int_{\mathcal{S}^{n-1}} p s_{A}(p) v(d p), \quad \sigma_{B}=n \int_{\mathcal{S}^{n-1}} p s_{B}(p) v(d p)=-n \int_{\mathcal{S}^{n-1}} q s_{B}(q) v(-d q)
$$
and
$$
\sigma_{C}=\left\{\begin{array}{l}
n \int_{\mathcal{S}^{n-1}} p\left[s_{A}(p)-s_{B}(p)\right] v(d p) \quad \text { or } \\
n \int_{\mathcal{S}^{n-1}} q\left[s_{A}(q)-s_{B}(q)\right] v(-d q)
\end{array}\right.
$$
le résultat découle de l'additivité de l'intégrale.
 \end{proof}

\section{Le cas des intervalles compacts dans $ \mathbb{R}^{n} $}

Dans cette section, nous considérons la différence gH des intervalles compacts dans  $ \mathbb{R}^{n}$. Si $ n = 1 $, c'est-à-dire pour les intervalles compacts unidimensionnels, la différence gH existe toujours. En fait, soit $ A = [a^{-}, a^{+}]$ et $ B = [b^{-}, b^{+}]$ être deux intervalles; la différence gH est
$$
\left[a^{-}, a^{+}\right] \Theta_{g}\left[b^{-}, b^{+}\right]=\left[c^{-}, c^{+}\right] \Longleftrightarrow\left\{\begin{array}{l}
\text { (i) }\left\{\begin{array}{l}
a^{-}=b^{-}+c^{-}, \\
a^{+}=b^{+}+c^{+}
\end{array}\right. \\
\text {or (ii) }\left\{\begin{array}{l}
b^{-}=a^{-}-c^{+} \\
b^{+}=a^{+}-c^{-}
\end{array}\right.
\end{array}\right.
$$
de sorte que $\left[a^{-}, a^{+}\right] \Theta_{g}\left[b^{-}, b^{+}\right]=\left[c^{-}, c^{+}\right]$ est toujours défini par :
  $$c^{-}=\min \left\{a^{-}-b^{-}, a^{+}-b^{+}\right\}, \quad c^{+}=\max \left\{a^{-}-b^{-}, a^{+}-b^{+}\right\}$$
i.e.
$$
[a, b] \Theta_{g}[c, d]=[\min \{a-c, b-d\}, \max \{a-c, b-d\}]
$$

Les conditions (i) et (ii) sont satisfaites simultanément si et seulement si les deux intervalles ont la même longueur et $c^{-} = c^{+}$. En outre, le résultat est $\{0\}$ si et seulement si $a^{-} = b^{-}$ et $a^{+} = b^{+}.$

Deux exemples simples sur des intervalles compacts réels illustrent la généralisation :
                $$[-1,1] \Theta[-1,0]= [0,1]$$
 comme en fait (i) est : 
 $$[-1,0]+[0,1]=[-1,1]$$ 
 mais 
 $$[0,0] \Theta_{g}[0,1]=[-1,0]$$ 
 et $[0,1] \Theta_{g}\left[-\frac{1}{2}, 1\right]=\left[0, \frac{1}{2}\right]$ satisfaire (ii).
 
Les intervalles symétriques sont intéressants $A=[-a, a]$ et $B=[-b, b]$ avec $a, b \geq 0 ;$ il est bien connu que les opérations de Minkowski à intervalles symétriques sont telles que $A-B=B-A=A+B$ et, en particulier, $A-A=A+A=2 A$. On a $$[-a, a] \Theta_{g}[-b, b]=[-|a-b|,|a-b|].$$
As $\mathcal{S}^{0}=\{-1,1\}$ and the support functions satisfy $s_{A}(-1)=-a^{-}, s_{A}(1)=a^{+}, s_{B}(-1)=-b^{-}, s_{B}(1)=b^{+}$, the same results as before can be deduced from definition (\ref{9}).

\begin{remarque}
Une représentation alternative d'un intervalle $A=\left[a^{-}, a^{+}\right]$ est par l'utilisation du point médian $\widehat{a}=a^{-}+a^{+} / 2$ et la (semi) largeur $\bar{a}=a^{+}-a^{-} / 2$ et nous pouvons écrire $A=(\widehat{a}, \bar{a}), \bar{a} \geq 0,$ de sorte que $a^{-}=\widehat{a}-\bar{a}$ et $a^{+}=\widehat{a}+\bar{a}$.
Si $B=(\widehat{b}, \bar{b}), \bar{b} \geq 0$ est un deuxième intervalle, l'addition de Minkowski est $A+B=(\widehat{a}+\widehat{b}, \bar{a}+\bar{b})$ et La différence gH est obtenue par $A \Theta_{g} B=(\widehat{a}-\widehat{b},|\bar{a}-\bar{b}|) .$On voit tout de suite que $A \Theta_{g} A=\{0\}, A=B \quad \Longleftrightarrow A \Theta_{g} B=\{0\}$,
$(A+B) \Theta_{g} B=A,$ mais $A+\left(B \Theta_{g} A\right)=B$ seulement si $\bar{a} \leq \bar{b}$
\end{remarque}
 Soit maintenant $A=\times_{i=1}^{n} A_{i}$ et $B=\times_{i=1}^{n} B_{i}$ où $A_{i}=\left[a_{i}^{-}, a_{i}^{+}\right], B_{i}=\left[b_{i}^{-}, b_{i}^{+}\right]$ sont des intervalles réels compacts $\left(\times_{i=1}^{n}\right.$
désigne le produit cartésien). Si $A \Theta_{g} B$ existe, alors l'égalité suivante est vraie:
$$
A \Theta_{g} B=\times_{i=1}^{n}\left(A_{i} \Theta_{g} B_{i}\right)
$$
En fait, considérez la fonction de support de $A$ (et de même pour $B$ ), Défini par
$s_{A}(p)=\max _{x}\left\{\langle p, x\rangle \mid a_{i}^{-} \leq x_{i} \leq a_{i}^{+}\right\}, \quad p \in \mathcal{S}^{n-1}$

il peut être obtenu simplement par $s_{A}(p)=\sum_{p_{i}>0} p_{i} a_{i}^{+}+\sum_{p_{i}<0} p_{i} a_{i}^{-}$ en tant que maxima contraint par la boîte des fonctions objectifs linéaires $\langle p, x\rangle$ ci-dessus sont atteints aux sommets $\widehat{x}(p)=\left(\widehat{x}_{1}(p), \ldots, \widehat{x}_{i}(p), \ldots, \widehat{x}_{n}(p)\right)$ de $A,$ i.e. $\widehat{x}_{i}(p) \in\left\{a_{i}^{-}, a_{i}^{+}\right\}, i=1,2, \ldots, n .$ Alors
$$
s_{A}(p)-s_{B}(p)=\sum_{p_{i}>0} p_{i}\left(a_{i}^{+}-b_{i}^{+}\right)+\sum_{p_{i}<0} p_{i}\left(a_{i}^{-}-b_{i}^{-}\right)
$$
et, étant
 $$s_{-A}(p)=s_{A}(-p)=-\sum_{p_{i}<0} p_{i} a_{i}^{+}-\sum_{p_{i}>0} p_{i} a_{i}^{-}$$
$$
s_{-B}(p)-s_{-A}(p)=\sum_{p_{i}>0} p_{i}\left(a_{i}^{-}-b_{i}^{-}\right)+\sum_{p_{i}<0} p_{i}\left(a_{i}^{+}-b_{i}^{+}\right)
$$

Des relations ci-dessus, nous en déduisons que
$$
A \Theta_{g} B=C \Longleftrightarrow\left\{\begin{array}{l}
\text { (i) }\left\{\begin{array}{l}
C=\times_{i=1}^{n}\left[a_{i}^{-}-b_{i}^{-}, a_{i}^{+}-b_{i}^{+}\right] \\
\text {provided that } a_{i}^{-}-b_{i}^{-} \leq a_{i}^{+}-b_{i}^{+},
\end{array} \quad \forall i\right. \\
\text { or (ii) }\left\{\begin{array}{l}
C=\times_{i=1}^{n}\left[a_{i}^{+}-b_{i}^{+}, a_{i}^{-}-b_{i}^{-}\right] \\
\text {provided that } a_{i}^{-}-b_{i}^{-} \geq a_{i}^{+}-b_{i}^{+},
\end{array}\right.
\end{array}\right.
$$
et :
\begin{proposition}
La différence gH $ A \Theta_{g} B $ existe si et seulement si l'une des deux conditions est satisfaite:\\
(i) $\quad a_{i}^{-}-b_{i}^{-} \leq a_{i}^{+}-b_{i}^{+}, \quad i=1,2, \ldots, n$\\
ou\\
(ii) $a_{i}^{-}-b_{i}^{-} \geq a_{i}^{+}-b_{i}^{+}, \quad i=1,2, \ldots, n .$
\end{proposition} 
\begin{exemple}
 1. cas (i): $A_{1}=[5,10], A_{2}=[1,3], B_{1}=[3,6], B_{2}=[2,3]$ for which $\left(A_{1} \Theta_{g} B_{1}\right)=[2,4],\left(A_{2} \Theta_{g} B_{2}\right)=$
[-1,0] and $A \Theta_{g} B=C=[2,4] \times[-1,0]$ existe avec $B+C=A, A+(-1) C \neq B$

2. cas (ii): $A_{1}=[3,6], A_{2}=[2,3], B_{1}=[5,10], B_{2}=[1,3]$ Pour qui $\left(A_{1} \Theta_{g} B_{1}\right)=[-4,-2],\left(A_{2} \Theta_{g} B_{2}\right)=$
[0,1] et $A \Theta_{g} B=C=[-4,-2] \times[0,1]$ existe avec $B+C \neq A, A+(-1) C=B$.

3. cas (i)+(ii): $A_{1}=[3,6], A_{2}=[2,3], B_{1}=[5,8], B_{2}=[3,4]$ Pour qui $\left(A_{1} \Theta_{g} B_{1}\right)=[-2,-2]=\{-2\},$
$\left(A_{2} \Theta_{g} B_{2}\right)=[-1,-1]=\{-1\}$ et $A \Theta_{g} B=C=\{(-2,-1)\}$ existe avec $B+C=A$ et $A+(-1) C=B.$
\end{exemple}
Nous avons vu que, généralement pour  $A, B \in \mathcal{K}_{C}^{n},$ une condition nécessaire pour $A \Theta_{g} B$ existe c'est que soit $A$ contient une translation de $B$ ou $B$ contient une translation de $A$. Dans le cas d'intervalles multidimensionnels, la même condition est également suffisante.

\begin{proposition}
Soit $A=\times_{i=1}^{n} A_{i}$ et $B=\times_{i=1}^{n} B_{i}$ où $A_{i}=\left[a_{i}^{-}, a_{i}^{+}\right], B_{i}=\left[b_{i}^{-}, b_{i}^{+}\right] .$ Si A contient une translation de $B$ ou si $B$ contient une translation de $A($ en particulier si $A \subseteq B$ ou si $B \subseteq A)$, alors $A \Theta_{g} B=\times_{i=1}^{n}\left(A_{i} \Theta_{g} B_{i}\right) .$
\end{proposition}
\begin{proof}
Considérons d'abord le cas $B \subseteq A$, c'est à dire $\left[b_{i}^{-}, b_{i}^{+}\right] \subseteq\left[a_{i}^{-}, a_{i}^{+}\right]$ et $b_{i}^{-} \geq a_{i}^{-}, b_{i}^{+} \leq a_{i}^{+} \forall i ;$ alors la proposition 10 s'applique et $A \Theta_{g} B$ existe selon le cas (i). De manière analogue, si $A \subseteq B$, c'est à dire $\left[a_{i}^{-}, a_{i}^{+}\right] \subseteq\left[b_{i}^{-}, b_{i}^{+}\right]$ et $b_{i}^{-} \leq a_{i}^{-}, b_{i}^{+} \geq a_{i}^{+} \forall i$, alors la proposition 10 s'applique et $A \Theta_{g} B$ existe selon le cas (ii). Si $A$ contient une translation de $B$, i.e. $\exists \widehat{b}=\left(\widehat{b}_{1}, \widehat{b}_{2}, \ldots, \widehat{b}_{n}\right) \in$ $\mathbb{R}^{n}$ tel que $\{\widehat{b}\}+B \subseteq A$, alors $A \Theta_{g}(\{\widehat{b}\}+B)=C^{\prime}$ existe selon le cas (ii) avec $a_{i}^{-}-b_{i}^{-}-\widehat{b}_{i} \leq a_{i}^{+}-b_{i}^{+}-\widehat{b}_{i}$
$\forall i ;$ si suit ça $a_{i}^{-}-b_{i}^{-} \leq a_{i}^{+}-b_{i}^{+}, \forall i$ et par proposition 10 $ A \Theta_{g} B$ existe selon le cas (i). Enfin, par un raisonnement similaire, si $ B $ contient une translation $\{\widehat{a}\}+A$ de $A$, alors $(\{\widehat{a}\}+A) \Theta_{g} B$ existe selon le cas (ii) de sorte que $a_{i}^{-}-b_{i}^{-} \geq a_{i}^{+}-b_{i}^{+}, \forall i$ et par proposition 10 $ A \Theta_{g} B$ existe selon le cas (ii).
\end{proof}
\section{Différence généralisé de HUKUHARA des nombres flous}

Un ensemble flou général sur un ensemble (ou espace) donné $\mathbb{X}$ des éléments (l'univers) est généralement défini par sa fonction d'appartenance $\mu: \mathbb{X} \longrightarrow \mathbb{T} \subseteq[0,1]$ et un flou (sous-)ensemble $u$ de $\mathbb{X}$ est uniquement caractérisé par les paires $\left(x, \mu_{u}(x)\right)$ pour chaque $x \in \mathbb{X} ;$ la valeur $\mu_{u}(x) \in[0,1]$ est la qualité de membre de $x$ à l'ensemble flou $u$ et $\mu_{u}$ est la fonction d'appartenance d'un ensemble flou $u$ over $\mathbb{X}$ pour les origines de la théorie des ensembles flous). Le support de $ u $ est le sous-ensemble des points de $\left.\mathbb{X} \right.$ telle que $\mu_{u}(x)$ est positive: $\operatorname{supp}(u)=\left\{x \mid x \in \mathbb{X}, \mu_{u}(x)>0\right\} .$ Pour $\left.\left.\alpha \in\right] 0,1\right],$ le $\alpha$-niveau $coupe$ de $u$ (ou simplement le $\alpha-Coupe$ ) est défini par $[u]_{\alpha}=\left\{x \mid x \in \mathbb{X}, \mu_{u}(x) \geq \alpha\right\}$ et pour $\alpha=0$ (ou $\alpha \rightarrow+0$ ) par la fermeture du support $[u]_{0}=c l\left\{x \mid x \in \mathbb{X}, \mu_{u}(x)>0\right\}$

Nous examinerons le cas $\mathbb{X}=\mathbb{R}^{n}$ avec $n \geq 1$. Une classe particulière d'ensembles flous $ u $ est lorsque le support est un ensemble convexe et la fonction d'appartenance est quasi-concave (i.e. $\mu_{u}\left((1-t) x^{\prime}+t x^{\prime \prime}\right) \geq \min \left\{\mu_{u}\left(x^{\prime}\right), \mu_{u}\left(x^{\prime \prime}\right)\right\}$ pour tout $x^{\prime}, x^{\prime \prime} \in \operatorname{supp}(u)$
et $t \in[0,1]) .$ De manière équivalente, $\mu_{u}$ est quasi-concave si le niveau  $[u]_{\alpha}$ sont des ensembles convexes pour tous $\alpha \in[0,1] .$ Nous exigerons également que  $[u]_{\alpha}$ sont des ensembles fermés pour tous $\alpha \in[0,1]$ et que la fonction d'appartenance est normale,c'est à dire. le noyau $[u]_{1}=\left\{x \mid \mu_{u}(x)=1\right\}$ est compact et non vide.

Les propriétés suivantes caractérisent les ensembles flous normaux, convexes et semi-continus supérieurs (en termes de coupes de niveau):

(F1) $[u]_{\alpha} \in \mathcal{K}_{C}\left(\mathbb{R}^{n}\right)$ pour tout $\alpha \in[0,1] ;$

(F2) $[u]_{\alpha} \subseteq[u]_{\beta}$ pour $\alpha \geq \beta$ (c'est à dire ils sont imbriqués);

(F3) $[u]_{\alpha}=\bigcap_{k=1}^{\infty}[u]_{\alpha_{k}}$ pour toutes les séquences croissantes $\alpha_{k} \uparrow \alpha$ convergent vers $\alpha$.

De plus, toute famille $\left\{U_{\alpha} \mid \alpha \in[0,1]\right\}$  satisfaisant les conditions $(\mathrm{F} 1)-(\mathrm{F} 3)$ représente les coupes de niveau d'un ensemble flou $u$ ayant $[u]_{\alpha}=U_{\alpha}$

Nous désignerons par $\mathcal{F}^{n}$ l'ensemble des ensembles flous avec les propriétés ci-dessus (également appelés quantités floues). L'espace $\mathcal{F}^{n}$ des quantités floues réelles est structurée par une addition et une multiplication scalaire, définies soit par les level sets, soit, de manière équivalente, par le principe d'extension de Zadeh.

Soit $u, v \in \mathcal{F}^{n}$ avoir des fonctions d'adhésion $\mu_{u}, \mu_{v}$ et $\alpha-coupe$ $s[u]_{\alpha},[v]_{\alpha}, \alpha \in[0,1],$ respectivement. Dans le cas unidimensionnel $u \in \mathcal{F},$ nous désignerons par $[u]_{\alpha}=\left[u_{\alpha}^{-}, u_{\alpha}^{+}\right]$ les intervalles compacts formant les coupes $ \alpha $ et les quantités floues seront appelés nombres flous.

L'addition $u+v \in \mathcal{F}^{n}$ et la multiplication scalaire $k u \in^{n} \mathcal{F}$ avoir des coupes de niveau :
\begin{equation}
\mid[u+v]_{\alpha}=[u]_{\alpha}+[v]_{\alpha}=\left\{x+y \mid x \in[u]_{\alpha}, y \in[v]_{\alpha}\right\}
\label{10}
\end{equation}
\begin{equation}
[k u]_{\alpha}=k[u]_{\alpha}=\left\{k x \mid x \in[u]_{\alpha}\right\}
\label{11}
\end{equation}
Dans le contexte flou ou dans le contexte arithmétique d'intervalle, l'équation $u=v+w$ n'est pas équivalent à $w=u-v=u+(-1) v$ ou pour $v=u-w=u+(-1) w$ et cela a motivé l'introduction de la différence Hukuhara suivante :

\begin{definition}
Étant donné $u, v \in \mathcal{F}^{n},$ la différence H est définie par
$$u \Theta v=w \Longleftrightarrow u=v+w .$$
Clairement, $u \Theta u=\{0\} ;$ si $u \Theta v$ existe, c'est unique.
\end{definition}
Dans le cas unidimensionnel $(n=1),$ le $\alpha$-coupe de la différence H sont $[u \Theta v]_{\alpha}=\left[u_{\alpha}^{-}-v_{\alpha}^{-}, u_{\alpha}^{+}-v_{\alpha}^{+}\right]$ où $[u]_{\alpha}=$ $\left[u_{\alpha}^{-}, u_{\alpha}^{+}\right]$ and $[v]_{\alpha}=\left[v_{\alpha}^{-}, v_{\alpha}^{+}\right]$

La différence Hukuhara est également motivée par le problème de l'inversion de l'addition : si $x, y$ sont des chiffres précis alors $(x+y)-y=x$ mais ce n'est pas vrai si $x, y$ sont flous. Il est possible de voir que, si $u$ et $v$ sont des nombres flous (et pas en général des ensembles flous), alors $(u+v) \Theta v=u$ c'est à dire la différence H inverse l'addition de nombres flous.

La différence gH pour les nombres flous peut être définie comme suit:
\begin{definition}
 Étant donné $u, v \in \mathcal{F}^{n},$ la différence gH est la quantité floue $w \in \mathcal{F}^{n},$ s'il existe, de telle sorte que
\begin{equation}
u \Theta_{g} v=w \Longleftrightarrow\left\{\begin{aligned}
\text { (i) } u=v+w \\
\text { or (ii) } v=u+(-1) w
\end{aligned}\right.
\label{aa}
\end{equation}
Si $u \Theta_{g} v$ et $u \Theta v$ existe, $u \ominus v=u \Theta_{g} v ;$ si (i) et (ii) sont satisfaits simultanément, alors $ w $ est une quantité nette. Aussi, $u \Theta_{g} u=u \Theta u=\{0\}$
\end{definition}

\subsection{Fonctions de support et différence gh floue}
Une définition équivalente de $w=u \Theta_{g} v$ pour les nombres flous multidimensionnels peuvent être obtenus en termes de fonctions de support d'une manière similaire à Eq. (\ref{10}):
\begin{equation}
\color{purple}{
s_{w}(p ; \alpha)=\left\langle\begin{array}{ll}
s_{u}(p ; \alpha)-s_{v}(p ; \alpha) & \text { in case (i), } \\
s_{(-1) v}(p ; \alpha)-s_{(-1) u}(p ; \alpha) & \text { in case (ii) }
\end{array} \quad \alpha \in[0,1]\right.}
\label{12}
\end{equation}
où, pour une quantité floue $ u, $ les fonctions de support sont considérées pour chaque $\alpha-coupe$ et défini par :
$$
\color{purple}{
\begin{aligned}
s_{u} &: \mathbb{S}^{n-1} \times[0,1] \longrightarrow \mathbb{R} \\
s_{u}(p ; \alpha) &=\sup \left\{\langle p, x\rangle \mid x \in[u]_{\alpha}\right\} \quad \text { Pour tout } p \in \mathbb{S}^{n-1}, \alpha \in[0,1]
\end{aligned}}
$$
En tant que fonction de $\alpha, s_{u}(p, \cdot)$ n'augmente pas pour tous $p \in \mathbb{S}^{n-1},$ en raison de la propriété de nidification du $\alpha$-coupe.

\begin{proposition}
 Soit $s_{u}(p, \alpha)$ et $s_{v}(p, \alpha)$ être les fonctions de support de deux quantités floues $u, v \in \mathcal{F}^{n} .$ Considérons $s_{1}=$ $s_{u}-s_{v}, s_{2}=s_{v}-s_{u} ;$ les quatre cas suivants s'appliquent:
 
1) Si $s_{1}$ et $s_{2}$ sont tous les deux sous-additifs en p pour tout $\alpha \in[0,1]$ et n'augmentent pas pour tous $p,$ alors $u \Theta_{g} v$ existe; (i) $a n d$ (ii) in (\ref{12}) sont satisfaits simultanément et $u \Theta_{g} v$ est crisp.

2) Si $s_{1}$ est sous-additif dans $p$ pour tout $\alpha \in[0,1]$ et sans augmentation pour tous $p$ et $s_{2}$ n'est pas, alors $w=u \Theta_{g} v$ existe,(i) est satisfait et $s_{w}=s_{u}-s_{v}$

3) Si $s_{2}$ est sous-additif dans $p$ pour tout $\alpha \in[0,1]$ et sans augmentation pour tous $p$ et $s_{1}$ n'est pas, alors $w=u \Theta_{g} v$ existe,(ii) est satisfait et $s_{w}=s_{-v}-s_{-u}$

4) Si $s_{1}$ et $s_{2}$ ne sont pas à la fois sous-additifs et non croissants pour tous $p,$ alors $u \Theta_{g} v$ n'existe pas.

\end{proposition}

\begin{proof}
La preuve est similaire à la preuve de la proposition $7.$ Pour $\alpha \in] 0,1]$ considérez les ensembles $W_{\alpha}=\left\{x \in \mathbb{R}^{n} \mid\langle p, x\rangle \leq s_{u}(p ; \alpha)\right.$ pour tout $p \in \mathbb{R}^{n}$ avec $\left.\|p\|=\alpha\right\}$, et $W_{0}=c l\left(\bigcup_{\alpha \in] 0,1]} W_{\alpha}\right) .$ Propriété (F1) pour $\left\{W_{\alpha}\right\}$ est assurée par la validité de la propriété sous-additive pour tous $\alpha \in[0,1] .$ La condition de monotonie $\forall p$ assure la propriété d'imbrication (F2) du $\alpha$-coupe. Il reste à montrer la propriété (F3), c'est-à-dire $W_{\alpha}=\bigcap_{k=1}^{\infty} W_{\alpha_{k}}$ pour toutes les séquences croissantes $\alpha_{k} \uparrow \alpha$ convergent vers $\left.\left.\alpha \in\right] 0,1\right]$. Comme $W_{\alpha} \subseteq W_{\alpha_{k}}$ on a  $W_{\alpha} \subseteq \bigcap_{k=1}^{\infty} W_{\alpha_{k}} ;$ soit maintenant $x \in \bigcap_{k=1}^{\infty} W_{\alpha_{k}}$, pour tout $p$ ayant $\|p\|=\alpha$ et tout $k=1,2, \ldots$ on a 
$\left\langle\left(\alpha_{k} / \alpha\right) p, x\right\rangle \leq s_{u}\left(\left(\alpha_{k} / \alpha\right) p ; \alpha\right)$ et en prenant la limite pour $k \longrightarrow \infty$ nous obtenons, comme $s_{w}$ est continue (semi-continue supérieure), $\langle p, x\rangle=\lim \left\langle\left(\alpha_{k} / \alpha\right) p, x\right\rangle \leq \lim \sup s_{u}\left(\left(\alpha_{k} / \alpha\right) p ; \alpha\right)=s_{u}(p ; \alpha)$ et $x \in W_{\alpha} .$ La preuve est complète.
\end{proof}

Il suit immédiatement une condition nécessaire et suffisante pour que  $u \Theta_{g} v$ soit existe :
\begin{proposition}
 Soit $u, v \in \mathcal{F}^{n}$ être donné avec des fonctions de supports $s_{u}(p, \alpha)$ et $s_{v}(p, \alpha) ;$ alors $u \Theta_{g} v \in \mathcal{F}^{n}$ existe si et seulement si au moins une des deux fonctions $s_{u}-s_{v}, s_{-v}-s_{-u}$ est une fonction de support et n'augmente pas avec $\alpha$ pour tout $p .$
 \end{proposition}
Dans le cas unidimensionnel, les conditions de la définition de $w=u \Theta_{g} v$ sont
$$
[w]_{\alpha}=\left[w_{\alpha}^{-}, w_{\alpha}^{+}\right]=[u]_{\alpha} \Theta_{g}[v]_{\alpha}
$$
et
\begin{equation}
\left\{\begin{array}{l}
w_{\alpha}^{-}=\mid \min \left\{u_{\alpha}^{-}-v_{\alpha}^{-}, u_{\alpha}^{+}-v_{\alpha}^{+}\right\} \\
w_{\alpha}^{+}=\max \left\{u_{\alpha}^{-}-v_{\alpha}^{-}, u_{\alpha}^{+}-v_{\alpha}^{+}\right\}
\end{array}\right.
\label{13}
\end{equation}
à condition que $w_{\alpha}^{-}$ ne diminue pas, $w_{\alpha}^{+}$ n'augmente pas et $w_{1}^{-} \leq w_{1}^{+} ;$ en particulier, pour $\alpha \in[0,1],$\\
\begin{tcolorbox}
\begin{equation}
\begin{aligned}
(1) \left\{\begin{array}{l}w_{\alpha}^{-}=u_{\alpha}^{-}-v_{\alpha}^{-} \\ w_{\alpha}^{+}=u_{\alpha}^{+}-v_{\alpha}^{+}\end{array} \quad\right. if \operatorname{len}\left([u]_{\alpha}\right) \leq \operatorname{len}\left([v]_{\alpha}\right),\\
(2) \left\{\begin{array}{l}w_{\alpha}^{-}=u_{\alpha}^{+}-v_{\alpha}^{+} \\ w_{\alpha}^{+}=u_{\alpha}^{-}-v_{\alpha}^{-}\end{array} \quad\right. if \operatorname{len}\left([u]_{\alpha}\right) \geq \operatorname{len}\left([v]_{\alpha}\right),
\end{aligned}
\label{14}
\end{equation}
\end{tcolorbox}
où $\operatorname{len}\left([u]_{\alpha}\right)=u_{\alpha}^{+}-u_{\alpha}^{-}$ est la longueur du $\alpha$ -coupe de $u$ (de la même manière $\operatorname{len}\left([v]_{\alpha}\right)$ pour $\left.v\right)$

\begin{proposition}
 Soit $u, v \in \mathcal{F}$ être deux nombres flous avec $\alpha$-coupe donné par$[u]_{\alpha}$ et $[v]_{\alpha},$ respectivement; les $g H$ -difference $u \Theta_{g} v \in \mathcal{F}$ existe si et seulement si l'une des deux conditions est satistaite :
 $$
\text { (a) }\left\{\begin{array}{l}
\operatorname{len}\left([u]_{\alpha}\right) \geq \operatorname{len}\left([v]_{\alpha}\right) \text { for all } \alpha \in[0,1], \\
u_{\alpha}^{-}-v_{\alpha}^{-} \text {is increasing with respect to } \alpha, \\
u_{\alpha}^{+}-v_{\alpha}^{+} \text {is decreasing with respect to } \alpha
\end{array}\right.
$$
ou
$$
\text { (b) }\left\{\begin{array}{l}
\operatorname{len}\left([u]_{\alpha}\right) \leq \operatorname{len}\left([v]_{\alpha}\right) \text { for all } \alpha \in[0,1], \\
u_{\alpha}^{+}-v_{\alpha}^{+} \text {is increasing with respect to } \alpha, \\
u_{\alpha}^{-}-v_{\alpha}^{-} \text {is decreasing with respect to } \alpha .
\end{array}\right.
$$
 
\end{proposition}
\begin{proof}
En fait, considérons la fonction de support de $ u $ (et de même pour $ v $), obtenue par :
\begin{equation}
\begin{aligned}
s_{u}(p ; \alpha) &=\max _{x}\left\{p x \mid u_{\alpha}^{-} \leq x \leq u_{\alpha}^{+}\right\} \\
p & \in \mathcal{S}^{0}=\{-1,1\} \\
\alpha & \in[0,1]
\end{aligned}
\label{15}
\end{equation}
c'est-à-dire simplement par $s_{u}(-1 ; \alpha)=-u_{\alpha}^{-}$ et $s_{u}(1 \alpha)=u_{\alpha}^{+}$. Alors
\begin{equation}
s_{u}(p ; \alpha)-s_{v}(p ; \alpha)=\left\{\begin{array}{ll}
-\left(u_{\alpha}^{-}-v_{\alpha}^{-}\right) & \text {if } p=-1 \\
u_{\alpha}^{+}-v_{\alpha}^{+} & \text {if } p=1
\end{array}\right.
\label{16}
\end{equation}
et
\begin{equation}
s_{-v}(p ; \alpha)-s_{-u}(p ; \alpha)=\left\{\begin{array}{ll}
-\left(u_{\alpha}^{+}-v_{\alpha}^{+}\right) & \text {if } p=-1 \\
u_{\alpha}^{-}-v_{\alpha}^{-} & \text {if } p=1
\end{array}\right.
\label{17}
\end{equation}
À partir des relations ci-dessus et de la proposition 13, nous en déduisons que (symboles $\nearrow$ et $\searrow$ signifie que la fonction augmente ou diminue, respectivement):

\begin{equation}
u \Theta_{g} v=w \Longleftrightarrow\left\{\begin{array}{l}
\text { (a) }\left\{\begin{array}{l}
{[w]_{\alpha}=\left[u_{\alpha}-v_{\alpha}^{-}, u_{\alpha}^{+}-v_{\alpha}^{+}\right]} \\
\text {provided that } u_{\alpha}-v_{\alpha}^{-} \leq u_{\alpha}^{+}-v_{\alpha}^{+}, \quad \forall \alpha \\
\text { and } u_{\alpha}-v_{\alpha}^{-} \nearrow, u_{\alpha}^{+}-v_{\alpha}^{+} \searrow
\end{array}\right. \\
\text { or } \\
\text { (b) }\left\{\begin{array}{l}
{[w]_{\alpha}=\left[u_{\alpha}^{+}-v_{\alpha}^{+}, u_{\alpha}-v_{\alpha}^{-}\right]} \\
\text {provided that } u_{\alpha}^{-}-v_{\alpha}^{-} \geq u_{\alpha}^{+}-v_{\alpha}^{+}, \\
\text {and } u_{\alpha}^{+}-v_{\alpha}^{+} \nearrow, u_{\alpha}-v_{\alpha}^{-}
\end{array} \forall \alpha\right.
\end{array}\right.
\label{18}
\end{equation}
et la preuve est complète.
\end{proof}
 La monotonie de $u_{\alpha}^{-}-v_{\alpha}^{-}$ et $u_{\alpha}^{+}-v_{\alpha}^{+}$ selon (a) ou (b) dans \ref{18} est une condition importante pour l'existence de $u \Theta_{g} v$ et doit être vérifié explicitement car en fait, il peut ne pas être satisfait. Considérons $[u]_{\alpha}=[5+4 \alpha, 11-2 \alpha]$ et $[v]_{\alpha}=[12+3 \alpha, 19-4 \alpha] ;$ alors $u_{\alpha}^{-}-v_{\alpha}^{-}=-6+\alpha, u_{\alpha}^{+}-v_{\alpha}^{+}=-8+2 \alpha$ et $u_{\alpha}^{+}-v_{\alpha}^{+}<u_{\alpha}^{-}-v_{\alpha}^{-}$ mais $u_{\alpha}^{-}-v_{\alpha}^{-}$ ne diminue pas comme dans \ref{18} (b).
\begin{remarque}
Les conditions (a) et (b) de la proposition ci-dessus sont toutes deux valides si $\operatorname{len}\left([u]_{\alpha}\right)=\operatorname{len}\left([v]_{\alpha}\right)$ pour tout $\alpha \in[0,1] ;$ dans ce cas, $u \Theta_{g} v$ est une quantité grisp.
\end{remarque}
\begin{exemple}(Cas d'appartenance linéaire)\\
 Si $u$ et $v$ sont des nombres flous de forme linéaire trapézoïdale, désignés par $u=$ $\left\langle u_{0}^{-}, u_{1}^{-}, u_{1}^{+}, u_{0}^{+}\right\rangle$ avec $u_{0}^{-} \leq u_{1}^{-} \leq u_{1}^{+} \leq u_{0}^{+}$ (de même pour $v$ ) et $\alpha-coupe$ $ [u]_{\alpha}=\left[u_{0}^{-}+\alpha\left(u_{1}^{-}-u_{0}^{-}\right), u_{0}^{+}+\alpha\left(u_{1}^{+}-u_{0}^{+}\right)\right],$
alors $u \Theta_{g} v$ existe si et seulement si :

(a) $u_{0}^{-}-v_{0}^{-} \leq u_{1}^{-}-v_{1}^{-} \leq u_{1}^{+}-v_{1}^{+} \leq u_{0}^{+}-v_{0}^{+}$

ou

 (b) $u_{0}^{-}-v_{0}^{-} \geq u_{1}^{-}-v_{1}^{-} \geq u_{1}^{+}-v_{1}^{+} \geq u_{0}^{+}-v_{0}^{+}$ 
 
En particulier, si $u$ et $v$ sont des nombres flous triangulaires (linéaires) (i.e. $u_{1}^{-}=u_{1}^{+}$ et $\left.v_{1}^{-}=v_{1}^{+}\right),$ désigné par $u=\left\langle u_{0}^{-}, \widehat{u}, u_{0}^{+}\right\rangle$ et $v=\left\langle v_{0}^{-}, \widehat{v}, v_{0}^{+}\right\rangle,$ then $w=u \Theta_{g} v$ existe si et seulement si :

(a) $u_{0}^{-}-v_{0}^{-} \leq \widehat{u}-\widehat{v} \leq u_{0}^{+}-v_{0}^{+}$

ou 

(b) $u_{0}^{-}-v_{0}^{-} \geq \widehat{u}-\widehat{v} \geq u_{0}^{+}-v_{0}^{+}$

\end{exemple}

Pour illustrer:

 $\langle 12,15,19\rangle \Theta_{g}\langle 5,9,11\rangle$ n'existe pas en tant que $u_{\alpha}^{-}-v_{\alpha}^{-}=-6+\alpha$ et $u_{\alpha}^{+}-v_{\alpha}^{+}=-8+2 \alpha$ augmentent tous les deux par rapport à $\alpha \in[0,1]$ et les conditions (\ref{18}) ne sont pas satisfaites;

 $\langle 12,15,19\rangle \odot_{g}\langle 5,7,10\rangle=\langle 7,8,9\rangle$, selon (a);
 
  $\langle 12,15,19\rangle \Theta_{g}\langle 9,13,18\rangle=\langle 1,2,3\rangle$, selon (b).
  
Si $u \Theta_{g} v$ est un nombre flou, il a les mêmes propriétés que celles illustrées dans la section 2 pour les intervalles. En particulier, les mêmes propriétés que dans la proposition 6 sont également immédiates.
\begin{proposition}
 Soit $u, v \in \mathcal{F} .$ Si $u \Theta_{g} v$ existe, it est unique et possède les propriétés suivantes $(0$ désigne l'ensemble croustillant \{0\} ):
 
(1) $u \ominus_{g} u=0$

(2) (a) $(u+v) \Theta_{g} v=u ;$ (b) $u \Theta_{g}(u-v)=v$.

(3) Si $u \Theta_{g} v$ existe alors aussi $(-v) \Theta_{g}(-u)$ fait et $0 \ominus_{g}\left(u \Theta_{g} v\right)=(-v) \Theta_{g}(-u)$.

(4) $u \Theta_{g} v=v \Theta_{g} u=w$ si et seulement si $w=-w$ (en particulier $w=0$ si et seulement si $\left.u=v\right)$.

(5) Si $v \Theta_{g} u$ existe alors soit $u+\left(v \Theta_{g} u\right)=u$ ou $v-\left(v \Theta_{g} u\right)=u$ et si les deux égalités tiennent alors $v \Theta_{g} u$ est un crisp ensemble.
\end{proposition}

\subsection{Décomposition des nombres flous et différence gh}
Étant donné $u \in \mathcal{F}$ avec $\alpha$ -coupe $[u]_{\alpha}=\left[u_{\alpha}^{-}, u_{\alpha}^{+}\right],$ définir les quantités suivantes:
\begin{equation}
\widehat{u}=\left[\widehat{u}^{-}, \widehat{u}^{+}\right].
\label{19}
\end{equation}
 le noyau $\left[u_{1}^{-}, u_{1}^{+}\right],$ correspondant au $(\alpha=1)-c u t ;$
\begin{equation}
\tilde{u}_{\alpha}=\frac{u_{\alpha}^{-}+u_{\alpha}^{+}}{2}-\frac{\widehat{u}^{-}+\widehat{u}^{+}}{2}
\label{20}
\end{equation}

le profil de symétrie (brièvement le profil) de $u$;
\begin{equation}
\bar{u}_{\alpha}=\frac{u_{\alpha}^{+}-u_{\alpha}^{-}}{2}-\frac{\widehat{u}^{+}-\widehat{u}^{-}}{2}
\label{21}
\end{equation}
La composante floue symétrique de $ u $. Clairement, nous avons
\begin{equation}
u_{\alpha}^{-}=\widehat{u}^{-}+\widetilde{u}_{\alpha}-\bar{u}_{\alpha}, \quad u_{\alpha}^{+}=\widehat{u}^{+}+\tilde{u}_{\alpha}+\bar{u}_{\alpha}
\label{22}
\end{equation}
et, en notation d'intervalle,
\begin{equation}
[u]_{\alpha}=\widehat{u}+\left\{\tilde{u}_{\alpha}\right\}+\left[-\bar{u}_{\alpha}, \bar{u}_{\alpha}\right], \quad \alpha \in[0,1]
\label{23}
\end{equation}
On obtient alors une décomposition de $u \in \mathcal{F}$ en termes de trois composants:

1. $\widehat{u}=\left[\widehat{u}^{-}, \widehat{u}^{+}\right] \in \mathbb{I}$ est un intervalle réel compact standard;

2. $\tilde{u}:[0,1] \longrightarrow \mathbb{R}$ est une fonction donnée telle que, pour $\alpha=1, \tilde{u}_{1}=0 ;$ dénoter avec $\mathbb{P}$ l'ensemble de toutes ces fonctions (de profil);

3. $\bar{u}$ est un nombre flou symétrique dont le noyau est donné par le singleton \{0\} ; dénoté par $\mathbb{S}_{0}$ la famille de tous ces nombres flous.

En utilisant les trois éléments précédents, tout nombre ou intervalle flou peut être représenté par un triplet (avec un petit abus de notation, on note $ \bar{u} $ à la fois la fonction dans (\ref{21)} et le 0 -symétrique flou nombre obtenu avec):
$$
u=(\widehat{u}, \tilde{u}, \bar{u}) \in \mathbb{\mathbb{I}} \times \mathbb{P} \times \mathbb{S}_{0}
$$
i.e.

 $\widehat{u} \in \mathbb{I}$ (nombre ou intervalle précis),
 
  $\tilde{u} \in \mathbb{P}$ (profil de symétrie net),
  
   $\bar{u} \in \mathbb{S}_{0} \quad(0$ -nombre flou symétrique $)$
   
La fonction de profil $\tilde{u} \in \mathbb{P}$ et le nombre flou symétrique $\bar{u} \in \mathbb{S}_{0}$ forment ce que nous appelons une paire valide:
\begin{definition}
 Une paire d'éléments $(\tilde{u}, \bar{u}) \in \mathbb{P} \times \mathbb{S}_{0}$ est dit former une paire valide s'il représente un nombre flou ayant $\alpha$ -coupe $\left[\tilde{u}_{\alpha}-\bar{u}_{\alpha}, \tilde{u}_{\alpha}+\bar{u}_{\alpha}\right]$, i.e. si les conditions suivantes sont satisfaits:
$$
\alpha^{\prime}<\alpha^{\prime \prime} \Longrightarrow\left\{\begin{array}{l}
\tilde{u}_{\alpha^{\prime}}-\bar{u}_{\alpha^{\prime}} \leq \widetilde{u}_{\alpha^{\prime \prime}}-\bar{u}_{\alpha^{\prime \prime}} \\
\widetilde{u}_{\alpha^{\prime}}+\bar{u}_{\alpha^{\prime}} \geq \widetilde{u}_{\alpha^{\prime \prime}}+\bar{u}_{\alpha^{\prime \prime}},
\end{array}\right.
$$
i.e. si $\tilde{u}_{\alpha}-\bar{u}_{\alpha}$ est une fonction non décroissante et $\tilde{u}_{\alpha}+\bar{u}_{\alpha}$ est une fonction non croissante (notez que pour $\alpha=1$ tout les deux $\widetilde{u}$ et $\bar{u}$ sont nuls).
\end{definition}
Par la définition ci-dessus, nous pouvons définir la décomposition suivante des nombres flous (intervalles) $u \in \mathcal{F}:$

\begin{proposition}
Tout nombre flou (intervalle) $u \in \mathcal{F}$ avec $\alpha$ -coupe $[u]_{\alpha}=\left[u_{\alpha}^{-}, u_{\alpha}^{+}\right]$ peut être représenté sous la forme $u=$ $(\widehat{u}, \tilde{u}, \bar{u}) \in \mathbb{\mathbb{I}} \times \mathbb{P} \times \mathbb{S}_{0}$ défini dans $(\ref{19})-(\ref{21})$ et $(\widetilde{u}, \bar{u})$ est une paire valide. Vice versa, tout triplet valide $(\widehat{u}, \tilde{u}, \bar{u}) \in \mathbb{I} \times \mathbb{P} \times \mathbb{S}_{0}$
(i.e. $(\widetilde{u}, \bar{u})$ est une paire valide) représente un nombre flou (intervalle) $u \in \mathcal{F}$ avec $\alpha$ -coupe donné par (\ref{23}).
\end{proposition}
\begin{proof}
see \cite{Ref9}
\end{proof}
\begin{definition} Nous appelons
$$
u=(\widehat{u}, \tilde{u}, \bar{u}) \in \mathbb{I} \times \mathbb{P} \times \mathbb{S}_{0}
$$
(où nous supposons $(\tilde{u}, \bar{u})$ être une paire valide) comme la décomposition CPS de $u \in \mathcal{F}(\mathrm{C}=\mathrm{Crisp}, \mathrm{P}=$ Profile $, \mathrm{S}=$ 0-Symétrique floue). Nous écrivons ceci comme
$$\mathcal{F}=\mathbb{\mathbb{I}} \times \mathbb{V}\left(\mathbb{P}, \mathbb{S}_{0}\right)$$
où $\mathbb{V}\left(\mathbb{P}, \mathbb{S}_{0}\right) \subset \mathbb{P} \times \mathbb{S}_{0}$ est l'ensemble de toutes les paires valides.
\end{definition}
Maintenant, nous pouvons trouver la différence gH $u \Theta_{g} v$ de $u, v \in \mathcal{F}$ en termes de décomposition CPS.

\begin{proposition}
 Étant donné $u=(\widehat{u}, \tilde{u}, \bar{u}) \in \mathcal{F}$ and $v=(\widehat{v}, \tilde{v}, \bar{v}) \in \mathcal{F},$ aa $g H$-difference $w=u \Theta_{g} v$ est donné par
  $$
\begin{aligned}
w=(\widehat{w}, \widetilde{w}, \bar{w}) & \in \mathcal{F} \text { with } \\
\widehat{w} &=\widehat{u} \Theta_{g} \widehat{v} \\
\widetilde{w}_{\alpha} &=\widetilde{u}_{\alpha}-\widetilde{v}_{\alpha} \\
\bar{w}_{\alpha} &=\left\{\begin{array}{ll}
\bar{u}_{\alpha}-\bar{v}_{\alpha} & \text { if } \bar{u}-\bar{v} \in \mathbb{S}_{0} \\
\bar{v}_{\alpha}-\bar{u}_{\alpha} & \text { if } \bar{v}-\bar{u} \in \mathbb{S}_{0}
\end{array}\right.
\end{aligned}
$$
et $u \Theta_{g} v$ existe si et seulement si l'une des deux conditions est vérifié:

(1) $\bar{u}-\bar{v} \in \mathbb{S}_{0}, \widehat{u} \Theta_{g} \widehat{v}$ existe dans le cas
(i) et $(\tilde{u}-\tilde{v}, \bar{u}-\bar{v}) \in \mathbb{V}\left(\mathbb{P}, S_{0}\right)$

 ou

(2) $\bar{v}-\bar{u} \in \mathcal{S}_{0}, \widehat{u} \Theta_{g} \widehat{v}$ existe dans le cas (ii) $a n d(\tilde{u}-\tilde{v}, \bar{v}-\bar{u}) \in \mathbb{V}\left(\mathbb{P}, S_{0}\right)$

En particulier, si $u=(\widehat{u}, \tilde{u}, \bar{u}) \in \mathcal{F}$ and $v=(\widehat{v}, \tilde{v}, \bar{v}) \in \mathcal{F}$ sont des nombres flous, i.e. $\widehat{u}^{-}=\widehat{u}^{+}=\widehat{u}$ et $\widehat{v}^{-}=\widehat{v}^{+}=\widehat{v}$
alors $u \Theta_{g} v$ existe (dans ce case $\widehat{w}=\widehat{u}-\widehat{v}$ est croustillant) si et seulement si l'une des deux conditions est satisfaite:

$$\left(1^{\prime}\right) \bar{u}-\bar{v} \in \mathbb{S}_{0} \texttt{ et } (\tilde{u}-\widetilde{v}, \bar{u}-\bar{v}) \in \mathbb{V}\left(\mathbb{P}, \mathbb{S}_{0}\right)$$

 ou
 
$$\left(2^{\prime}\right) \bar{v}-\bar{u} \in \mathbb{S}_{0} \texttt{ et } (\tilde{u}-\widetilde{v}, \bar{v}-\bar{u}) \in \mathbb{V}\left(\mathbb{P}, S_{0}\right)$$
\end{proposition}
\begin{proof}
Tout d'abord, comme $\widehat{u}, \widehat{v} \in \mathbb{l}$, alors $\widehat{w}=\widehat{u} \Theta_{g} \widehat{v} \in \mathbb{l}$ existe toujours et si $\tilde{u}, \widetilde{v} \in \mathbb{P}$ alors $\widetilde{w}=\tilde{u}-\widetilde{v} \in \mathbb{P} .$ Si $u \odot_{g} v$ existe alors clairement $\bar{u}-\bar{v} \in \mathbb{S}_{0}$ ou $\bar{v}-\bar{u} \in \mathbb{S}_{0}$ (comme en fait $\bar{w}=\bar{u}-\bar{v}$ ou $\bar{w}=\bar{v}-\bar{u}$ de Eqs. (\ref{14})) et $(\widetilde{u}-\widetilde{v}, \bar{u}-\bar{v})$ ou $(\tilde{u}-\widetilde{v}, \bar{v}-\bar{u})$ est une paire valide. Vice versa, si $\bar{u}-\bar{v} \in \mathbb{S}_{0}$ et $(\tilde{u}-\widetilde{v}, \bar{u}-\bar{v})$ est une paire valide (ou si $\bar{v}-\bar{u} \in \mathbb{S}_{0}$ et $(\widetilde{u}-\widetilde{v}, \bar{v}-\bar{u})$ is une paire valide), puis le triplet $\left(\widehat{u} \odot_{g} \widehat{v}, \tilde{u}-\widetilde{v}, \bar{u}-\bar{v}\right)$ (ou le triplet $\left.\left(\widehat{u} \odot_{g} \widehat{v}, \tilde{u}-\widetilde{v}, \bar{v}-\bar{u}\right)\right)$ représente un nombre flou (intervalle) et par $(\ref{13})$ c'est $u \Theta_{g} v .$ La deuxième partie est immédiate car la différence de gH $\widehat{u} \Theta_{g} \widehat{v}$ est la différence nette standard et $\widehat{w}$ est croustillant.
\end{proof}
\begin{remarque}
 Si les fonctions $u_{\alpha}^{-}$ et $u_{\alpha}^{+}$ sont différenciables par rapport à $\alpha$, alors $\tilde{u}_{\alpha}$ et $\bar{u}_{\alpha}$ sont différenciables et la condition pour $(\widetilde{u}, \bar{u})$ être une paire valide est $\left.\left|\tilde{u}_{\alpha}^{\prime}\right| \leq-\bar{u}_{\alpha}^{\prime} \forall \alpha \in\right] 0,1\left[\right.$ (rappeler que $\bar{u}_{\alpha}$ est une fonction décroissante).
\end{remarque}

\begin{exemple}
 Si nous avons des nombres flous symétriques $u, v \in \mathcal{F}$ (i.e. $\widehat{u}^{-}=\widehat{u}^{+}=\widehat{u}, \widehat{v}^{-}=\widehat{v}^{+}=\widehat{v},$ et $\tilde{u}_{\alpha}=\widetilde{v}_{\alpha}=0$
$\forall \alpha \in[0,1])$ alors (seule) condition $|\bar{u}-\bar{v}| \in \mathbb{S}_{0}$ i.e. $\left(\bar{u}-\bar{v} \in \mathbb{S}_{0}\right.$ or $\left.\bar{v}-\bar{u} \in \mathbb{S}_{0}\right)$ est nécessaire et suffisant pour $u \Theta_{g} v$
d'exister et $(\widehat{u}, 0, \bar{u}) \Theta_{g}(\widehat{v}, 0, \bar{v})=(\widehat{u}-\widehat{v}, 0,|\bar{u}-\bar{v}|)$ i.e. avec $\alpha$ -coupe
$$
\left[u \Theta_{g} v\right]_{\alpha}=\left[(\widehat{u}-\widehat{v})-\left|\bar{u}_{\alpha}-\bar{v}_{\alpha}\right|,(\widehat{u}-\widehat{v})+\left|\bar{u}_{\alpha}-\bar{v}_{\alpha}\right|\right]
$$
Par exemple, $\left(1,0,1-\alpha^{2}\right) \ominus_{g}(0,0,2-2 \alpha)=\left(1,0,(1-\alpha)^{2}\right),$ selon le cas (ii) de différence de gH, a $\alpha$ -coupe
$\left[1-(1-\alpha)^{2}, 1+(1-\alpha)^{2}\right].$
\end{exemple}
\begin{exemple}
Si nous avons des nombres flous triangulaires symétriques (avec appartenance linéaire) $u=\langle\widehat{u}-\Delta u, \widehat{u}, \widehat{u}+\Delta u\rangle$ et $v=\langle\widehat{v}-\Delta v, \widehat{v}, \widehat{v}+\Delta v\rangle$ alors $w=u \Theta_{g} v$ existe toujours et est le nombre flou symétrique triangulaire
$$
\begin{array}{l}
w=\langle\widehat{w}-\Delta w, \widehat{w}, \widehat{w}-\Delta w\rangle \quad \text { where } \\
\widehat{w}=\widehat{u}-\widehat{v} \text { and } \Delta w=|\Delta u-\Delta v| .
\end{array}
$$
En fait, dans ce cas, $\left|\bar{u}_{\alpha}-\bar{v}_{\alpha}\right|=|\Delta u-\Delta v|(1-\alpha)$, i.e. $\bar{u}-\bar{v} \in \mathbb{S}_{0}$ si $\Delta u-\Delta v \geq 0$ ou $\bar{v}-\bar{u} \in \mathbb{S}_{0}$ si $\Delta u-\Delta v \leq 0$ de sorte que
$$
\left[u \Theta_{g} v\right]_{\alpha}=[(\widehat{u}-\widehat{v})-|\Delta u-\Delta v|(1-\alpha),(\widehat{u}-\widehat{v})+|\Delta u-\Delta v|(1-\alpha)]
$$

En particulier, si $\Delta u=\Delta v$ alors $u \Theta_{g} v$ est un nombre précis.

 Pour illustrer, $\langle 2,4,6\rangle \Theta_{g}\langle-2,1,4\rangle=\langle 2,3,4\rangle$, selon le cas (ii) de la différence de gH, et en fait

$\langle 2,4,6\rangle+(-1)\langle 2,3,4\rangle=\langle 2,4,6\rangle+\langle-4,-3,-2\rangle=\langle-2,1,4\rangle$;

$\langle-2,1,4\rangle \Theta_{g}\langle 2,4,6\rangle=\langle-4,-3,-2\rangle$, selon le cas (i) de la différence de gH, et en fait

 $\langle 2,4,6\rangle+\langle-4,-3,-2\rangle=$ $\langle-2,1,4\rangle$
\end{exemple}

\subsection{Différence approximative floue gh}
Si les différences de gH $[u]_{\alpha} \Theta_{g}[v]_{\alpha}$ ne définissez pas un nombre flou approprié, nous pouvons utiliser la propriété imbriquée du $\alpha$ -coupe et obtenir un nombre flou approprié en
$$
[u \widetilde{\Theta} v]_{\alpha}:=c l\left(\bigcup_{\beta \geq \alpha}\left([u]_{\beta} \Theta_{g}[v]_{\beta}\right)\right) \quad \text { for } \alpha \in[0,1]
$$
Comme chaque différence de gH $[u]_{\beta} \Theta_{g}[v]_{\beta}$ existe pour $\beta \in[0,1]$ et la formule ci-dessus définit un nombre flou propre, il s'ensuit que $z=u \widetilde{\Theta} v$ peut être considérée comme une généralisation de la différence de Hukuhara pour les nombres flous, existant pour tout $u, v .$
\begin{exemple}
 $\langle 12,15,19\rangle$ $\Theta_{g}\langle 5,9,11\rangle$ n'existe pas; on obtient $[u]_{\beta} \Theta_{g}[v]_{\beta}=[8-2 \beta, 7-\beta]$ and $[u \widetilde{\Theta} v]_{\alpha}=[6,7-\alpha]$.
\end{exemple}
Une version discrétisée de $z=u \widetilde{\Theta} v$ peut être obtenu en choisissant une partition $0=\alpha_{0}<\alpha_{1}<\cdots<\alpha_{N}=1$ do $[0,1]$ et de $\left[w_{i}^{-}, w_{i}^{+}\right]=[u]_{\alpha_{i}} \odot_{g}[v]_{\alpha_{i}}$ par l'itération inverse suivante:
$$
\begin{aligned}
z_{N}^{-} &=w_{N}^{-}, \quad z_{N}^{+}=w_{N}^{+} \\
\text {For } k &=N-1, \ldots, 0:\left\{\begin{array}{l}
z_{k}^{-}=\min \left\{z_{k+1}^{-}, w_{k}^{-}\right\} \\
z_{k}^{+}=\max \left\{z_{k+1}^{+}, w_{k}^{+}\right\}
\end{array}\right.
\end{aligned}
$$
Une troisième possibilité pour une différence gH de nombres flous peut être obtenue en définissant $z=u \widetilde{\Theta}_{g} v$ être le nombre flou dont $\alpha$-coupes sont aussi proches que possible des différences de gH $[u]_{\alpha} \Theta_{g}[v]_{\alpha}$, par exemple en minimisant la fonction $\left(\omega_{\alpha} \geq 0\right.$ et $\gamma_{\alpha} \geq 0$ sont des fonctions de pondération)
\begin{equation}
G(z \mid u, v)=\int_{0}^{1}\left(\omega_{\alpha}\left[z_{\alpha}-\left(u \Theta_{g} v\right)_{\alpha}\right]^{2}+\gamma_{\alpha}\left[z_{\alpha}^{+}-\left(u \Theta_{g} v\right)_{\alpha}^{+}\right]^{2}\right) d \alpha
\label{24}
\end{equation}
telle que $z_{\alpha}^{-}$ augmente avec $\alpha, z_{\alpha}^{+}$ diminue avec $\alpha$ et $z_{\alpha}^{-} \leq z_{\alpha}^{+} \forall \alpha \in[0,1]$. Une version discrétisée de $G(z \mid u, v)$ peut être obtenu en choisissant une partition $0=\alpha_{0}<\alpha_{1}<\cdots<\alpha_{N}=1$ de $[0,1]$ et définissant le discrétisé $G(z \mid u, v)$ comme
$$
G_{N}(z \mid u, v)=\sum_{i=0}^{N} \omega_{i}\left[z_{i}^{-}-\left(u \Theta_{g} v\right)_{i}^{-}\right]^{2}+\gamma_{i}\left[z_{i}^{+}-\left(u \Theta_{g} v\right)_{i}^{+}\right]^{2},
$$

nous minimisons $G_{N}(z \mid u, v)$ avec le donné $\operatorname{data}\left(u \odot_{g} v\right)_{i}^{-}=\min \left\{u_{\alpha_{i}}-v_{\alpha_{i}}^{-}, u_{\alpha_{i}}^{+}-v_{\alpha_{i}}^{+}\right\}$ et $\left(u \odot_{g} v\right)_{i}^{+}=\max \left\{u_{\alpha_{i}}-v_{\alpha_{i}}^{-}, u_{\alpha_{i}}^{+}-\right.$
$\left.v_{\alpha_{i}}^{+}\right\}$, soumis aux contraintes $z_{0}^{-} \leq z_{1}^{-} \leq \cdots \leq z_{N}^{-} \leq z_{N}^{+} \leq z_{N-1}^{+} \leq \cdots \leq z_{0}^{+} .$ On obtient une minimisation des moindres carrés contraints linéairement de la forme

$$
\min_{z \in \mathbb{R}^{2 N+2}}(z-w)^{T} D^{2}(z-w) \text { s.t. } E z \geq 0 
$$
où
$
 z=\left(z_{0}^{-}, z_{1}^{-}, \ldots, z_{N}^{-}, z_{N}^{+}, z_{N-1}^{+}, \ldots, z_{0}^{+}\right), w_{i}^{-}=\left(u \Theta_{g} v\right)_{i}^{-}, w_{i}^{+}=\left(u \odot_{g} v\right)_{i}^{+}, w=\left(w_{0}^{-}, w_{1}^{-}, \ldots, w_{N}^{-}, w_{N}^{+},\right.
 $
 $
\left.w_{N-1}^{+}, \ldots, w_{0}^{+}\right), D=\operatorname{diag}\left\{\sqrt{\omega_{0}}, \ldots, \sqrt{\omega_{N}}, \sqrt{\gamma_{N}}, \ldots, \sqrt{\gamma_{0}}\right\}$ et $E$ est le $(N, N+1)$ matrice
$$
E=\left[\begin{array}{cccccc}
-1 & 1 & 0 & \cdots & \cdots & 0 \\
0 & -1 & 1 & 0 & \cdots & 0 \\
\vdots & \vdots & \vdots & \vdots & \cdots & \vdots \\
0 & 0 & \cdots & \cdots & -1 & 1
\end{array}\right]
$$
qui peut être résolu par des procédures efficaces standard. Si, à la solution $z^{*}$, on a  $z^{*}=w$, alors nous obtenons la différence de gH telle que définie dans (\ref{aa}).

\section{Division généralisée}
Une idée similaire à la différence gH peut être utilisée pour introduire une division des intervalles réels et des nombres flous. Considérons d'abord le cas d'intervalles compacts réels $A=\left[a^{-}, a^{+}\right]$ et $B=\left[b^{-}, b^{+}\right]$ avec $b^{-}>0$ ou $b^{+}<0$ (i.e. $0 \notin B)$. L'intervalle $C=\left[c^{-}, c^{+}\right]$ définir la multiplication $C=A B$ est donné par
$$
c^{-}=\min \left\{a^{-} b^{-}, a^{-} b^{+}, a^{+} b^{-}, a^{+} b^{+}\right\}, \quad c^{+}=\max \left\{a^{-} b^{-}, a^{-} b^{+}, a^{+} b^{-}, a^{+} b^{+}\right\}
$$
et l '«inverse» multiplicatif (ce n'est pas l'inverse au sens algébrique) d'un intervalle $B$ est défini par $B^{-1}=$ $\left[1 / b^{+}, 1 / b^{-}\right]$
\begin{definition}
 Pour $A=\left[a^{-}, a^{+}\right]$ et $B=\left[b^{-}, b^{+}\right]$ on définit la division généralisée (g-division) $\div g$ comme suit:
 \begin{equation}
 A \div_{g} B=C \Longleftrightarrow \left\{
 \begin{array}{ll}
 (i) & A=B C  \texttt{ Or  }\\

 (ii) & B=A C^{-1}.
\end{array}\right.
 \label{25}
 \end{equation}
\end{definition}
Si les deux cas (i) et (ii) sont valides, nous avons $C C^{-1}=C^{-1} C=\{1\}$, i.e. $C=\{\widehat{c}\}, C^{-1}=\{1 / \widehat{c}\}$ avec $\widehat{c} \neq 0$. Il est immédiat de voir que $A \div_{g} B$ existe toujours et est unique pour donné $A=\left[a^{-}, a^{+}\right]$ et $B=\left[b^{-}, b^{+}\right]$ avec $0 \notin B$.\\
Il est facile de voir que les six cas suivants sont possibles, avec les règles indiquées:\\
Case 1: Si $0<a^{-} \leq a^{+}$ et $b^{-} \leq b^{+}<0$ alors
\begin{center}
Si $a^{-} b^{-} \geq a^{+} b^{+}$ alors $c^{-}=\frac{a^{+}}{b^{-}}, c^{+}=\frac{a^{-}}{b^{+}}$ et (i) est satisfait,\\
if $a^{-} b^{-} \leq a^{+} b^{+}$ alors $c^{-}=\frac{a^{-}}{b^{+}}, c^{+}=\frac{a^{+}}{b^{-}}$ et (ii) est satisfait.
\end{center}
Case 2: If $0<a^{-} \leq a^{+}$ et $0<b^{-} \leq b^{+}$ alors
\begin{center}
if $a^{-} b^{+} \leq a^{+} b^{-}$ alors $c^{-}=\frac{a^{-}}{b^{-}}, c^{+}=\frac{a^{+}}{b^{+}}$ et (i) est satisfait,,\\
si $a^{-} b^{+} \geq a^{+} b^{+}$ alors $c^{-}=\frac{a^{+}}{b^{+}}, c^{+}=\frac{a^{-}}{b^{-}}$ et (ii) est satisfait.
\end{center}
Cas 3 : Si $a^{-} \leq a^{+}<0$ et $b^{-} \leq b^{+}<0$ alors
\begin{center}
Si $a^{+} b^{-} \leq a^{-} b^{+}$ alors $c^{-}=\frac{a^{+}}{b^{+}}, c^{+}=\frac{a^{-}}{b^{-}}$ et (i) est satisfait,\\
si $a^{+} b^{-} \geq a^{-} b^{+}$ alors $c^{-}=\frac{a^{-}}{b^{-}}, c^{+}=\frac{a^{+}}{b^{+}}$ et (ii) est satisfait.
\end{center}
Cas 4: Si $a^{-} \leq a^{+}<0$ et $0<b^{-} \leq b^{+}$ alors
\begin{center}
si $a^{-} b^{-} \leq a^{+} b^{+}$ alors $c^{-}=\frac{a^{-}}{b^{+}}, c^{+}=\frac{a^{+}}{b^{-}}$ et (i) est satisfait,\\
si $a^{-} b^{-} \geq a^{+} b^{+}$ alors $c^{-}=\frac{a^{+}}{b^{-}}, c^{+}=\frac{a^{-}}{b^{+}}$ et (ii) est satisfait.
\end{center}
Cas 5 : Si $a^{-} \leq 0, a^{+} \geq 0$ et $b^{-} \leq b^{+}<0$ alors la solution ne dépend pas de $b^{+}$,
\begin{center}
$c^{-}=\frac{a^{+}}{b^{-}}, \quad c^{+}=\frac{a^{-}}{b^{-}} \quad$ et (i) est satisfait.
\end{center}
Cas 6: Si $a^{-} \leq 0, a^{+} \geq 0$ et $0<b^{-} \leq b^{+}$ alors la solution ne dépend pas de $b^{-}$,
\begin{center}
$c^{-}=\frac{a^{-}}{b^{+}}, \quad c^{+}=\frac{a^{+}}{b^{+}} \quad$ et (i) est satisfait.
\end{center}

\begin{remarque}
 Si $0 \in] b^{-}, b^{+}\left[\right.$ la division g n'est pas définie; pour les intervalles $B=\left[0, b^{+}\right]$ ou $B=\left[b^{-}, 0\right]$ la division est possible mais obtenir des résultats illimités $C$ de la forme $C=\left]-\infty, c^{+}\right]$ or $C=\left[c^{-},+\infty\left[:\right.\right.$ nous travaillons avec $B=\left[\varepsilon, b^{+}\right]$ ou $B=\left[b^{-},-\varepsilon\right]$ et nous obtenons le résultat par la limite pour $\varepsilon \longrightarrow 0^{+}$. Example: Pour $[-2,-1] \div g[0,3]$ nous considérons $[-2,-1] \div g[\varepsilon, 3]=\left[c_{\varepsilon}^{-}, c_{\varepsilon}^{+}\right]$ avec (cas 2) $c_{\varepsilon}^{-}=\min \left\{\frac{-2}{3},-1 / \varepsilon\right\}$ et $c_{\varepsilon}^{+}=\max \left\{-2 / \varepsilon, \frac{-1}{3}\right\}$ et obtenez le résultat $C=\left[-\infty,-\frac{1}{3}\right]$ à la limite $\varepsilon \longrightarrow 0^{+} .$
\end{remarque}
Les propriétés suivantes sont immédiates.
\begin{proposition}
 Pour toute $A=\left[a^{-}, a^{+}\right]$ et $B=\left[b^{-}, b^{+}\right]$ avec $0 \notin B,$ nous avons (ici 1 est le même que \{1\} ):\\
1. $B \div_{g} B=1, B \div_{g} B^{-1}=\left\{b^{-} b^{+}\right\}\left(=\left\{\widehat{b}^{2}\right\}\right.$ if $\left.b^{-}=b^{+}=\widehat{b}\right)$\\
2. $(A B) \div_{g}  B=A$.\\
3. $1 \div_{g} B=B^{-1}$ et $1 \div_{g} B^{-1}=B$.\\
4. Au moins une des égalités $B(A \div_{g} B)=A$ or $A(A \div_{g} B)^{-1}=B$ est valide et les deux tiennent si et seulement si $A \div_{g} B$ est un singleton.
\end{proposition}

\subsection{Le cas flou}
La recherche de définitions alternatives de l'opérateur de division entre les nombres flous a reçu une certaine attention dans la littérature récente, avec l'objectif d'inverser la multiplication; \\
Ici, nous suggérons une approche de la division en tant qu'opérateur inverse de multiplication floue, similaire à la différence gH (en tant qu'opérateur inverse d'addition floue).
\begin{definition}
 Soit $u, v \in \mathcal{F}$ avoir $\alpha-coupe [u]_{\alpha}=\left[u_{\alpha}^{-}, u_{\alpha}^{+}\right],[v]_{\alpha}=\left[v_{\alpha}^{-}, v_{\alpha}^{+}\right]$, avec $0 \notin[v]_{\alpha} \forall \alpha \in[0,1]$. la division généralisée
$\div g$ est l'opération qui calcule le nombre flou $w=u \div_{g} v \in \mathcal{F}$ avoir des coupes de niveau $[w]_{\alpha}=\left[w_{\alpha}^{-}, w_{\alpha}^{+}\right]$ Défini par
$$
[u]_{\alpha} \div g[v]_{\alpha}=[w]_{\alpha} \Longleftrightarrow\left\{\begin{array}{l}
\text { (i) }[u]_{\alpha}=[v]_{\alpha}[w]_{\alpha} \\
\text { or }(\mathrm{ii})[v]_{\alpha}=[u]_{\alpha}[w]_{\alpha}^{-1}
\end{array}\right.
$$
à condition que $ w $ soit un nombre flou approprié, où les multiplications entre les intervalles sont effectuées dans le paramètre arithmétique d'intervalle standard.
\end{definition}
La division g floue $\div g$ est bien défini si le $\alpha$ -coupe $[w]_{\alpha}$ sont tels que $w \in \mathcal{F}\left(w_{\alpha}^{-}\right.$ ne diminue pas, $w_{\alpha}^{+}$ n'augmente pas, $\left.w_{1}^{-} \leq w_{1}^{+}\right).$\\
 Clairement, si $u \div g \in \mathcal{F}$ existe, il a les propriétés déjà illustrées pour le cas d'intervalle.
\begin{proposition}
 Soit $u, v \in \mathcal{F}$ (ici 1 est le même que\{1\}$) .$ On a :\\
1. si $0 \notin[u]_{\alpha} \forall \alpha,$ alors $u \div_{g} u=1$\\
2. si $0 \notin[v]_{\alpha} \forall \alpha,$ alors $(u v) \div_{g} v=u ;$\\
3. si $0 \notin[v]_{\alpha} \forall \alpha,$ alors $1 \div g v=v^{-1}$ et $1 \div g v^{-1}=v$;\\
4. si $v \div g$ u existe alors Soit $u(v \div g u)=v$ ou $v(v \div g u)^{-1}=u$ et les deux égalités sont valables si et seulement si $v \div_{g} u$ est croustillant.
\end{proposition}
Dans le cas flou, il est possible que la $g$-division de deux nombres flous n'existe pas. Par exemple, nous pouvons considérer un nombre flou triangulaire $u=(1,1.5,5)$ et $v=(-4,-2,-1)$ au niveau des coupes, les divisions $\mathrm{g}$ existent mais les intervalles résultants ne sont pas les $\alpha$-coupes d'un nombre flou.

Pour résoudre cette lacune, dans \cite{Ref11} une nouvelle division entre les nombres flous a été proposée, une division qui existe toujours.

Nous illustrons la division généralisée $\div_{g}$ avec quelques exemples (les nombres flous sont définis en fonction de leur $\alpha$-coupe).\\
1.$[1+2 \alpha, 7-4 \alpha] \div_{g}[-3+\alpha,-1-\alpha]=[(7-4 \alpha) /(-3+\alpha),(1+2 \alpha) /(-1-\alpha)]$ selon le cas (i) \\
2.$[-3+3 \alpha, 2-2 \alpha] \div_{g}[3+2 \alpha, 8-3 \alpha]=[(-3+3 \alpha) /(8-3 \alpha),(2-2 \alpha) /(8-3 \alpha)]$ selon le cas (i); notez que dans ce cas, le résultat ne dépend pas de $ v_{\alpha}^{-} .$\\
3. $[1+0.5 \alpha, 5-3.5 \alpha] \div g[-4+2 \alpha,-1-\alpha]$ does not exist\\
4. $[-7+2 \alpha,-4-\alpha] \div g[12+5 \alpha,-4-3 \alpha]=[(-7+2 \alpha) /(-12+5 \alpha),(-4-\alpha) /(-4-3 \alpha)]$ selon le cas (ii).\\
5. $[-5+\alpha,-3-\alpha] \div_{g}[4+2 \alpha, 11-5 \alpha]=[(-3-\alpha) /(4+2 \alpha),(-5+\alpha) /(11-5 \alpha)$ ] selon le cas (ii). Notez que pour les exemples 4 et 5 la division $u \oslash v$ n'existe pas.

\subsection{Division généralisé floue approximative}
Si les divisions en g $[u]_{\alpha} \div g[v]_{\alpha}$ ne définissez pas un nombre flou propre, nous pouvons procéder de la même manière que ce qui est fait dans la sous-section $ 2.3.3 $ et obtenir une division floue approchée avec $ \alpha $-coupe
\begin{equation}
[u \stackrel{\sim}\div v]_{\alpha}:=c l\left(\bigcup_{\beta \geq \alpha}\left([u]_{\beta} \div g[v]_{\beta}\right)\right)
\label{26}
\end{equation}
Comme chaque g-division $[u]_{\beta} \div_{g}[v]_{\beta}$ existe pour $\beta \in[0,1], z=u \div v$ peut être considérée comme une généralisation de la division des nombres flous, existant pour tout $u, v$ avec $0 \notin[v]_{\beta}$ pour tout $\beta \in[0,1]$.

Une version discrétisée de $z=u \stackrel{\sim}\div  v $ sur une partition $0=\alpha_{0}<\alpha_{1}<\cdots<\alpha_{N}=1$ de $[0,1]$ est obtenu en utilisant $\left[w_{i}^{-}, w_{i}^{+}\right]=[u]_{\alpha_{i}} \div_{g}[v]_{\alpha_{i}}$ et $ z_{N}^{-} =w_{N}^{-}, \quad z_{N}^{+}=w_{N}^{+} $
$$
\text{For } k =N-1, \ldots, 0:\left\{\begin{array}{l}
z_{k}^{-}=\min \left\{z_{k+1}^{-}, w_{k}^{-}\right\} \\
z_{k}^{+}=\max \left\{z_{k+1}^{+}, w_{k}^{+}\right\}
\end{array}\right.
$$
\begin{exemple}
Pour toutes les valeurs de $\alpha \in[0,1]$ l'intervalle g-divisions

 $[1+0.5 \alpha, 5-3.5 \alpha] \div g[-4+2 \alpha,-1-\alpha]=$ $[(5-3.5 \alpha) /(-4+2 \alpha),(1+0.5 \alpha) /(-1-\alpha)]$ existent mais les intervalles qui en résultent ne sont pas les $\alpha$-coupes d'un nombre flou; appliquer $(\ref{26})$ on obtient le nombre flou $[1+0.5 \alpha, 5-3.5 \alpha] \stackrel{\sim}{\div}[-4+2 \alpha,-1-\alpha]=[(5-3.5 \alpha) /(-4+2 \alpha),-0.75]$.
\end{exemple}
\chapter{La dérivation fractionnaire}
L’idée principale de la dérivation et d’intégration fractionnaire est la généralisation
de la dérivation et d’intégration itérées. Le terme fractionnaire est un terme trompeur
mais il est retenu pour suivre l’usage dominant.
\section{Outil de base}
\subsection{La fonction Gamma et la fonction Béta}
1. La fonction $\Gamma$ d'Euler est une fonction qui prolonge la factorielle aux valeurs réelles et complexes \cite{Ref64}.
Pour $\operatorname{Re}(\alpha)>0$ on définie $\Gamma(\alpha)$ par :
$$
\Gamma(\alpha)=\int_{0}^{+\infty} t^{\alpha-1} e^{-t} d t
$$
La fonction $\Gamma$ s'étend (en une fonction holomorphe) a $\mathbb{C} \backslash Z^{-}$ tout entier.

On a $\Gamma(\alpha+1)=\alpha \Gamma(\alpha)$ et pour $n$ entier on a $n !=\Gamma(n+1)$, pour plus d'informations sur la fonction $\Gamma$ voir \cite{Ref64}.\\

2. La fonction Béta est définie par :
$$
B(p, q)=\int_{0}^{1} \tau^{p-1}(1-\tau)^{q-1} d \tau=\frac{\Gamma(p) \Gamma(q)}{\Gamma(p+q)} \text { avec } \operatorname{Re}(p)>0 \text { et } \operatorname{Re}(q)>0.
$$
\section{Intégration fractionnaire}
Le moyen le plus standard et naturel pour définir les dérivées fractionnaires est au moyen de leur connexion aux intégrales fractionnaires. Il s'avère que les définitions et les propriétés des dérivées fractionnaires dépendent essentiellement des espaces de fonctions, où elles sont définies (un opérateur est un triple $(A, X, Y)$ constitué du domaine $X$, l'intervalle $Y$, et la correspondance $A : X \rightarrow Y$ ). Dans cette section, nous discutons les dérivées fractionnaires connues comme les dérivées fractionnaires de Riemann-Liouville, de Caputo et Hilfer et de Hukuhara.
\begin{definition}
Soit $I^{\alpha}, \alpha \geq 0$ la famille des intégrales fractionnaires de Riemann-Liouville définies par :
$$
\left(I^{\alpha} f\right)(x)=\left\{\begin{array}{ll}
\frac{1}{\Gamma(\alpha)} \int_{0}^{x}(x-t)^{\alpha-1} f(t) d t, & \alpha>0 \\
f(x), & \alpha=0
\end{array}\right.
$$
Une famille à un paramètre $D^{\alpha}, \alpha \geq 0$ des opérateurs linéaires est appelée les dérivées fractionnaires si et seulement si elle satisfait le théorème fondamental du calcul fractionnaire formulé ci-dessous.
\end{definition}
\begin{theorem}(Théorème fondamental de CF).\\
 Pour les dérivées fractionnaires $D^{\alpha}, \alpha \geq 0$  et les intégrales fractionnaires de Riemann-Liouville $I^{\alpha}, \alpha \geq 0$, la relation
\begin{equation}
\left(D^{\alpha} I^{\alpha} \phi\right)(x)=\phi(x), x \in[0,1]
\label{w1}
\end{equation}
est vrai sur les espaces de fonctions non triviaux appropriés.
\end{theorem}
En fait, le théorème fondamental de CF fait partie de la définition $21$, c'est-à-dire qu'un opérateur linéaire est appelé une dérivée fractionnaire si et seulement si c'est un opérateur inverse à gauche de l'intégrale fractionnaire de Riemann-Liouville sur un certain espace de fonctions. Il s'avère qu'il existe une infinité de familles différentes de dérivées fractionnaires au sens de la définition $21$ Dans la suite de cette section, nous discutons de quelques familles connues et nouvelles de dérivées fractionnaires sur l'intervalle $[0,1]$.
\begin{remarque}
En calcul, la formule de type (\ref{w1}) avec $\alpha=1$ est généralement appelée le 1er théorème fondamental du calcul. Le 2e théorème fondamental du calcul stipule que
$$\int_{0}^{x} f^{\prime}(t) d t=f(x)-f(0) .$$
\end{remarque}
\begin{remarque}
Il convient de mentionner que la formule (\ref{w1}) et la relation $\left(I^{0} f\right)(x)=$ $f(x)$ définissent de manière unique la dérivée fractionnaire d'ordre $\ alpha=0$ comme opérateur d'identité :
$$\left(D^{0} f\right)(x)=f(x) .$$
 Par conséquent, dans ce qui suit, nous restreignons principalement notre attention au cas $\alpha>0$.
\end{remarque}
\begin{remarque}
L'intégrale fractionnaire de Riemann-Liouville $I^{\alpha}$ est injective sur $L_{1}(0,1)$, c'est-à-dire que son noyau ne contient que la fonction nulle $f(x)=0$ a.e. sur $[0,1]$. Évidemment, cette affirmation est vraie pour tout opérateur linéaire qui possède un opérateur linéaire inverse à gauche. Pour $I^{\alpha}$, cela résulte de la réalisation du théorème fondamental de CF pour les dérivées fractionnaires de Riemann-Liouville.
\end{remarque}
\begin{remarque}
Par souci de simplicité des formulations, dans ce qui suit, nous nous restreignons à l'espace des fonctions $L_{1}(0,1)$ et ses sous-espaces (une théorie similaire peut être développée pour, disons, $L_{p}(0 ,1), 1<p<+\infty$ et ses sous-espaces) et aux ordres $\alpha, 0<\alpha \leq 1$ des dérivées fractionnaires (le cas $n-1<\alpha \leq n \ in \mathbb{N}$ peut être couvert par analogie à la théorie connue des dérivées fractionnaires de Riemann-Liouville d'ordre $\left.\alpha, \alpha \in \mathbb{R}_{+}\right)$.
\end{remarque}
\section{Les approches des dérivées fractionnaires}
Il y a beaucoup d’approches pour la dérivation fractionnaire, nous allons souligner les
approches qui sont fréquemment utilisées dans les applications.
\subsection{Approche de Riemann-Liouville}
 Dans cette sous-section, certains résultats connus sont présentés sous une forme légèrement différente adaptée à nos constructions ultérieures.

On part de la formule (\ref{w1}) et on la réécrit sous la forme équivalente de deux équations :
\begin{equation}
\left(I^{\alpha} \phi\right)(x)=f(x), \quad\left(D^{\alpha} f\right)(x)=\phi(x), x \in[0,1].
\label{w2}
\end{equation}
La seconde des équations de (\ref{w2}) définit la dérivée fractionnaire $D^{\alpha}$ d'une fonction $f$ comme la solution $\phi$ de l'équation intégrale d'Abel avec le membre de droite $f .$
\begin{theorem}
Sur l'espace des fonctions $X^{0}=I^{\alpha}\left(L_{1}(0,1)\right)$, la dérivée fractionnaire unique $D^{\alpha}$ d'ordre $\alpha, 0<\alpha<1$ est donné par la formule
\begin{equation}
\left(D^{\alpha} f\right)(x)=\frac{d}{d x}\left(I^{1-\alpha} f\right)(x)=\frac{d}{d x} \frac{1}{\Gamma(1-\alpha)} \int_{0}^{x}(x-t)^{-\alpha} f(t) d t .
\label{w3}
\end{equation}
Alors l'équation intégrale d'Abel (la première formule dans (\ref{w2})) a une solution unique donnée par la formule $(\ref{w3})$.

La première partie de la formule (\ref{w3}) peut être utilisée pour définir la dérivée fractionnaire $D^{\alpha}$ de l'ordre $\alpha=1$ comme dérivée du premier ordre :
\begin{equation}
\left(D^{1} f\right)(x)=\frac{d}{d x}\left(I^{0} f\right)(x)=\frac{d f}{d x}
\label{w4}
\end{equation}
Dans ce qui suit, on se réfère à l'opérateur $D^{\alpha}=\frac{d}{dx} I^{1-\alpha} : X^{0} \rightarrow L_{1}(0,1) $ avec $0<\alpha \leq 1$ quand à la dérivée fractionnaire de base de Riemann-Liouville d'ordre $\alpha$ (le terme "basique" fait référence au domaine $X^{0}$ de $D^{\alpha} $ ).
\end{theorem}
\begin{remarque}
La dérivée fractionnaire de base de Riemann-Liouville $D^{\alpha} : X^{0} \rightarrow L_{1}(0,1)$ est une application binivoque de $X^{0}$ sur $ L_{1}(0,1) .$ Pour $0<\alpha<1$, cela peut être facilement vérifié directement.
\end{remarque}
Continuons maintenant à définir la dérivée fractionnaire de Riemann-Liouville sur l'intervalle $[0,1] .$ Évidemment, la formule (\ref{w3}) a un sens pour un espace de fonctions plus grand que $X^{0}$ , à savoir pour l'espace
\begin{equation}
X_{R L}^{1}=\left\{f: I^{1-\alpha} f \in \mathrm{AC}([0,1])\right\}
\label{w5}
\end{equation}
En effet, pour $f \in X_{R L}^{1}$, la représentation
\begin{equation}
\left(I^{1-\alpha} f\right)(x)=\left(I^{1-\alpha} f\right)(0)+\int_{0}^{x} \phi(t) d t, x \in[0,1], \phi \in L_{1}(0,1)
\label{w6}
\end{equation}
est vrai, et donc
\begin{equation}
\left(D^{\alpha} f\right)(x)=\frac{d}{d x}\left(I^{1-\alpha} f\right)(x)=\phi(x), x \in[0,1]
\label{w7}
\end{equation}
\begin{definition}
L'extension de la dérivée fractionnaire de base de Riemann-Liouville $D^{\alpha}: X^{0} \rightarrow L_{1}(0,1)$ au domaine $X_{R L}^{1}$ est appelée la dérivée fractionnaire de Riemann-Liouville d'ordre  $\alpha, 0<\alpha \leq 1:$
\begin{equation}
\left(D_{R L}^{\alpha} f\right)(x)=\frac{d}{d x}\left(I^{1-\alpha} f\right)(x), D_{R L}^{\alpha}: X_{R L}^{1} \rightarrow L_{1}(0,1)
\label{w8}
\end{equation}
\end{definition}
Contrairement à la dérivée fractionnaire de base de Riemann-Liouville, la dérivée fractionnaire de RiemannLiouville n'est pas injective et son noyau est un espace vectoriel à une dimension :
\begin{equation}
\operatorname{Ker}\left(D_{R L}^{\alpha}\right)=\left\{c_{1} x^{\alpha-1}, c_{1} \in \mathbb{R}\right\}
\label{w9}
\end{equation}
Selon la formule de l'intégrale fractionnaire de Riemann-Liouville d'une fonction de loi de puissance suivante :
$$
\left(I^{\alpha} t^{\beta}\right)(x)=\frac{\Gamma(\beta+1)}{\Gamma(\alpha+\beta+1)} x^{\alpha+\beta}, \alpha \geq 0, \beta>-1
$$
et Remarque 14, Il convient de mentionner que la fonction de base $f_{1}(x)=x^{\alpha-1}$ n'appartient pas à l'espace $X^{0}$ car $\left(I^{1-\alpha} t^{\alpha-1}\right)(x) \equiv \Gamma(\alpha) \forall x \in[0,1]$ qui contredit la condition $\left(I^{1-\alpha} f\right)(0)=0$
qui est remplie pour tout $f \in X^{0}$. Sinon, nous avons les inclusions
\begin{equation}
X^{0} \subset X_{R L}^{1}, \mathrm{AC}([0,1]) \subset X_{R L}^{1}
\label{w10}
\end{equation}
Pour les autres propriétés de la dérivée fractionnaire de Riemann-Liouville introduite ci-dessus, nous renvoyons les lecteurs à $\cite{Ref12}$.

Ici, nous mentionnons juste que pour la dérivée fractionnaire de Riemann-Liouville, le théorème fondamental de CF (théorème (5) est évidemment valable sur un espace encore plus grand des fonctions $X_{RL}^{2}=L_{1}(0, 1), X_{RL}^{1} \subset X_{RL}^{2}$, c'est-à-dire la formule
\begin{equation}
\left(D_{R L}^{\alpha} I^{\alpha} f\right)(x)=f(x), x \in[0,1], f \in X_{R L}^{2}
\label{w11}
\end{equation}
qui est vrai.

\begin{exemple}{\color{blue}{La dérivée non entière d’une fonction constante au sens de Riemann-Liouville.}}

En générale la dérivée non entière d’une fonction constante au sens de RiemannLiouville n’est pas nulle ni constante, mais on a :
$$
{ }^{R} D^{p} C=\frac{C}{\Gamma(1-p)}(t-a)^{-p}.
$$
\end{exemple}
\begin{exemple}{\color{blue}{La dérivée de $f(t)=(t-a)^{\alpha}$ au sens de Riemann-Liouville.}}\\
Soit $p$ non entier et $0 \leq n-1<p<n$ et $\alpha>-1$, alors on a :
$$
{ }^{R} D^{p}(t-a)^{\alpha}=\frac{1}{\Gamma(n-p)} \frac{d^{n}}{d t^{n}} \int_{a}^{t}(t-\tau)^{n-p-1}(\tau-a)^{\alpha} d \tau
$$
En faisant le changement de variable $\tau=a+s(t-a)$, on aura :
$$
\begin{aligned}
{ }^{R} D^{p}(t-a)^{\alpha} &=\frac{1}{\Gamma(n-p)} \frac{d^{n}}{d t^{n}}(t-a)^{n+\alpha-p} \int_{a}^{t}(1-s)^{n-p-1} s^{\alpha} d s \\
&=\frac{\Gamma(n+\alpha-p+1) B(n-p, \alpha+1)}{\Gamma(n-p)}(t-\alpha)^{\alpha-p} \\
&=\frac{\Gamma(n+\alpha-p+1) \Gamma(n-p) \Gamma(\alpha+1)}{\Gamma(n-p) \Gamma(\alpha-p+1) \Gamma(n+\alpha-p+1)}(t-\alpha)^{\alpha-p} \\
&=\frac{\Gamma(\alpha+1)}{\Gamma(\alpha-p+1)}(t-\alpha)^{\alpha-p} .
\end{aligned}
$$
A titre d’exemple
$$
{ }^{R} D^{0.5} t^{0.5}=\frac{\Gamma(1.5)}{\Gamma(1)}=\Gamma(1.5).
$$
\end{exemple}

\subsection{Approche de Caputo}
Comme indiqué dans le théorème $6$, la dérivée fractionnaire de base de Riemann-Liouville est l'unique dérivée fractionnaire à un paramètre sur l'espace des fonctions $X^{0}=I^{\alpha}\left(L_{1}(0,1 )\right) .$ Son extension au plus grand espace $X_{RL}^{1}$ qui conduit à la dérivée fractionnaire de Riemann-Liouville est également unique. car en mathématiques on travaille généralement avec les domaines maximaux et les extensions de formules valables sur ces domaines, du point de vue mathématique, les dérivées fractionnaires de RiemannLiouville peuvent être considérées comme la seule famille "correcte" à un paramètre de les dérivées fractionnaires définies sur un intervalle fini. En effet, dans la littérature mathématique classique  et plusieurs centaines de références, principalement cette dérivée et certaines de ses modifications ont été considérées sur un intervalle fini. Où est alors la place de la dérivée fractionnaire de Caputo ?

L'astuce avec sa définition est qu'il faut d'abord contracter l'espace de base $X^{0}$ et introduire un espace de fonctions, où la dérivée du premier ordre commute avec l'intégrale fractionnaire d'ordre de Riemann-Liouville $1-\alpha(0<\alpha \leq 1)$ :
\begin{equation}
X_{C}^{0}=\left\{f \in X^{0}: \frac{d}{d x} I^{1-\alpha} f=I^{1-\alpha} \frac{d f}{d x}\right\}
\label{w12}
\end{equation}
En particulier, l'espace $X_{C}^{0}$ contient les fonctions $f \in \mathrm{AC}([0,1])$ qui vérifient la condition $f(0)=0 .$ En effet, ces fonctions peuvent être représentées sous la forme
$$
f(x)=\int_{0}^{x} \phi(t) d t, x \in[0,1], \phi \in L_{1}(0,1) .
$$
On a alors la chaîne de relations suivante $(0 \leq \alpha)$ :
$$
\begin{gathered}
\left(I^{\alpha} \frac{d f}{d x}\right)(x)=\left(I^{\alpha} \frac{d}{d x} I^{1} \phi\right)(x)=\left(I^{\alpha} \phi\right)(x)= \\
\frac{d}{d x}\left(I^{1} I^{\alpha} \phi\right)(x)=\frac{d}{d x}\left(I^{\alpha} I^{1} \phi\right)(x)=\frac{d}{d x}\left(I^{\alpha} f\right)(x) .
\end{gathered}
$$
La dérivée fractionnaire de base de Caputo d'ordre $\alpha, 0<\alpha \leq 1$, est introduite comme suit :
\begin{equation}
\left(D_{C}^{\alpha} f\right)(x)=\left(I^{1-\alpha} \frac{d f}{d x}\right)(x), D_{C}^{\alpha}: X_{C}^{0} \rightarrow L_{1}(0,1)
\label{w13}
\end{equation}
Bien sûr, la dérivée fractionnaire de base de Caputo est identique à la dérivée fractionnaire de base de RiemannLiouville restreinte au domaine $X_{C}^{0}$ et donc ce n'est pas nouveau. Cependant, nous obtenons un nouvel opérateur par extension de son domaine ! L'opérateur (\ref{w13}) est bien défini, disons, sur l'espace $X_{C}^{1}=\mathrm{AC}([0,1])$
\begin{definition}
L'extension de la dérivée fractionnaire de base de Caputo $D_{C}^{\alpha} : X_{C}^{0} \rightarrow$ $L_{1}(0,1)$ au domaine $X_{C}^ {1}$ est appelé la dérivée fractionnaire de Caputo de $\alpha$, $0<\alpha \leq 1:$
\begin{equation}
\left(D_{C}^{\alpha} f\right)(x)=\left(I^{1-\alpha} \frac{d}{d x} f\right)(x), D_{C}^{\alpha}: X_{C}^{1} \rightarrow L_{1}(0,1)
\label{w14}
\end{equation}
\end{definition}
Évidemment, le noyau de la dérivée fractionnaire de Caputo coïncide avec le noyau de la dérivée du premier ordre :
\begin{equation}
\operatorname{Ker}\left(D_{C}^{\alpha}\right)=\left\{c_{1}, c_{1} \in \mathbb{R}\right\}
\label{w15}
\end{equation}
Pour les fonctions de $X_{C}^{1}$, il existe une connexion simple entre les dérivées fractionnaires de RiemannLiouville et de Caputo :
\begin{equation}
\left(D_{C}^{\alpha} f\right)(x)=\left(D_{R L}^{\alpha} f\right)(x)-\frac{f(0)}{\Gamma(1-\alpha)} x^{-\alpha}, x>0, f \in X_{C}^{1}
\label{w16}
\end{equation}
Comme dans le cas de la dérivée fractionnaire de Riemann-Liouville, le théorème fondamental de CF pour la dérivée fractionnaire de Caputo est valable sur un espace de fonctions encore plus grand
\begin{equation}
X_{F T}=\left\{f: I^{\alpha} f \in \mathrm{AC}([0,1]) \text { and }\left(I^{\alpha} f\right)(0)=0\right\},
\label{w17}
\end{equation}
c'est-à-dire la formule
\begin{equation}
\left(D_{C}^{\alpha} I^{\alpha} f\right)(x)=f(x), x \in[0,1], f \in X_{F T}
\label{w18}
\end{equation}
qui est vrai. Prouvons-le, pour une fonction $f$ de $X_{F T}$, la représentation
\begin{equation}
\left(I^{\alpha} f\right)(x)=\left(I^{1} \phi\right)(x), x \in[0,1]
\label{w19}
\end{equation}
est vrai avec une fonction $\phi \in L_{1}(0,1)$. On obtient alors la chaîne d'égalités suivante :
$$
\left(D_{C}^{\alpha} I^{\alpha} f\right)(x)=\left(I^{1-\alpha} \frac{d}{d x} I^{\alpha} f\right)(x)=\left(I^{1-\alpha} \frac{d}{d x} I^{1} \phi\right)(x)=\left(I^{1-\alpha} \phi\right)(x)
$$
Car pour $\phi \in L_{1}(0,1)$, l'intégrale fractionnaire $I^{1-\alpha} \phi$ appartient aussi à $L_{1}(0,1)$, on peut appliquer l'opérateur $I^{\alpha}$ à la dernière formule :
$$
\left(I^{\alpha}\left(D_{C}^{\alpha} I^{\alpha} f\right)\right)(x)=\left(I^{\alpha} I^{1-\alpha} \phi\right)(x)=\left(I^{1} \phi\right)(x)=\left(I^{\alpha} f\right)(x)
$$
L'intégrale fractionnaire de Riemann-Liouville est injective (Remarque $5.3$ ) et donc la formule (\ref{w18}) découle de la dernière relation.

Il est à noter que l'espace $X_{F T}$ défini par (\ref{w17}) peut également être caractérisé comme suit :
\begin{equation}
X_{F T}=I^{1-\alpha}\left(L_{1}(0,1)\right)\left(\forall f \in X_{F T} \exists \phi \in L_{1}(0,1): f(x)=\left(I^{1-\alpha} \phi\right)(x)\right)
\label{w20}
\end{equation}
\paragraph*{• Relation avec la dérivée de Riemann-Liouville\\}
Soit $p>0$ avec $n-1<p<n,\left(n \in \mathbb{N}^{*}\right)$, supposons que $f$ est une fonction telle que ${ }_{a}^{C} D_{t}^{p} f(t)$ et ${ }_{a}^{R} D_{t}^{p} f(t)$ existent alors 
$$
{ }^{C} D^{p} f(t)={ }^{R} D^{p} f(t)-\sum_{k=0}^{n-1} \frac{f^{(k)}(a)(t-a)^{k-p}}{\Gamma(k-p+1)}
$$
On déduit que si $f^{(k)}(a)=0$ pour $k=0,1,2, \ldots \ldots, n-1$, on aura
 $${ }^{C} D^{p} f(t)={ }^{R} D^{p} f(t).$$
\begin{exemple}{\color{blue}{La dérivée d’une fonction constante au sens de Caputo.}}\\
La dérivée d’une fonction constante au sens de Caputo est nulle
$$
{ }^{C} D^{p} C=0
$$
\end{exemple}
\begin{exemple}{\color{blue}{La dérivée de $f(t)=(t-a)^{\alpha}$ au sens de Caputo.}}\\
Soit $p$ un entier et $0 \leq n-1<p<n$ avec $\alpha>n-1$, alors on a
$$
f^{(n)}(\tau)=\frac{\Gamma(\alpha+1)}{\Gamma(\alpha-n+1)}(\tau-a)^{\alpha-n}
$$
d'où
$$
{ }^{C} D^{p}(t-a)^{\alpha}=\frac{\Gamma(\alpha+1)}{\Gamma(n-p) \Gamma(\alpha-n+1)} \int_{a}^{t}(t-\tau)^{n-p-1}(\tau-a)^{\alpha-n} d \tau
$$
effectuant le changement de variable $\tau=a+s(t-a)$ on obtient
$$
\begin{aligned}
{ }^{C} D^{p}(t-a)^{\alpha} &=\frac{\Gamma(\alpha+1)}{\Gamma(n-p) \Gamma(\alpha-n+1)} \int_{a}^{t}(t-\tau)^{n-p-1}(\tau-a)^{\alpha-n} d \tau \\
&=\frac{\Gamma(\alpha+1)}{\Gamma(n-p) \Gamma(\alpha-n+1)}(t-a)^{\alpha-p} \int_{a}^{1}(1-s)^{n-p-1} s^{\alpha-n} d s \\
&=\frac{\Gamma(\alpha+1) B(n-p, \alpha-n+1)}{\Gamma(n-p) \Gamma(\alpha-n+1)}(t-a)^{\alpha-p} \\
&=\frac{\Gamma(\alpha+1) \Gamma(n-p) \Gamma(\alpha-n+1)}{\Gamma(n-p) \Gamma(\alpha-n+1) \Gamma(\alpha-p+1)}(t-a)^{\alpha-p} \\
&=\frac{\Gamma(\alpha+1)}{\Gamma(\alpha-p+1)}(t-a)^{\alpha-p} .
\end{aligned}
$$
\end{exemple}
Pour les autres propriétés de la dérivée fractionnaire de Caputo, nous renvoyons les lecteurs à $\cite{Ref13}$.
\subsection{Approche de Hilfer}
La troisième famille connue des dérivées fractionnaires définies sur un intervalle fini qui remplit le théorème fondamental de CF est la famille des dérivées fractionnaires généralisées de Riemann-Liouville. Ils ont été introduits par Hilfer dans \cite{Ref14} et sont aujourd'hui appelés dérivés fractionnaires de Hilfer.

Le schéma de construction de la dérivée fractionnaire de Hilfer d'ordre $\alpha, 0<$ $\alpha \leq 1$ sur un intervalle fini est le même que celui employé pour la dérivée fractionnaire de Caputo. Nous commençons par définir un espace de base approprié de fonctions, où la dérivée du premier ordre commute avec une certaine intégrale fractionnaire de Riemann-Liouville. Soit un paramètre $\gamma_{1} \in \mathbb{R}$ satisfaire les conditions :
\begin{equation}
0 \leq \gamma_{1} \leq 1-\alpha
\label{w21}
\end{equation}
L'espace des fonctions pour la dérivée fractionnaire de Hilfer de base est défini comme suit :
\begin{equation}
X_{H}^{0}=\left\{f \in X^{0}: \frac{d}{d x} I^{\gamma} f=I^{\gamma_{1}} \frac{d f}{d x}\right\}
\label{w22}
\end{equation}
Comme dans le cas de l'espace $X_{C}^{0}$ pour la dérivée fractionnaire de Caputo, l'espace $X_{H}^{0}$ contient notamment les fonctions $f \in \mathrm{AC}( [0,1])$ qui satisfont la condition $f(0)=0$
La dérivée fractionnaire de Hilfer de base d'ordre $\alpha, 0<\alpha \leq 1$ et de type $\gamma_{1}, 0 \leq \gamma_{1} \leq$ $1-\alpha$ est introduite comme suit :
\begin{equation}
\left(D_{H}^{\alpha, \gamma_{1}} f\right)(x)=\left(I^{\gamma_{1}} \frac{d}{d x} I^{1-\alpha-\gamma_{1}} f\right)(x), D_{H}^{\alpha, \gamma_{1}}: X_{H}^{0} \rightarrow L_{1}(0,1)
\label{w23}
\end{equation}
Sur l'espace $X_{H}^{0}$, la dérivée fractionnaire de base de Hilfer est identique à la dérivée fractionnaire de base de Riemann-Liouville restreinte au domaine $X_{H}^{0}$ :
$$
\begin{gathered}
\left(D_{H}^{\alpha, \gamma_{1}} f\right)(x)=\left(I^{\gamma_{1}} \frac{d}{d x} I^{1-\alpha-\gamma_{1}} f\right)(x)=\frac{d}{d x}\left(I^{\gamma_{1}} I^{1-a-\gamma_{1}} f\right)(x)= \\
\frac{d}{d x}\left(I^{1-\alpha} f\right)(x)=\left(D_{R L}^{\alpha} f\right)(x), f \in X_{H}^{0}
\end{gathered}
$$
Cependant, le domaine de la dérivée fractionnaire de Hilfer de base peut être étendu au plus grand espace des fonctions :
\begin{equation}
X_{H}^{1}=\left\{f: I^{1-\alpha-\gamma} f \in \mathrm{AC}([0,1])\right\}
\label{w24}
\end{equation}
\begin{definition}
 L'extension de la dérivée fractionnaire de base de Hilfer $D_{H}^{\alpha} : X_{H}^{0} \rightarrow$ $L_{1}(0,1)$ au domaine $X_{H}^ {1}$ est appelé la dérivée fractionnaire de Hilfer de $\alpha, 0<$ $\alpha \leq 1$ et tapez $\gamma_{1}, 0 \leq \gamma_{1} \leq 1-\alpha : $
\begin{equation}
\left(D_{H}^{\alpha, \gamma_{1}} f\right)(x)=\left(I^{\gamma_{1}} \frac{d}{d x} I^{1-\alpha-\gamma_{1}} f\right)(x), D_{H}^{\alpha, \gamma_{1}}: X_{H}^{1} \rightarrow L_{1}(0,1)
\label{w25}
\end{equation}
\end{definition}

Le théorème fondamental de CF (théorème 5) pour la dérivée fractionnaire de Hilfer est valable sur l'espace $X_{F T}$ défini par (\ref{w17}) :
\begin{equation}
\left(D_{H}^{\alpha, \gamma_{1}} I^{\alpha} f\right)(x)=f(x), x \in[0,1], f \in X_{F T}, 0<\alpha \leq 1,0 \leq \gamma_{1} \leq 1-\alpha
\label{w26}
\end{equation}
Sa preuve suit les étapes de la preuve de la formule (\ref{w18}) pour la dérivée fractionnaire de Caputo. Nous commençons par la représentation (\ref{w19}) et la substituons dans la partie gauche de la formule (\ref{w26}) :
$$
\begin{gathered}
\left(D_{H}^{\alpha, \gamma_{1}} I^{\alpha} f\right)(x)=\left(I^{\Upsilon} \frac{d}{d x} I^{1-\alpha-\gamma} I^{\alpha} f\right)(x)=\left(I^{\gamma_{1}} \frac{d}{d x} I^{1-\alpha-\gamma_{1}} I^{1} \phi\right)(x)= \\
\left(I^{\gamma_{1}} \frac{d}{d x} I^{1} I^{1-a-\gamma_{1}} \phi\right)(x)=\left(I^{\gamma} I^{1-\alpha-\gamma} \phi\right)(x)=\left(I^{1-\alpha} \phi\right)(x)
\end{gathered}
$$
Le reste de la preuve est exactement le même que la preuve de la formule (\ref{w18}) pour la dérivée fractionnaire de Caputo que nous avons présentée dans la sous-section précédente.
\begin{remarque} 
Pour chaque type $\gamma_{1}, 0 \leq \gamma_{1} \leq 1-\alpha$, les dérivés de Hilfer $D_{H}^{\alpha, \gamma_{1}}$ des ordres $\ alpha, 0<\alpha \leq 1$ forment les familles à un paramètre des dérivées fractionnaires avec les dérivées fractionnaires de Riemann-Liouville, tandis que pour $\gamma_{1}=1-\alpha$ on obtient les dérivées fractionnaires de Caputo.
\end{remarque}
Le noyau de la dérivée fractionnaire de Hilfer peut être facilement calculé et on a :
\begin{equation}
\operatorname{Ker}\left(D_{H}^{\alpha, \gamma_{1}}\right)=\left\{c_{1} x^{\alpha+\gamma_{1}-1}, c_{1} \in \mathbb{R}\right\}
\label{w27}
\end{equation}
Pour les autres propriétés de la dérivée fractionnaire de Hilfer, nous renvoyons les lecteurs à $\cite{Ref14},\cite{Ref15}$
\subsection{Approche de Grünwald-Letnikov}
L’idée est de généraliser la définition classique de la dérivation entière d’une fonction à des ordres de dérivée arbitraires, donc on peut exprimer la dérivée d’ordre entier p (si p est positif ) et l’intégrale répétée (p) fois (si p est négatif) d’une fonction f par la formule suivante:
\begin{equation}
D^{p} f(t)=\lim _{h \rightarrow 0} h^{-p} \sum_{k=0}^{n}(-1)^{k}\left(\begin{array}{l}
p \\
k
\end{array}\right) f(t-k h), \text { avec }\left(\begin{array}{l}
p \\
k
\end{array}\right)=\frac{p(p-1) \ldots(p-k+1)}{k !}
\label{r1}
\end{equation}
La généralisation de cette formule pour $p$ non entier (avec $0 \leq n-1<p<n$ ) et comme :
\begin{equation}
\begin{aligned}
(-1)^{k}\left(\begin{array}{l}
p \\
k
\end{array}\right) &=\frac{-p(1-p) \ldots(k-p-1)}{k !} \\
&=\frac{\Gamma(k-p)}{\Gamma(k+1) \Gamma(-p)}
\end{aligned}
\label{r2}
\end{equation}
nous obtenons
\begin{equation}
{ }^{G} D^{p} f(t)=\lim _{h \rightarrow 0} h^{-p} \sum_{k=0}^{n} \frac{\Gamma(k-p)}{\Gamma(k+1) \Gamma(-p)} f(t-k h)
\label{r3}
\end{equation}
et
\begin{equation}
{ }^{G} D^{-p} f(t)=\lim _{h \rightarrow 0} h^{p} \sum_{k=0}^{n} \frac{\Gamma(k+p)}{\Gamma(k+1) \Gamma(p)} f(t-k h)
\label{r4}
\end{equation}
Si $f$ est de classe $C^{n}$, alors en utilisant l'intégration par parties on obtient:
\begin{equation}
{ }^{G} D^{-p} f(t)=\sum_{k=0}^{n-1} \frac{f^{(k)}(a)(t-a)^{k+p}}{\Gamma(k+p+1)}+\frac{1}{\Gamma(n+p)} \int_{a}^{t}(t-\tau)^{n+p-1} f^{(n)}(\tau) d \tau
\label{r5}
\end{equation}
aussi
\begin{equation}
{ }^{G} D^{p} f(t)=\sum_{k=0}^{n-1} \frac{f^{(k)}(a)(t-a)^{k-p}}{\Gamma(k-p+1)}+\frac{1}{\Gamma(n-p)} \int_{a}^{t}(t-\tau)^{n-p-1} f^{(n)}(\tau) d \tau
\label{r6}
\end{equation}
\begin{exemple}{\color{blue}{La dérivée d’une fonction constante au sens de Grünwald-Letnikov.}}\\
 En générale la dérivée d’une fonction constante au sens de Grünwald-Letnikov n’est
pas nulle ni constante.\\
Si $f(t)=c$ et $p$ non entier positif on a :
$$
f^{(k)}(t)=0 \text { pour } k=1,2, \ldots, n
$$
\begin{equation}
\begin{aligned}
{ }^{G} D^{p} f(t)=& \frac{c}{\Gamma(1-p)}(t-a)^{-p} \\
&+\sum_{k=1}^{n-1} \frac{f^{(k)}(a)(t-a)^{k-p}}{\Gamma(k-p+1)} \\
&+\frac{1}{\Gamma(n-p)} \int_{a}^{t}(t-\tau)^{n-p-1} f^{(n)}(\tau) d \tau \\
=& \frac{c}{\Gamma(1-p)}(t-a)^{-p} .
\end{aligned}
\label{r7}
\end{equation}
\end{exemple}
\begin{exemple}{\color{blue}{La dérivée de $f(t)=(t-a)^{\alpha}$ au sens de Grünwald-Letnikov.}}\\
Soit $p$ non entier et $0 \leq n-1<p<n$ avec $\alpha>n-1$ alors on a :
$$
f^{(k)}(a)=0 \text { pour } k=0,1, \ldots, n-1,
$$
et
$$
f^{(n)}(\tau)=\frac{\Gamma(\alpha+1)}{\Gamma(\alpha-n+1)}(\tau-a)^{\alpha-n}
$$
d'où
$$
{ }^{G} D^{p} f(t)=\frac{\Gamma(\alpha+1)}{\Gamma(n-p) \Gamma(\alpha-n+1)} \int_{a}^{t}(t-\tau)^{n-p-1}(\tau-a)^{-p} d \tau .
$$
En faisant le changement de variable $\tau=a+s(t-a)$ on trouve :
$$
\begin{aligned}
{ }^{G} D^{p}(t-a)^{\alpha} &=\frac{\Gamma(\alpha+1)}{\Gamma(n-p) \Gamma(\alpha-n+1)} \int_{a}^{t}(t-\tau)^{n-p-1}(\tau-a)^{\alpha-n} d \tau \\
&=\frac{\Gamma(\alpha+1)}{\Gamma(n-p) \Gamma(\alpha-n+1)}(t-a)^{\alpha-p} \int_{a}^{1}(1-s)^{n-p-1} s^{\alpha-n} d s \\
&=\frac{\Gamma(\alpha+1) B(\alpha-p, \alpha-n+1)}{\Gamma(n-p) \Gamma(\alpha-n+1)}(t-a)^{\alpha-p} \\
&=\frac{\Gamma(\alpha+1) \Gamma(n-p) \Gamma(\alpha-n+1)}{\Gamma(n-p) \Gamma(\alpha-n+1) \Gamma(\alpha-p+1)}(t-a)^{\alpha-p} \\
&=\frac{\Gamma(\alpha+1)}{\Gamma(\alpha-n+1)}(t-a)^{\alpha-p} .
\end{aligned}
$$
A titre d'exemple
$$
{ }^{G} D^{1 / 2} t=\frac{\Gamma(2)}{\Gamma(1.5)} \sqrt{t}=\frac{\sqrt{t}}{\Gamma(1.5)}
$$

\end{exemple}
\subsection{Approche de Hukuhara}
Rappelez-vous qu'un sous-ensemble flou de $\mathbb{R}$ est défini en termes d'une fonction d'appartenance qui assigne à chaque point $x \in \mathbb{R}$ un grade d'appartenance à l'ensemble flou. Une telle fonction d'appartenance :
$$
u: \mathbb{R} \rightarrow I=[0,1]
$$
est utilisé pour désigner l'ensemble flou correspondant. Dénoter par $\mathbb{F}$ l'ensemble de tous les ensembles flous de $\mathbb{R}$.
\subsubsection{Définitions et propriétés}
\begin{definition}\cite{Ref7}\\
 Soit $ x_{0} \in (a, b) $ et $ h $ tels que $ x_{0} + h \in (a, b) $, puis la dérivée généralisée de Hukuhara d'une fonction de valeur floue $ f: (a, b) \rightarrow E^{1} $ en $ x_{0} $ est défini comme
\begin{equation}
\lim _{h \rightarrow 0}\left\|\frac{f\left(x_{0}+h\right)-_{g} f\left(x_{0}\right)}{h}-_{g} f_{g}^{\prime}\left(x_{0}\right)\right\|=0
\label{1}
\end{equation}
Si $ f^{\prime}_{g} \left (x_{0} \right) \in E^{1} $ satisfaisant (\ref{1}) existe, on dit que $ f $ est généralisé Hukuhara différentiable (g-différentiable pour faire court) à $ x_{0} $.
\end{definition}
\begin{definition}\cite{Ref7}\\
 Soit $f:[a, b] \rightarrow E^{1}$ et $x_{0} \in(a, b)$, avec $\underline{f}(x, \alpha)$ et $\bar{f}(x, \alpha)$ tous deux différentiables à $ x_{0} $. Nous disons que:\\
1) $f$ est $[(i)-g]$-differentiable en $x_{0}$ si
\begin{equation}
f_{i, g}^{\prime}\left(x_{0}\right)=\left[\underline{f}^{\prime}(x, \alpha), \bar{f}^{\prime}(x, \alpha)\right]
\label{2}
\end{equation}
2) $f$ est $[(i i)-g]$-differentiable en $x_{0}$ si
\begin{equation}
f_{i i, g}^{\prime}\left(x_{0}\right)=\left[\bar{f}^{\prime}(x, \alpha), \underline{f}^{\prime}(x, \alpha)\right]
\label{3}
\end{equation}

\end{definition}
\begin{theorem}
 Soit $f: J \subset \mathbb{R} \rightarrow E^{1}$ et $g: J \rightarrow \mathbb{R}$ et $x \in J .$ Supposons que $ g (x) $ est une fonction différentiable à $ x $ et que la fonction à valeur floue $ f(x) $ est $ g $-différenciable à $ x $. Donc
$$
(f g)_{g}^{\prime}=\left(f^{\prime} g\right)_{g}+\left(f g^{\prime}\right)_{g}
$$
\end{theorem}
\begin{proof}
En utilisant la proposition 4, pour $ h $ assez petit on obtient :
$$\left\|\frac{f(x+h) g(x+h)-_{g} f(x) g(x)}{h}-_{g}\left(\left(f^{\prime}(x) g(x)\right)_{g}+\left(f(x) g^{\prime}(x)\right)_{g}\right)\right\|$$
$$ = \| \frac{f(x+h) g(x+h)-_{g} f(x) g(x+h)+f(x) g(x+h)}{h} 
-_{g} \frac{f(x) g(x)}{h}-_{g}\left(\left(f^{\prime}(x) g(x)\right)_{g}+\left(f(x) g^{\prime}(x)\right)_{g}\right) \| $$
$$
=\|  \frac{\left(f(x+h)-_{g} f(x)\right) g(x+h)+f(x)\left(g(x+h)-_{g} g(x)\right)}{h}
-_{g}\left(\left(f^{\prime}(x) g(x)\right)_{g}+\left(f(x) g^{\prime}(x)\right)_{g}\right) \| \\
$$
$$
\leq \left\|\frac{\left(f(x+h)-_{g} f(x)\right) g(x+h)}{h}-{ }_{g}\left(\left(f^{\prime}(x) g(x)\right)_{g H}\right)\right\|+\left\|\frac{\left(f(x)\left(g(x+h)-_{g} g(x)\right)\right.}{h}-_{g}\left(\left(f(x) g^{\prime}(x)\right)_{g H}\right)\right\| $$
\begin{equation}
\leq \left\|\frac{\left(f(x+h)-_{g} f(x)\right)}{h} g(x+h)-_{g}\left(\left(f^{\prime}(x) g(x)\right)_{g}\right)\right\|+\left\|f(x) \frac{\left(\left(g(x+h)-_{g} g(x)\right)\right.}{h}-_{g}\left(\left(f(x) g^{\prime}(x)\right)_{g}\right)\right\|
\label{4}
\end{equation}
qui complètent la preuve en passant à la limite.
\end{proof}
\begin{definition}\cite{Ref7}\\
On dit qu'un point $ x_{0} \in (a, b) $, est un point de commutation pour la différentiabilité de $ f $, si dans un voisinage $ V $ de $ x_{0} $ il existe des points $ x_{1} <x_{0} <x_{2} $ tel que:\\
1) type (1). En $x_{1}$ (\ref{2}) satisfaite tandis que (\ref{3}) ne satisfaite pas et à $x_{2}$ (\ref{3}) satisfaite et (\ref{2}) ne satisfaite, ou \\
2) type (2). En $x_{1}$ (\ref{3}) satisfaite tandis que (\ref{2}) ne satisfaite pas et à $x_{2}$ (\ref{2}) satisfaite et (\ref{3}) ne satisfaite pas.
\end{definition}
\begin{definition}
Soit $f:(a, b) \rightarrow E^{1}$. Nous disons que $f(x)$ est $g$-differentiable du 2ème ordre en $ x_{0} $ chaque fois que la fonction $ f(x) $ est g-différentiable de l'ordre $i, i=0,1$, à $x_{0},\left(\left(f\left(x_{0}\right)\right)^{(i)}_{g} \in E^{1}\right.$ ), de plus il n'y a pas de point de commutation sur $ (a, b) $. Alors il existe $(f)_{g}^{\prime \prime}\left(x_{0}\right) \in E^{1}$ tel que
$$
\lim _{h \rightarrow 0}\left\|\frac{f^{\prime}\left(x_{0}+h\right)-_{g} f^{\prime}\left(x_{0}\right)}{h}, f_{g}^{\prime \prime}\left(x_{0}\right)\right\|=0
$$
\end{definition}
\begin{definition}
Soit $f:[a, b] \rightarrow E^{1}$ et $f_{g}^{\prime}(x)$ soit $g$-differentiable à $x_{0} \in(a, b)$, de plus il n'y a pas de point de commutation sur $(a, b)$ et $\underline{f}(x, \alpha)$ et $\bar{f}(x, \alpha)$ tous deux différenciables en $x_{0}$. Nous disons que :
\begin{itemize}
\item[•] $f^{\prime}$ est $[(i)-g]$-differentiable en $x_{0}$ si
$$
f_{i, g}^{\prime \prime}\left(x_{0}\right)=\left[\underline{f}^{\prime \prime}(x, \alpha), \bar{f}^{\prime \prime}(x, \alpha)\right]
$$
\item[•] $f^{\prime}$ est $[(i i)-g]$-differentiable en $x_{0}$ si
$$
f_{i i, g}^{\prime \prime}\left(x_{0}\right)=\left[\bar{f}^{\prime \prime}(x, \alpha), \underline{f}^{\prime \prime}(x, \alpha)\right]
$$
\end{itemize}
\end{definition}
\begin{definition}
Soit $f:[a, b] \rightarrow E^{1}$. On dit que $ f (x) $ est un Riemann flou intégrable sur $I \in E^{1}$ si pour tout $\epsilon>0$, il existe $\delta>0$ tel que pour toute division $P=\{[u, v] ; \xi\}$ avec les normes $\Delta(P)<\delta$,on a
$$
d\left(\sum_{p}^{*}(v-u) f(\xi), I\right)<\epsilon
$$
où $\sum_{p}^{*}$ désigne la sommation floue. Nous choisissons d'écrire $I=\int_{a}^{b} f(x) d x$.

\end{definition}
\begin{theorem}
Si $f$ est $g$-différentiable sans point de commutation dans l'intervalle $[a, b]$ alors on a :
$$
\int_{a}^{b} f(t) d t=f(b)-_{g} f(a)
$$
\end{theorem}
\begin{theorem}
Soit $f(x)$  une fonction à valeur floue sur $(-\infty, \infty)$ et il est représenté par
 $$f(x, \alpha)=[\underline{f}(x, \alpha), \bar{f}(x, \alpha)]$$ 
 Pour tout fixe $\alpha \in[0,1] .$ Suppose que
$|\underline{f}(x, \alpha)|$ et $|\bar{f}(x, \alpha)|$ sont Riemann intégrables sur $(-\infty, \infty)$
pour tout $\alpha \in[0,1]$. Alors $f(x)$ est impropre floue Riemann intégrable sur $(-\infty, \infty)$ et l'intégrale floue de Riemann impropre est un nombre flou. De plus, nous avons
$$
\int_{-\infty}^{\infty} f(x) d x=\left[\int_{-\infty}^{\infty} \underline{f}(x, \alpha) d x, \int_{-\infty}^{\infty} \bar{f}(x, \alpha) d x\right]
$$
\end{theorem}
À partir de ce théorème, nous pouvons discuter de l'intégrale impropre de Fuzzy Riemann.
\begin{lemme}
Soit $f: \mathbb{R} \times \mathbb{R}^{+} \rightarrow E^{1}$, donné par $f(x, t ; \alpha)=$
$[\underline{f}(x, t ; \alpha), \bar{f}(x, t ; \alpha)]$, et soit $a \in \mathbb{R}^{+}$\\
Si $\int_{a}^{\infty} \underline{f}(x, t ; \alpha) d t$ et $\int_{a}^{\infty} \bar{f}(x, t ; \alpha) d t$ convergent alors :
$$
\int_{a}^{\infty} f(x, t ; \alpha) d t \in E^{1}
$$
\end{lemme}
\begin{proof}
Utilisez simplement les conditions (1).
\end{proof}
\begin{theorem}
Soit $f: \mathbb{R} \times \mathbb{R}^{+} \rightarrow E^{1}$  une fonction à valeur floue telle que $f(x, t ; \alpha)=[f(x, t ; \alpha), \bar{f}(x, t ; \alpha)]$. Supposons que pour chaque $x \in[a, \infty)$, l'intégrale floue $\int_{c}^{\infty} f(x, t) d t$ est convergent et de plus $\int_{a}^{\infty} f(x, t) d x$ en tant que fonction de $t$ est convergent sur $[c, \infty) .$ Alors :
$$
\int_{c}^{\infty} \int_{a}^{\infty} f(x, t) d x d t=\int_{a}^{\infty} \int_{a}^{\infty} f(x, t) d t d x.
$$
\end{theorem}
\begin{proof}
Application du théorème de Fubini-Tonelli à ces deux fonctions $f(x, t ; \alpha)$ et $\bar{f}(x, t ; \alpha)$, utilisez les conditions (1).
\end{proof}
\begin{theorem}
 Supposons les deux, $f(x, t)$ et $\partial_{x_{g H}} f(x, t)$, sont floues continues dans $[a, b] \times[c, \infty)$. Supposons aussi que l'intégrale converge pour $x \in \mathbb{R}$, et l'intégrale $\int_{c}^{\infty} f(x, t) d t$ converge uniformément sur $[a, b]$. Alors $F$ est $g H$-différenciable sur $[a, b]$ et :
$$
F_{g H}^{\prime}(x)=\int_{c}^{\infty} \partial_{x_{g H}} f(x, t) d t
$$
\end{theorem}
\begin{proof}
La continuité de $\partial_{x_{g H}} f(x, t)$ sur $[a, b]$ par le théorème de convergence domainée $\underline{f}(x, t ; \alpha)$ et $\bar{f}(x, t ; \alpha)$ et utiliser le condition (1).
\end{proof}
 D'après le théorème (7) on obtient
\begin{theorem}
 Soit $f:[a, b] \rightarrow E^{1}$ et $g:[a, b] \rightarrow \mathbb{R}$ sont deux fonctions différenciables $(f$ et $g H$-differentiable), alors :
$$
\int_{a}^{b} f_{g H}^{\prime}(x) g(x) d x=f(b) g(b)-_{g} f(a) g(a)-_{g} \int_{a}^{b} f(x) g^{\prime}(x) d x.
$$
\end{theorem}
\begin{remarque}
Si $f, g \in A^{E^{1}}$ avec $\lim _{|x| \rightarrow \infty} f(x)=0$,
$\lim _{|x| \rightarrow \infty} g(x)=0$ alors
$$
\int_{-\infty}^{\infty} f_{g H}^{\prime}(x) g(x) d x=\int_{-\infty}^{\infty} f(x) g^{\prime}(x) d x.
$$
\end{remarque}
\subsubsection{Dérivée fractionnaire floue généralisée}
Nous présentons les dérivés fractionnaires flous généralisés et leurs propriétés.

\begin{definition} 
Soit $f \in L^{E^{1}([a, b]) .}$ Le flou Riemann-Liouville intégral de fonction à valeur floue $\mathrm{f}$ est défini comme suit :
$$
I^{q} f(t)=\frac{1}{\Gamma(q)} \int_{a}^{t}(t-s)^{q-1} f(s) d s, \quad a<s<t, \quad 0<q<1
$$
\end{definition}
\begin{definition}
Soit $f(x, t ; \alpha)=[\underline{f}(x, t ; \alpha), \bar{f}(x, t ; \alpha)]$
 une fonction floue valorisée. L'intégrale floue de Riemann-Liouville de $ f $ est définie comme suit:
$$
{ }_{g H} D_{t}^{q} f(t, x ; \alpha)=\frac{1}{\Gamma(1-q)} \int_{a}^{t}(t-s)^{-q} f_{g H}^{\prime}(s) d s, \texttt{  a<s<t,\quad 0<q<1 }
$$
On dit aussi que $f$ est $[(i)-g H]-$differentiable en $t_{0}$ si
$$
{ }_{g H} D_{t}^{q} f(x, t ; \alpha)=\left[D^{q} \underline{f}(x, t ; \alpha), \bar{f}(x, t ; \alpha)\right].
$$
Et $f$ est $[(i i)-g H]$-differentiable en $t_{0}$ si
$$
{ }_{g H} D_{t}^{q} f(x, t ; \alpha)=\left[D^{q} \bar{f}(x, t ; \alpha), \underline{f}(x, t ; \alpha)\right].
$$
\end{definition}
\begin{lemme}
Soit $f \in A^{E^{1}}$ et $r \in(0,1)$,alors\\
1) Si $f$ est $[(i)-g H]$-differentiable en $t_{0}$ alors $D^{r} f$ est $[(i)-g H]$-differentiable en $t_{0}$.\\
2) Si $f$ est $[(i i)-g H]$-differentiable en $t_{0}$ alors $D^{r} f$ est $[(i i)-g H]$-differentiable en $t_{0}$
\end{lemme}
\begin{proof}
Noter que
$$
{ }_{g H} D^{q} f(t)=\frac{1}{\Gamma(1-q)} \int_{0}^{t}(t-s)^{-q} f_{g H}^{\prime}(s) d s.
$$
Depuis $\frac{1}{\Gamma(1-q)}(t-s)^{-q}$ est une quantité non négative chaque fois que $0<t<s$.
\begin{theorem}
Soit $f \in A^{E^{1}}$ et $q \in(1,2)$,alors
$$
{ }_{g H} D^{q} f(t)={ }_{g H} D^{q-1} f_{g H}^{\prime}(t).
$$
\end{theorem}
\end{proof}
\begin{proof}
Nous fixons $f(t)=[f(t ; \alpha), \bar{f}(t ; \alpha)]$ et utilisez le lemme $(4)$\\
Si $f$ est $[(i)-$differentiable $]$ alors :
$$
f(t)^{\prime}=\left[f^{\prime}(t ; \alpha), \bar{f}^{\prime}(t ; \alpha)\right].
$$
Et
$$
D^{q-1} f(t)^{\prime}=\left[D^{q-1} f^{\prime}(t ; \alpha), D^{q-1} \bar{f}^{\prime}(t ; \alpha)\right].
$$
Si $f$ est $[(i)-$differentiable $]$ alors :
$$
f(t)^{\prime}=\left[\bar{f}^{\prime}(t ; \alpha), \underline{f}^{\prime}(t ; \alpha)\right].
$$
Et
$$
D^{q-1} f(t)^{\prime}=\left[D^{q-1} \bar{f}^{\prime}(t ; \alpha), D^{q-1} \underline{f}^{\prime}(t ; \alpha)\right]
$$
\end{proof}
\begin{proposition}
 Soit $f: L^{E^{1}}$. Si $D^{\gamma-1} f(t)=g(t)$, alors $f(t)=f(0)+t f_{g H}^{\prime}(0)+I^{\gamma-1} g(t).$
\end{proposition}
\begin{proof}
Nous fixons $f(t)=[\underline{f}(t ; \alpha), \bar{f}(t ; \alpha)]$ et $g(t)=$
$[\underline{g}(t ; \alpha), \bar{g}(t ; \alpha)]$\\
1) Si $f$ est $[(i)$-differentiable] par théorème (13)
$$
\begin{array}{c}
D^{\gamma-1} f(t)=\left[D^{\gamma-1} \underline{f}(t ; \alpha), D^{\gamma-1} \bar{f}(t ; \alpha)\right] \\
=[g(t ; \alpha), \bar{g}(t ; \alpha)]
\end{array}
$$
Ce qui implique que
$$
\left\{\begin{array}{l}
D^{\gamma-1} \underline{f}(t ; \alpha)=\underline{g}(t ; \alpha) \\
D^{\gamma-1} \bar{f}(t ; \alpha)=\bar{g}(t ; \alpha)
\end{array}\right.
$$
Par \cite{Ref8} on a :
$$
\left\{\begin{array}{l}
\underline{f}(t ; \alpha)=\underline{f}(0 ; \alpha)+t \underline{f}^{\prime}(0 ; \alpha)+I^{\gamma-1} \underline{g}(t ; \alpha) \\
\bar{f}(t ; \alpha)=\bar{f}(0 ; \alpha)+t \bar{f}^{\prime}(0 ; \alpha)+I^{\gamma-1} \bar{g}(t ; \alpha)
\end{array}\right.
$$
dans le même si $f$ est $[(i i)-$differentiable] alors
$$
\left\{\begin{array}{l}
\underline{f}(t ; \alpha)=f(0 ; \alpha)+t \bar{f}^{\prime}(0 ; \alpha)+I^{\gamma-1} \underline{g}(t ; \alpha) \\
\bar{f}(t ; \alpha)=\bar{f}(0 ; \alpha)+t \underline{f}^{\prime}(0 ; \alpha)+I^{\gamma-1} \bar{g}(t ; \alpha)
\end{array}\right.
$$
Ainsi
$$
f(t)=f(0)+t f_{g H}^{\prime}(0)+I^{\gamma-1} g(t)
$$
\end{proof}
\section{Quelques propriétés des dérivées fractionnaires}
\subsection{Linéarité}
La diférentiation fractionnaire est une opération linéaire :
$$
D^{p}(\lambda f(t)+\mu g(t))=\lambda D^{p} f(t)+\mu D^{p} g(t)
$$

où $ D^{p} $ désigne n’importe quelle approche de dérivation considérée dans ce mémoire.

\subsection{Règle de Leibniz}
Pour $n$ entier on a :
$$
\frac{d^{n}}{d t^{n}}(f(t) . g(t))=\sum_{k=0}^{n}\left(\begin{array}{l}
n \\
k
\end{array}\right) f^{(k)}(t) g^{(n-k)}(t)
$$
La généralisation de cette formule nous donne :
$$
D^{p}(f(t) . g(t))=\sum_{k=0}^{n}\left(\begin{array}{l}
p \\
k
\end{array}\right) f^{(k)}(t) D^{p-k} g(t)-R_{n}^{p}(t)
$$
ou $n \geq p+1$ et $R_{n}^{p}(t)=\frac{1}{n ! \Gamma(-p)} \int_{a}^{t}(t-\tau)^{-p-1} g(\tau) d \tau \int_{\tau}^{t} f^{(n+1)}(\xi)(\tau-\xi)^{n} d \xi$. (on a
$\left.\lim _{n \rightarrow+\infty} R_{n}^{p}(t)=0\right)$
Si $f$ et $g$ avec toutes ses dérivées sont continues dans $[a, t]$ la formule devient:
$$
D^{p}(f(t) \cdot g(t))=\sum_{k=0}^{\infty}\left(\begin{array}{l}
p \\
k
\end{array}\right) f^{(k)}(t) D^{p-k} g(t)
$$
$D^{p}$ est la dérivée fractionnaire au sens de Grünwald-Letnikov et au sens de Riemann-
Liouville.
\section{Les équations différentielles floues}
Cette section présente diverses approches traitant de la définition d'une dérivée de fonctions floues de type 1 ou de type 2. La partie la plus importante de l'histoire des EDF´s est constituée de différentes définitions de dérivées floues. En effet, le concept de dérivée étant l'élément fondamental d'une équation différentielle, l'évolution des dérivées floues joue un rôle clé dans l'évolution des EDF´s. Les dérivées floues peuvent être classées en : dérivées floues d'ordre entier et d'ordre fractionnaire. Les dérivées floues d'ordre entier sont sous-classées en dérivées floues d'ordre entier des fonctions floues de type 1 (ou dérivées floues de type 1) et dérivées floues d'ordre entier des fonctions floues de type 2 (ou dérivées floues de type 2). De même, il existe des dérivés fractionnaires flous de type 1 et de type 2. Il est à noter que correspondant à chaque classe ou sous-classe de dérivés flous, des EDF´s peuvent être classés. Par exemple, ce que l'on peut appeler des équations différentielles fractionnaires floues de type 1 sont associées aux EDF´s dans lesquelles la dérivée est du type dérivées fractionnaires floues de type 1.
\subsection{Les équations différentielles floues d'ordre entier}
Bien que le terme d'équations différentielles floues soit apparu pour la première fois dans la littérature en 1978, les EDF´s, comme on les appelle de nos jours, ont été initiées en 1982 sur la base d'une définition d'une dérivée floue qui peut être appelée dérivée de Dubois-Prade. Par la suite, différentes définitions des dérivés flous ont été proposées parmi lesquelles le dérivé de Hukuhara (ou dérivé de Puri-Ralescu) présenté en 1983, le dérivé de Goetschel-Voxman en 1986, le dérivé de Seikkala en 1987 et le dérivé de Friedman-Ming-Kandel introduit en 1996, respectivement. Malgré le fait que toutes ces dérivées floues aient été présentées sous des formes différentes, il a été prouvé qu'elles sont équivalentes à condition que les fonctions floues soumises aux coupes $\alpha-$coupe inférieures et supérieures soient des fonctions continues, pour plus de détails voir \cite{Ref16}.

Parmi les dérivés flous mentionnés, les dérivés de Hukuhara et de Seikkala sont plus largement connus. La différence entre les définitions des dérivés de Hukuhara et de Seikkala est que le dérivé de Hukuhara (dérivé $H$) est, en substance, défini sur la base de ce qu'on appelle la différence de Hukuhara (différence-$H$), mais le dérivé de Seikkala est défini sur la base des dérivés du  inférieur et supérieur. $\alpha-$coupe de la fonction floue en question. L'existence et l'unicité de la solution pour les EDF´s sous dérivée H et dérivée de Seikkala ont été étudiées dans \cite{Ref17}.

Un grand nombre d'études menées sur les EDF´s, démontrent que les dérivés de Hukuhara et Seikkala, bien qu'équivalents, sont des définitions plus acceptables. Cependant, les résultats de la recherche ont révélé que ces dérivées souffrent d'un certain nombre de limitations majeures parmi lesquelles la plus grave est que le diamètre de la fonction floue à l'étude doit être nécessairement non décroissant. Une telle limitation fait que la solution obtenue d'un EDF, dans un grand nombre de cas, diffère de ce qui est réalisé intuitivement de la nature du système ou phénomène modélisé par le EDF. A titre d'illustration, le diamètre de la solution obtenue d'un EDF sous la forme $\dot{\tilde{x}}(t)=-\tilde{x}(t)$ dont la condition initiale est un nombre flou, augmente au fur et à mesure que le temps passe. C'est alors que nous nous attendons intuitivement à ce que le comportement naturel d'une telle équation différentielle montre que $\tilde{x}$ diminue avec le temps. En conclusion, considérer de telles définitions dans une EDF nécessite que le flou de la solution soit non décroissant, ce qui impose une grande restriction à leurs applications réelles.

Pour surmonter ce problème, en $1990$ puis avec plus de détails en 1997, il a été suggéré de considérer les EDF´s comme des inclusions différentielles floues. Presque simultanément, une approche alternative basée sur l'utilisation du principe d'extension de Zadeh (PEZ) pour traiter les EDF´s a été introduite en 1999. Bien que ces approches aient attiré une attention considérable et ont conduit à de nombreuses études remarquables sur les EDF´s, elles ne s'accompagnent pas d'une définition de dérivée floue. En clair, le concept de dérivée floue est en effet perdu dans les approches proposées. Un autre effort fait afin de surmonter le problème venant de l'application du dérivé de Hukuhara (ou de manière équivalente Seikkala), a été la présentation des dérivés de même ordre et d'ordre inverse qui ont été faits sur la base du dérivé de Seikkala. Il convient de souligner que cette approche a une relation étroite avec celles présentées pour les dérivées floues appelées dérivées de Hukuhara généralisées et dérivées de Seikkala généralisées qui seront expliquées dans la suite.

L'année 2004 est venue avec un point de départ pour faire évoluer le traitement des EDF´s en introduisant le concept de dérivé fortement généralisé de Hukuhara ($SGH$) qui a été présenté de manière plus complète en $2005$. La structure du dérivé $SGH$, en général et dans certaines conditions, présente deux formes de différentiabilité d'une fonction floue qui peuvent être appelées la première forme et la seconde forme de différentiabilité. La première forme coïncide avec le dérivé Hukuhara. Mais c'est la seconde forme de différentiabilité, si elle existe, qui répond à la question du diamètre non décroissant d'une fonction floue dérivable. En termes simples, si la fonction floue $\tilde{f}$ : $t \rightarrow \tilde{f}(t)$ est différentiable-$SGH$ sous la seconde forme, alors son diamètre est non croissant, c'est-à-dire $\frac{ d \mathcal{D}(\bar{f}(t))}{dt} \leq 0$.

Ainsi, la résolution du EDF $\dot{\tilde{x}}(t)=-\tilde{x}(t)$ dont la condition initiale est un nombre flou, au sens de la dérivée $SGH$ seconde forme, aboutit à une solution qui satisfait ce qui est intuitivement attendu de la nature de la structure de l'équation.

Cette approche s'accompagne également d'un concept intéressant appelé points de commutation qui sont les points dans un intervalle où se produit le basculement entre la première forme et la deuxième forme de différentiabilité. Ce concept a ouvert une porte à l'étude du comportement périodique de certains phénomènes dont les modèles mathématiques peuvent être considérés comme des EDF´s dans lesquels une définition de la dérivée floue a été présentée. Le résultat concernant l'existence et l'unicité des solutions d'un EDF sous SGH-dérivé donné dans \cite{Ref18} , montre qu'un premier ordre EDF, sous certaines conditions, a deux solutions qui peuvent être appelées la première forme et la deuxième forme solutions . Les solutions des première et deuxième formes sont associées aux concepts des première et deuxième formes de la différentiabilité. Bien que le dérivé de SGH ait marqué un tournant dans l'analyse des EDF´s et qu'un nombre considérable de recherches sur les EDF´s aient été effectuées sur la base d'un tel dérivé, il souffre de certaines lacunes dont les plus importantes sont décrites ci-dessous. Premièrement, puisque le dérivé $SGH$, en substance, a été introduit sur la base de la différence Hukuhara, l'existence d'un tel dérivé dépend de l'existence de la différence $H$. Néanmoins, dans la plupart des cas, la différence $H$ n'existe pas et les conditions d'existence d'une telle différence restreindraient considérablement l'applicabilité du dérivé $SGH$. Deuxièmement, la dérivée $SGH$ serait appliquée sur les fonctions floues avec des diamètres monotones. Plus précisément, pour prendre la dérivée d'une fonction floue au sens de la première forme de différentiabilité, le diamètre de la fonction floue doit être nécessairement non décroissant. De manière analogue, que le diamètre d'une fonction floue soit non croissant est une des conditions nécessaires de la différentiabilité de la fonction floue sous la seconde forme.

Le concept de dérivé-$\pi$ est une autre approche alternative qui a été introduite dans $2009$ Dans \cite{Ref19} , il a été déclaré que dans les conditions du théorème de représentation, le dérivée-$\pi$ d'une fonction floue existe. De plus, la solution obtenue d'une équation différentielle floue sous le concept de $\pi$-dérivé coïncide avec celle obtenue sous le concept de SGH-dérivé, sous certaines conditions.

Pour surmonter les limitations de la dérivée $SGH$, une dérivée de Hukuhara ($gH$) généralisée d'une fonction floue a été présentée en 2013. La dérivée gH a été définie sur la base de la différence de Hukuhara généralisée (différence $gH$) qui est un concept plus général que la différence H . Bien que l'existence de la gH-différence s'accompagne de moins de restrictions par rapport à la différence-$H$, il est possible que la différence-$gH$ de deux nombres flous n'existe pas. Par conséquent, on ne garantit pas l'existence d'une dérivée $gH$ d'une fonction floue. Malgré ce fait, le dérivé de la $gH$ aborde la deuxième limitation du dérivé de la $SGH$. C'est-à-dire que le diamètre d'une fonction floue différentiable-$gH$ n'a pas besoin d'être monotone. De plus, le concept de points de commutation a été mieux clarifié sur la base du dérivé $gH$.

En conséquence, l'étude des EDF´s sous le concept de gH-différentiabilité s'accompagne de beaucoup moins de restrictions par rapport à d'autres concepts. C'est peut-être l'une des raisons pour lesquelles un grand nombre d'études ont été menées sur la différentiabilité gH d'une fonction floue et les EDF´s équipés d'un tel concept. Concomitamment, afin d'aborder la question de l'existence de dérivé-$gH$ d'une fonction floue, le concept de dérivée généralisée (dérivé-$g$) basé sur la différence généralisée (différence-$g$) a été introduit dans $2013$. Bien qu'il ait été initialement affirmé que la différence-$g$ de deux nombres flous existe toujours, avec un contre-exemple présenté en 2015, il a été montré que ce n'était pas le cas. Par une petite modification dans la définition de différence-$g$, cependant, l'existence d'une telle différence de nombres flous peut être garantie. Il est à noter que sous certaines conditions, $g$-différentiabilité, $gH$-différentiabilité et $SGH$-différentiabilité deviennent des concepts équivalents. Notez que sur la base de la différence $gH$ au niveau de deux nombres flous, le concept de dérivé $gH$ au niveau du niveau (dérivé $LgH$) a également été défini dans \cite{Ref7} et approfondi dans \cite{Ref20}. Bien que la différentiabilité LgH d'une fonction floue soit un concept plus général que le concept de différentiabilité $gH$ et moins général que le concept de différentiabilité $g$, l'existence d'une dérivée $LgH$ d'une fonction floue n'est pas garantie. L'un des points concernant la caractéristique de la dérivée $gH$ et de la dérivée LgH qui doit être souligné est que l'existence de telles dérivées pour une fonction floue n'implique pas nécessairement que les extrémités de la fonction floue soient dérivables. Le développement de la théorie des EDF a été poursuivi, basé sur les dérivées d'ordre entier, par d'autres approches sur les dérivées floues, à savoir dérivé-$\hat{D}$, dérivé-$H_{2}$, dérivé interactif, le $gr $-derivative , et ainsi de suite qui seront discutés dans la suite. L'année 2013 est venue avec une approche pour l'étude des EDF dans laquelle la définition d'une dérivée floue a été tirée de la fuzzification de l'opérateur dérivé classique par l'utilisation du principe d'extension de Zadeh. Il existe une relation étroite entre cette approche et celle introduite pour les inclusions différentielles floues, et sous certaines hypothèses, les résultats obtenus par cette approche sont réduits à ceux obtenus par les inclusions différentielles floues. De plus, dans les conditions exprimées dans le théorème $3.17$ dans \cite{Ref21}, la dérivée $\hat{D}$ est équivalente à la dérivée $gH$. Ainsi, il est possible qu'un EDF de premier ordre ait plus d'une solution basée sur le concept de dérivé-$\hat{D}$.

En 2014, les équations différentielles floues de type 2 (EDFT2) impliquant des fonctions floues de type 2 et des nombres flous de type 2 ont été introduites, pour la première fois, dans \cite{Ref22}. Le concept de $H_{2}$-dérivé a été présenté pour traiter les EDFT2. Cette dérivée se présente sous la forme d'une dérivée $SGH$ et est définie sur la base de la différence H des nombres flous de type 2, c'est-à-dire la différence $H_{2}$. La principale raison de la présentation des EDFT2 vient du fait qu'une forme exacte d'un nombre flou de type 1 peut ne pas toujours être déterminée. Avec les notions de distribution de possibilité conjointe et d'opération arithmétique interactive considérées, la dérivée interactive d'une fonction floue a été introduite dans $2017$. Dans la dérivée interactive, la différence, dite différence interactive, a été définie sur la base du principe d'extension dit sup$J$. L'une des raisons de l'analyse des EDF par cette approche a été évoquée en raison de l'existence d'interactivités (ou de dépendances) possibles entre les variables d'un processus. Récemment, il a été prouvé que la dérivée $H$, la dérivée $gH$, la dérivée $\mathrm{g}$ et la dérivée $\pi$ sont des cas particuliers de la dérivée interactive.

L'année 2018 est arrivée avec une nouvelle approche pour l'analyse des EDF´s en introduisant un nouveau concept de dérivée floue appelée dérivée granulaire $(g r$-dérivé). La $g r$-dérivé a été définie sur la base de la notion de $g r$-différence. La principale différence entre cette approche et les autres est qu'elle utilise l'arithmétique des intervalles flous par mesure de distance relative (AIF-MDR) pour traiter les EDF´s. Un concept clé dans AIF-MDR est la fonction d'appartenance horizontale (FMH) basée sur laquelle les opérations sur les nombres flous sont définies. La principale raison pour laquelle $gr$-dérivé a été proposé est de surmonter les inconvénients des approches à savoir $H$-dérivé, $SGH$-dérivé, $\mathrm{gH}$-dérivé, $\mathrm{g}$-dérivé et $\pi $-dérivé qui utilise l'arithmétique floue des intervalles standard (AFIS) pour gérer les EDF. Les inconvénients seraient décrits comme : l'existence de la dérivée, l'incertitude monotone, la multiplicité des solutions, la propriété de doublement, l'incertitude symétrique autour du problème zéro (problème SUAZ) et le phénomène de comportement non naturel dans la modélisation (UBM).

La même année, la notion de dérivée de Seikkala généralisée (dérivé-$gS$) a été avancée dans \cite{Ref23}. Cette notion est, par essence, une combinaison de dérivées de même ordre et d'ordre inverse par l'utilisation d'opérateurs minimum et maximum. De plus, il a été prouvé que le dérivé gS est équivalent au dérivé $SGH$. Inspiré du calcul quantique et de la dérivée $q$, la dérivée $q$ floue généralisée de Hukuhara a été proposée dans $2019$ comme une combinaison de la dérivée $q$ et de la difference-$\mathrm{gH}$. Il est à noter qu'à partir d'une telle dérivée, une EDF du premier ordre, sous certaines conditions, pourrait avoir deux solutions.

Le tableau suivant montre certaines des dérivées floues expliquées dans cette section.

\begin{center} 
 \begin{table}[h]
\scalebox{0.67}{%
\centering
\begin{tabular}{|c|c|c|}
\hline
Dérivée floue & Définition & commentaires \\
\hline
La dérivé $H$    & $\begin{array}{ll} &\tilde{f}^{\prime}(t)=\lim\limits_{h \rightarrow 0^{+}} \frac{\tilde{f}(t+h) \underline{H} \tilde{f}(t)}{h}\\ &=\lim\limits_{h \rightarrow 0^{+}} \frac{\tilde{f}(t)^{\underline{H}} \tilde{f}(t-h)}{h} \end{array}$ & $\tilde{u}^{H} \bar{v}=\tilde{w} \Leftrightarrow \tilde{u}=\bar{v}+\tilde{w}$ \\

\hline 

Dérivé de Seikkala & {$\left[\tilde{f}^{\prime}(t)\right]^{\alpha}=\left[\underline{f^{\prime} \alpha}(t), \bar{f}^{\prime \alpha}(t)\right]$} &    \\

\hline 

Dérivés de même ordre et d'ordre inverse &  $ \begin{array}{ll} &(a)\left[\tilde{f}^{\prime \prime}(t)\right]^{\alpha}=\left[\underline{f}^{\prime \alpha}(t), \bar{f}^{\prime \alpha}(t)\right]\\ &
 (b)\left[\tilde{f}^{\prime}(t)\right]^{\alpha}=\left[\bar{f}^{\prime \alpha}(t), \underline{f}^{\prime \alpha}(t)\right]\end{array}$  &  (a) Same-order (b) Reverse-order \\
 
\hline

La dérivé $S G H$ & $\begin{aligned}&\tilde{f}^{\prime}(t) & =\lim\limits_{h \rightarrow 0^{+}} \frac{\tilde{f}(t+h) \underline{H} \tilde{f}(t)}{h}=\lim\limits_{h \rightarrow 0^{+}} \frac{\tilde{f}(t) \underline{H} \tilde{f}(t-h)}{h}  \\
&\tilde{f}^{\prime}(t) & =\lim\limits_{h \rightarrow 0^{+}} \frac{\tilde{f}(t) \underline{H} \tilde{f}(t+h)}{-h}=\lim\limits_{h \rightarrow 0^{+}} \frac{\tilde{f}(t-h) \underline{H} \tilde{f}(t)}{-h}\end{aligned}$ &  \\

\hline

La dérivé $\pi$ & $\tilde{f}^{\prime}(t)=\lim\limits_{h \rightarrow 0} \frac{\tilde{f}(t+h)^{\pi} \tilde{f}(t)}{h}$ & $\tilde{u}^{\pi} \tilde{v}=\tilde{w} \Leftrightarrow[\tilde{w}]^{\alpha}=\left[\min \left\{\underline{u}^{\alpha}-\underline{v}^{\alpha}\right\}, \max \left\{\bar{u}^{\alpha}-\bar{v}^{\alpha}\right\}\right]$ \\

\hline

La dérivé $g H$ & $\tilde{f}^{\prime}(t)=\lim\limits_{h \rightarrow 0} \frac{\tilde{f}(t+h)^{g H} \tilde{f}(t)}{h}$ & $\tilde{u}^{g H} \tilde{v}=\tilde{w} \Leftrightarrow(\tilde{u}=\tilde{v}+\tilde{w}$ or $\tilde{v}=\tilde{u}-\tilde{w})$ \\

\hline

La dérivé $G$ & $\tilde{f}^{\prime}(t)=\lim\limits_{h \rightarrow 0} \frac{\tilde{f}(t+h) \underline{g} \tilde{f}(t)}{h}$ & $\tilde{u} \underline{g} \tilde{v}=\tilde{w} \Leftrightarrow[\tilde{w}]^{\alpha}=c l\left(\operatorname{con} v \bigcup_{\beta \geq \alpha}\left([\tilde{u}]^{\beta} \stackrel{g H}{-}[\tilde{v}]^{\beta}\right)\right.$ \\

\hline

La dérivé $\hat{D}$ & $\tilde{f}^{\prime}(t)=\hat{D} \tilde{f}(t)$ & $\mu_{\tilde{D} \tilde{f}(t)}(y)=\left\{\begin{array}{ll}\sup\limits_{f(t) \in D-1_{y}} \mu_{f(t)}(f(t)) & \text { if } D^{-1} y \neq \emptyset \\
0 & \text { if } D^{-1} y=\emptyset  \end{array}\right.$ \\

\hline 

Dérivée interactive & $\tilde{f}^{\prime}(t)=\lim\limits_{h \rightarrow 0} \frac{\tilde{f}(t+h) \underline{J} \tilde{f}(t)}{h}$ & $\mu_{(i \hat{i}-\tilde{j})}(z)=\sup\limits_{x-y=z} \mu_{J}(x, y)$ \\

\hline

La dérivé $g S$ & $\left[\tilde{f}^{\prime}(t)\right]^{\alpha}=\left[\min \left\{\underline{f}^{\prime \alpha}(t), \bar{f}^{\prime \alpha}(t)\right\}, \max \left\{\underline{f}^{\prime \alpha}(t), \bar{f}^{\prime \alpha}(t)\right\}\right]$ & \\

\hline

La dérivé $g r$ & $\tilde{f}^{\prime}(t)=\lim_{h \rightarrow 0} \frac{\tilde{f}(t+h) \underline{g}{r} \tilde{f}(t)}{h}$ & $\begin{aligned}&\mathcal{H}(\tilde{u}) \triangleq \underline{u}^{\alpha}+\left(\bar{u}^{\alpha}-\underline{u}^{\alpha}\right) \beta_{u} \\&\tilde{u} \frac{g r}{v}=\tilde{w} \Leftrightarrow \mathcal{H}(\tilde{u})-\mathcal{H}(\tilde{v})=\mathcal{H}(\tilde{w})\end{aligned}$ \\

\hline
\end{tabular}}
\caption{Certains des dérivés flous.}
\end{table}
\end{center}

\subsection{Les équations différentielles floues d'ordre fractionnel}
L'idée d'étudier les équations différentielles floues d'ordre fractionnaire a été présentée pour la première fois dans $2010$ \cite{Ref24}. Il existe différentes définitions des dérivés fractionnaires classiques, notamment au sens de Riemann-Liouville, Caputo, Riemann-Liouville modifié, dérivé fractionnaire conforme, dérivé fractionnaire Caputo-Fabrizio, pour n'en citer que quelques-uns. Les combinaisons possibles de telles dérivées avec les concepts de dérivées floues et/ou de différences floues ont abouti à l'introduction de différentes définitions de dérivées fractionnaires floues sur la base desquelles des équations différentielles fractionnaires floues (EDFF) ont été examinées.

La dérivée fractionnaire floue de Riemann-Liouville au sens de dérivée H ; et l'existence et l'unicité de la solution pour une classe de EDFF à retard infini ont été présentées dans $2010$ \cite{Ref25} . En 2011, la dérivée fractionnaire floue de Riemann-Liouville au sens de dérivée de Seikkala a été proposée dans \cite{Ref26} ; et l'existence et l'unicité de la solution pour les EDFF´s avec des conditions initiales floues sous une telle dérivée ont été montrées dans \cite{Ref26}, \cite{Ref27}.

Une combinaison de dérivé fractionnel flou de Riemann-Liouville avec un dérivé de SGH Riemann-Liouville H-dérivé a été introduite en 2012. Il est à noter que bien qu'un tel dérivé ait été désigné comme RiemannLiouville H-dérivé dans la littérature, le dérivé a été défini dans la forme du dérivé SGH. De ce fait, similaire au dérivé SGH, en général, le dérivé H de Riemann-Liouville, dans certaines conditions, présente deux formes de différentiabilité d'une fonction floue. Ainsi, une FFDE d'ordre $\beta \in(0,1)$ sous une telle dérivée peut avoir deux solutions, c'est-à-dire une solution qui vient de la première forme de différentiabilité, et l'autre de la deuxième forme de différentiabilité. L'existence et l'unicité de la solution pour les EDFF utilisant la condition de type Krasnoselskii-Krein et la condition de type Nagumo ont été présentées dans \cite{Ref28}, \cite{Ref29}.

En 2012, les équations différentielles fractionnaires floues sous le concept de dérivée de Caputo en combinaison avec le dérivé SGH, c'est-à-dire le dérivé H de type Caputo, ont été étudiées dans deux représentations différentes dans \cite{Ref30} et \cite{Ref31}. La théorie de l'existence et de l'unicité de la solution pour les EDFF sous une telle dérivée présentée dans \cite{Ref31} montre qu'une EDFF d'ordre $\beta \in(0,1)$, sous certaines conditions, peut avoir deux solutions correspondant à la première et deuxième formes du dérivé H de type Caputo.

Il faut souligner qu'il existe quelques différences entre la dérivée de Caputo et la dérivée de Riemann-Liouville. Ce qui suit est deux différences plus importantes. Premièrement, la dérivée de Caputo d'une fonction constante est nulle, ce qui n'est pas le cas avec la dérivée de Riemann-Liouville. La deuxième différence concerne les conditions initiales d'une équation différentielle fractionnaire. Une équation différentielle fractionnaire sous le concept de dérivée de Riemann-Liouville implique des conditions initiales d'ordre fractionnaire qui ne se produisent pas sous le concept de dérivée de Caputo.

Avec une combinaison de dérivée de Riemann-Liouville et de dérivée de Goetschel-Voxman considérée, l'existence et l'unicité de la solution pour les EDFF ont été démontrées en 2013 \cite{Ref32}. Simultanément, la dérivée fractionnaire floue de Riemann-Liouville au sens de dérivée H a été définie dans \cite{Ref33} ; et en utilisant le théorème du point fixe de Schauder, l'existence de la solution d'un EDFF sous une telle dérivée a été étudiée dans \cite{Ref33}, \cite{Ref34}.

En 2014, des équations différentielles fractionnaires floues de type 2 (EDFFT2) sous le concept de dérivées fractionnaires floues de type 2 ont été établies dans \cite{Ref35}. Les EDFFT2 sont des EDFF dans lesquels des nombres flous de type 2 et des fonctions floues de type 2 sont impliqués. Les dérivées fractionnaires floues de type 2 ont été définies sous la forme des dérivées de Caputo et de Riemann-Liouville en combinaison avec $H_{2}$ -dérivé, c'est-à-dire Caputotype $H_{2}$-dérivé et Riemann-Liouville $H_{2 }$-dérivé. Ces dérivées, en général, sont sous la forme de dérivées SGH, et de ce fait, la théorie de l'existence et de l'unicité de la solution pour les EDFFT2 donnée dans \cite{Ref35} montre qu'un EDFFT2 d'ordre $\beta \in(0, 1)$ sous de telles dérivées peut avoir deux solutions qui peuvent être appelées la première forme et la deuxième forme solutions. La même année, les EDFF sous le concept de dérivée de Caputo floue généralisée (dérivée de Caputo gH) avec la théorie de l'existence et de l'unicité de leurs solutions par l'utilisation de la condition de Krasnoselskii-Krein ont été présentées dans \cite{Ref36}. Le dérivé gH de Caputo a été constitué par une combinaison du dérivé de Caputo et du dérivé $g H$.

Une combinaison de dérivée de Riemann-Liouville modifiée \cite{Ref37} et de $g$-dérivé pour les équations différentielles fractionnaires floues de type-1 et de type-2 a été introduite dans $2016$ \cite{Ref38}. La principale raison de présenter la dérivée fractionnaire floue de RiemannLiouville modifiée vient du fait que, contrairement à la dérivée gH de Caputo, elle n'exige pas que la fonction en question soit dérivable d'ordre supérieur. Par exemple, pour la dérivée fractionnaire d'ordre $\beta \in(0,1)$, la dérivée fractionnaire floue de Riemann-Liouville modifiée ne nécessite pas que la fonction en question soit dérivable au premier ordre. De plus, contrairement à la dérivée fractionnaire floue au sens de Riemann-Liouville, les conditions initiales apparaissent de la même manière que dans une équation différentielle d'ordre entier.

Les EDFF sous le concept de dérivé fractionnaire de Caputo-Fabrizio en combinaison avec un dérivé de SGH (Caputo-Fabrizio SGH-dérivé) ont été étudiés dans $2018$ \cite{Ref39}. La principale raison pour laquelle une telle dérivée a été introduite est que le noyau de la dérivée fractionnaire de Caputo-Fabrizio, contrairement aux dérivés de Riemann-Liouville et de Caputo, est non singulier. Il convient également de noter que l'idée d'appliquer la dérivée fractionnaire de Caputo-Fabrizio sur des équations différentielles fractionnaires incertaines où des fonctions à valeurs d'intervalle sont impliquées. Dans le même temps, en combinant les dérivées de Riemann-Liouville et Caputo avec la dérivée $g r$, c'est-à-dire la dérivée granulaire de Riemann-Liouville et la dérivée granulaire de Caputo, des dérivées fractionnaires floues granulaires ont émergé dans \cite{Ref40}. L'intégrale fractionnaire floue granulaire a également été présentée dans ce travail. La principale raison de l'analyse des EDFF sous des dérivés fractionnaires flous granulaires est le fait que l'étude des EDFF sous les notions de dérivés fractionnaires flous qui sont définis sur la base des approches dérivées SGH, dérivées gH et, en général, basées sur FSIA vient avec quelques restrictions. De telles restrictions telles que la multiplicité des solutions et le phénomène UBM peuvent être surmontées par l'utilisation de dérivés fractionnaires flous granulaires basés sur l'approche RDM-FIA. Il a également été montré qu'une EDFF d'ordre $\beta \in(0,1)$ sous le concept de dérivées fractionnaires floues granulaires n'a qu'une seule solution. Inspirées des dérivées fractionnaires floues granulaires, l'intégrale q-fractionnelle de Riemann-Liouville granulaire et la dérivée q-fractionnelle de Caputo granulaire ont été présentées dans \cite{Ref41} pour l'étude des EDFF à l'échelle du temps.

Caputo-Katugampola gH-dérivé et RiemannLiouville-Katugampola gH-dérivé comme généralisations de
Dérivé Caputo gH et dérivé H de Riemann-Liouville en utilisant le concept Katugampola de Katugampola et $\mathrm{gH}$- dérivés ont été introduits en 2019. L'existence et l'unicité de solutions pour les FFDE sous la dérivée gH de Caputo-Katugampola par l'utilisation d'approximations successives sous la condition de Lipschitz généralisée ont été montrées dans \cite{Ref42}. Plus précisément, similaire à d'autres définitions de dérivées floues établies par gH-dérivé ou $\mathrm{gH}$-difference, il a été prouvé que, sous certaines conditions, une EDFF d'ordre $\beta \in(0,1)$ a deux solutions, c'est-à-dire des solutions correspondant aux première et deuxième formes de différentiabilité.

Par une combinaison de la dérivée q-fractionnelle de Caputo, c'est-à-dire une dérivée fractionnaire provenant du calcul quantique, et de la dérivée gH, la dérivée q-fractionnelle de Caputo $\mathrm{gH}$ a été introduite en 2019 et l'existence et l'unicité des solutions pour Les EDFF ont été démontrés par des conditions de type Krasnoselskii-Krein. Étant donné que cette approche sous-tend également les approches basées sur FSIA, la résolution d'un EDFF, sous certaines conditions, nous apporte plus d'une solution.

En 2020, le concept du dérivé gH Atangana-Baleanu en tant que combinaison du dérivé fractionnaire Atangana-Baleanu \cite{Ref43} et du dérivé gH a été rapporté dans \cite{Ref44}. Une des raisons pour proposer une telle dérivée a été basée sur la propriété de non-localité et de non-singularité du nouveau noyau défini dans la dérivée fractionnaire d'Atangana-Baleanu. La même année, en employant la dérivée fractionnelle conformable \cite{Ref45} et la dérivée SGH, les EDFF ont été étudiées sous le concept de dérivé SGH conformable (ou dérivée conforme floue) dans \cite{Ref46}, \cite{Ref47} et, presque simultanément, dans \cite{Ref48}. Il a été dit que, contrairement à certaines autres définitions de dérivées fractionnaires, la dérivée fractionnaire conforme nous permet d'avoir une définition de la différentiabilité d'une fonction de la même manière que la définition d'une dérivée vient de la limite de la fonction.

\chapter{Etude d’un problème fractionnaire flou et les applications des EDF´s}
Les équations différentielles fractionnaires ont été d'un grand intérêt récemment. Elle est causée à la fois par le développement intensif de la théorie du calcul fractionnaire lui-même et par les applications, bien que les outils du calcul fractionnaire soient disponibles et applicables à divers domaines d'études.

Ces dernières années, les perturbations quadratiques des équations différentielles non linéaires ont attiré beaucoup d'attention. Nous appelons ces équations différentielles équations différentielles hybrides. Il y a eu de nombreux travaux sur la théorie des équations différentielles hybrides, Dhage et Lakshmikantham \cite{Ref8} ont discuté de l'équation différentielle hybride du premier ordre suivante :
\begin{equation}
\left\{\begin{array}{l}
\frac{\mathrm{d}}{\mathrm{d} t}\left[\frac{u(t)}{f(t, u(t))}\right]=g(t, u(t)), \quad \text { a.e. } t \in J, \\
u\left(t_{0}\right)=u_{0} \in \mathbb{R}
\end{array}\right.
\label{5}
\end{equation}
où $$f \in C(J \times \mathbb{R}, \mathbb{R} \backslash\{0\})$$ Et $$g \in \mathcal{C}(J \times \mathbb{R}, \mathbb{R}).$$
Ils ont établi l'existence et l'unicité des résultats et certaines inégalités différentielles fondamentales pour les équations différentielles hybrides initiant l'étude de la théorie de tels systèmes et ont prouvé en utilisant la théorie des inégalités, son existence de solutions extrêmes et un résultat de comparaison.

Le but de ce chapitre est d'étudier l'existence et l'unicité de solution de l'équation hybride fractionnaire floue analogue to \ref{5} .\cite{Ref10}

\section{Quelques résultats pour les équations différentielles hybrides}
On appelle le résultat qui établit l'existence de solution pour l'équation différentielle hybride du premier ordre (en bref EDH) à condition initiale. Ce résultat sera utile dans l'étude du problème flou correspondant

Nous considérons le problème de la valeur initiale (\ref{5}). 

Par une solution du EDH (\ref{5}) nous voulons dire une fonction $u \in A C(J, \mathbb{R})$ tel que :

(i) la fonction $t \mapsto \frac{u}{f(t, \mu)}$ est absolument continu pour chacun $u \in \mathbb{R}$, et

(ii) u satisfait les équations de (\ref{5}),

où $A C(J, \mathbb{R})$ est l'espace des fonctions réelles absolument continues définies sur $J=[0, T]$.

\begin{theorem}
Soit $S$ un sous-ensemble non vide, fermé, convexe et borné de l'algèbre de Banach $ X $ et soit $A: X \rightarrow X$ et $B: S \rightarrow X$ être deux opérateurs tels que

(a) $A$ est $\mathcal{D}$-Lipschitz avec $\mathcal{D}$-fonction $\psi$,

(b) $B$ est complètement continu,

(c) $x=A x B y \Rightarrow x \in S$ pour tout $y \in S$, et

(d) $M \psi(r)<r$, où $M=\|B(S)\|=\sup \{\|B x\|: x \in S\} .$

Puis l'équation de l'opérateur $A x B x=x$ a une solution dans $S .$

Nous considérons les hypothèses suivantes dans ce qui suit.

$\left(A_{0}\right)$ La fonction $x \mapsto \frac{x}{f(t, x)}$ augmente en $\mathbb{R}$ presque partout pour $t \in J$.

$\left(A_{1}\right)$ Il existe une constante $L>0$ tel que

\begin{equation}
|f(t, x)-f(t, y)| \leq L|x-y|
\label{a1}
\end{equation}
pour tout $t \in J$ et $x, y \in \mathbb{R}$.

$\left(A_{2}\right)$ Il existe une fonction $h \in L^{1}(J, \mathbb{R})$ tel que
$$
|g(t, x)| \leq h(t) \quad t \in J
$$
\end{theorem}
 Dans la section suivante, nous considérons une équation différentielle floue qui est un analogue flou de (\ref{5}).
\section{Équation différentielle hybride floue}
Nous examinerons le probleme de valeur initiale :
\begin{equation}
\begin{array}{c}
\frac{d}{d t}\left[\frac{u(t)}{f(t, u(t))}\right]=g(t, u(t)) \quad t \in J \\
u(0)=u_{0} \in \mathbb{E}
\end{array}
\label{a2}
\end{equation}
Le principe d'extension de Zadeh conduit à la définition suivante de $f (t, u)$ et $g (t, u)$ quand sont des nombres flous
$$
\begin{array}{ll}
f(t, u)(y)=\sup \{u(x): y=f(t, x), & x \in \mathbb{R}\} \\
g(t, u)(y)=\sup \{u(x): y=g(t, x), & x \in \mathbb{R}\} .
\end{array}
$$
Il s'ensuit que

$$
\begin{array}{l}
{[f(t, u)]^{\alpha}=\left[\min \left\{f(t, x): x \in\left[u_{1}^{\alpha}, u_{2}^{\alpha}\right]\right\}, \max \left\{f(t, x): x \in\left[u_{1}^{\alpha}, u_{2}^{\alpha}\right]\right\}\right]} \\
{[g(t, u)]^{\alpha}=\left[\min \left\{g(t, x): x \in\left[u_{1}^{\alpha}, u_{2}^{\alpha}\right]\right\}, \max \left\{g(t, x): x \in\left[u_{1}^{\alpha}, u_{2}^{\alpha}\right]\right\}\right]}
\end{array}
$$

Pour $u \in \mathbb{E}$ avec $[u]^{\alpha}=\left[u_{1}^{\alpha}, u_{2}^{\alpha}\right], 0<\alpha \leq 1$. Nous appelons $u: J \rightarrow \mathbb{E}$ une solution floue de (\ref{a2}),

Si
$$
\left[\frac{d}{d t}\left[u(t) \div_{G} f(t, u(t))\right]\right]^{\alpha}=[g(t, u(t))]^{\alpha} \text { and }[u(0)]^{\alpha}=\left[u_{0}\right]^{\alpha}
$$

Pour tout $t \in J$ et $\alpha \in[0,1]$. Dénoter $\widetilde{f}=\left(f_{1}, f_{2}\right)$ et $\widetilde{g}=\left(g_{1}, g_{2}\right)$,
$$
\begin{array}{c}
f_{1}(t, u)=\min \left\{f(t, x): x \in\left[u_{1}, u_{2}\right]\right\}, f_{2}(t, u)=\max \left\{f(t, x): x \in\left[u_{1}, u_{2}\right]\right\} \text { and } \\
g_{1}(t, u)=\min \left\{g(t, x): x \in\left[u_{1}, u_{2}\right]\right\}, g_{2}(t, u)=\max \left\{g(t, x): x \in\left[u_{1}, u_{2}\right]\right\}(\text { resp }),
\end{array}
$$
Où $u=\left(u_{1}, u_{2}\right) \in \mathbb{R}^{2}$. Ainsi pour fixe $\alpha$ nous avons un problème de valeur initiale dans $\mathbb{R}^{2}$
\begin{equation}
\begin{array}{l}
\begin{aligned}
\frac{d}{d t}\left[\frac{u_{1}^{\alpha}(t)}{\widetilde{f}\left(t, u_{1}^{\alpha}(t), u_{2}^{\alpha}(t)\right)}\right] &=\widetilde{g}\left(t, u_{1}^{\alpha}(t), u_{2}^{\alpha}(t)\right) \\
u_{1}^{\alpha}(0) &=u_{01}^{\alpha}
\end{aligned}\\
\text { and }\\
\begin{aligned}
\frac{d}{d t}\left[\frac{u_{2}^{\alpha}(t)}{\widetilde{f}\left(t, u_{1}^{\alpha}(t), u_{2}^{\alpha}(t)\right)}\right] &=\tilde{g}\left(t, u_{1}^{\alpha}(t), u_{2}^{\alpha}(t)\right)\\
u_{2}^{\alpha}(0) &=u_{02}^{\alpha}
\end{aligned}
\end{array}
\label{3a}
\end{equation}
Si nous pouvons le résoudre (uniquement), nous n'avons qu'à vérifier que les intervalles $\left[u_{1}^{\alpha}(t), u_{2}^{\alpha}(t)\right], \alpha \in[0,1]$, définissons un nombre flou $u(t)$ dans $\mathbb{E}$. Depuis $f$ et $g$ sont supposés continuer et Caratheodory (resp), le problème de la valeur initiale (\ref{3a}) est équivalent à l'équation intégrale hybride non linéaire (EIH) suivante
\begin{equation}
u(t)=\tilde{f}(t, u(t))\left(\frac{u_{0}}{\widetilde{f}(0, u(0))}+\int_{0}^{l} \widetilde{g}(s, u(s)) d s\right)
\label{q3}
\end{equation}
\begin{theorem}
 Supposons que  $\operatorname{sign}(u(0))=\operatorname{sign}(u(t))$, pour tout $t \in J$, soit $z(t)=u \div_{G} f(t, u(t))$ et $0 \notin[f(t, u)]^{\alpha} \alpha \in[0,1]$ et
$$
r(t)=f(t, u(t)) \div_{G}\left(z(0)+\int_{0}^{t} g(s, u(s)) d s\right), \quad 0 \notin\left[z(0)+\int_{0}^{t} g(s, u(s))\right]^{\alpha}
$$
1. Si $z(0) * g(t, u)>0$ alors la fonction $u(t) \in A C\left((0, T], \mathbb{E}\right)$ est une solution floue de (\ref{a2})

2. \begin{itemize}
\item[•] Si $z(t)$ est $G_{i}$-division ,
\item[•] ou $z(t)$ est $G_{i}$-division et $r(t)$ est $G_{i}$-division
\item[•] ou si $z(t)$ est $G_{i}$-division et $z_{1}^{\alpha}(0) \leq 0 \leq z_{2}^{\alpha}(0)$
\end{itemize}
 Alors $u(t)$ est une solution floue.
\end{theorem}
\begin{proof}
Nous résolvons le problème de valeur initiale en $\mathbb{R}^{2}$
$$
\begin{array}{l}
\frac{d}{d t} z_{1}^{\alpha}=\min \left\{g(t, x): x \in\left[u_{1}^{\alpha}(t), u_{2}^{\alpha}(t)\right]\right\}, u_{1}^{\alpha}(0)=u_{01}^{\alpha} \\
\frac{d}{d t} z_{2}^{\alpha}=\max \left\{g(t, x): x \in\left[u_{1}^{\alpha}(t), u_{2}^{\alpha}(t)\right]\right\}, u_{2}^{\alpha}(0)=u_{02}^{\alpha}
\end{array}
$$
\paragraph*{Étape 1 :}
On peut supposer que $(\ref{a1})$ implique
\begin{equation}
\|f(t, x)-f(t, y)\| \leq L\|x-y\|, \quad \text { for all } t \in J, \quad x, y \in \mathbb{R}
\label{q1}
\end{equation}
où le $\|.\|$ est défini par $\|u\|=\max \left\{\left|u_{1}\right|,\left|u_{2}\right| .\right.$ Il est bien connu que (\ref{q1}) et les hypothèses sur $ g $ Théorème 15 garantissent l'existence et la dépendance continue de la solution de :
\begin{equation}
\left\{\begin{array}{l}
\frac{d}{d t}\left[\frac{u(t)}{\bar{f}(t, \mu(t))}\right]=\widetilde{g}(t, u(t)), \\
u(0)=u_{0}
\end{array}\right.
\label{q2}
\end{equation}
Et cela pour toute fonction continue $u_{0} \in \mathbb{R}^{2}$ on a $(\ref{q3})$.

En choisissant $u_{0}=\left(u_{01}^{\alpha}, u_{02}^{\alpha}\right)$ dans (\ref{q2}) nous obtenons une solution $u^{\alpha}(t)=\left(u_{1}^{\alpha}(t), u_{2}^{\alpha}(t)\right)$ à (\ref{q2}) pour tous $\alpha \in(0,1]$.
\paragraph*{Étape 2 :}
Nous montrerons que les intervalles $\left[u_{1}^{\alpha}(t), u_{2}^{\alpha}(t)\right], \alpha \in[0,1]$, définir un nombre flou $u(t) \in \mathbb{E} .$ Pour simplifier, supposons $[u(0)]^{\alpha} \leq 0,[f(t, u(t))]^{\alpha}>0$ et $[g(t, u(t))]^{\alpha}<0$ pour tout $\alpha \in[0,1]$ (La preuve pour les autres cas est similaire et omise), alors nous avons des cas de remorquage.
\begin{itemize}
\item[\color{blue}{Cas I:}] Par Eq. (\ref{3a}), nous avons les deux EDH suivantes avec des conditions initiales
\begin{equation}
\left\{\begin{array}{l}
\left.\frac{d}{d t}\left[\frac{u_{1}^{\alpha}(t)}{f_{2}^{\alpha}(t, u(t))}\right]=g_{1}^{\alpha}\left(t, u_{(} t\right)\right) \\
u_{1}^{\alpha}(0)=u_{01}^{\alpha}
\end{array}\right.
\label{4a}
\end{equation}
et
\begin{equation}
\left\{\begin{array}{l}
\frac{d}{d t}\left[\frac{u_{2}^{A}(t)}{f_{1}^{\alpha}(t, u(t))}\right]=g_{2}^{\alpha}(t, u(t)) \\
u_{2}^{\alpha}(0)=u_{02}^{\alpha}
\end{array}\right.
\label{5a}
\end{equation}
En conséquence à l'étape 1, nous en déduisons que, pour chaque $\alpha \in[0,1]$, la solution aux problèmes $ (\ref{4a})-(\ref{5a}) $ sont respectivement
$$
\begin{array}{l}
u_{1}^{\alpha}(t)=f_{2}^{\alpha}(t, u(t))\left[\frac{u_{1}^{\alpha}(0)}{f_{2}^{\alpha}(0, u(0))}+\int_{0}^{t} g_{2}^{\alpha}(s, u(s)) d s\right] \\
u_{2}^{\alpha}(t)=f_{1}^{\alpha}(t, u(t))\left[\frac{u_{2}^{\alpha}(0)}{f_{1}^{\alpha}(0, u(0))}+\int_{0}^{t} g_{1}^{\alpha}(s, u(s)) d s\right]
\end{array}
$$
On vérifie que $\left\{\left[u_{1}^{\alpha}(t), u_{2}^{\alpha}(t)\right], \alpha \in[0,1]\right\}$ représente l'ensemble de niveaux d'un ensemble flou $u(t)$ dans $\mathbb{E}$, pour chaque $t \in J$ fixé, en appliquant le théorème 3 (section 1.7). En effet, nous fixons $t \in J$ et vérifions la validité des trois conditions.

(1) : Tout d'abord, nous vérifions que $u_{1}^{\alpha}(t) \leq u_{2}^{\alpha}(t)$, pour chaque $\alpha \in[0,1]$ et $t \in J$, En effet, pour chacun $\alpha \in[0,1]$ et $t \in J$ on a  $f_{1}^{\alpha}(t, u(t)) \leq f_{2}^{\alpha}(t, u(t))$ et
$$
\frac{u_{1}^{\alpha}(0)}{f_{2}^{\alpha}(0, u(0))}+\int_{0}^{t} g_{1}^{\alpha}(s, u(s)) d s \leq \frac{u_{2}^{\alpha}(0)}{f_{1}^{\alpha}(0, u(0))}+\int_{0}^{t} g_{2}^{\alpha}(s, u(s)) d s
$$
Et par l'arithmétique classique, nous avons
$$
\begin{aligned}
u_{1}^{\alpha}(t) &=f_{2}^{\alpha}(t, u(t))\left[\frac{u_{1}^{\alpha}(0)}{f_{2}^{\alpha}(0, u(0))}+\int_{0}^{t} g_{1}^{\alpha}(s, u(s)) d s\right] \\
& \leq f_{1}^{\alpha}(t, u(t))\left[\frac{u_{2}^{\alpha}(0)}{f_{1}^{\alpha}(0, u(0))}+\int_{0}^{t} g_{2}^{\alpha}(s, u(s)) d s\right]=u_{2}^{\alpha}(t)
\end{aligned}
$$
(2) : Soit $0 \leq \alpha \leq \beta \leq 1$. Depuit $u_{0} \in \mathbb{E}$, on a $f_{2}^{\beta}(t, u(t)) \leq f_{2}^{\alpha}(t, u(t))$ et
$$
\frac{u_{1}^{\alpha}(0)}{f_{2}^{\alpha}(0, u(0))}+\int_{0}^{t} g_{1}^{\alpha}(s, u(s)) d s \leq \frac{u_{1}^{\beta}(0)}{f_{2}^{\beta}(0, u(0))}+\int_{0}^{t} g_{1}^{\beta}(s, u(s)) d s
$$
On en déduit que
$$
\begin{aligned}
u_{1}^{\alpha}(t) &=f_{2}^{\alpha}(t, u(t))\left[\frac{u_{1}^{\alpha}(0)}{f_{2}^{\alpha}(0, u(0))}+\int_{0}^{t} g_{1}^{\alpha}(s, u(s)) d s\right] \\
& \leq f_{2}^{\beta}(t, u(t))\left[\frac{u_{1}^{\beta}(0)}{f_{2}^{\beta}(0, u(0))}+\int_{0}^{t} g_{1}^{\beta}(s, u(s)) d s\right]=u_{1}^{\beta}(t)
\end{aligned}
$$
Et, de même, $f_{1}^{\alpha}(t, u(t)) \leq f_{1}^{\beta}(t, u(t))$ et
$$
\frac{u_{2}^{\beta}(0)}{f_{1}^{\beta}(0, u(0))}+\int_{0}^{t} g_{2}^{\beta}(s, u(s)) d s \leq \frac{u_{2}^{\alpha}(0)}{f_{1}^{\alpha}(0, u(0))}+\int_{0}^{t} g_{2}^{\alpha}(s, u(s)) d s
$$
Donc
$$
u_{2}^{\beta}(t)=f_{1}^{\beta}(t, u(t))\left[\frac{u_{2}^{\beta}(0)}{f_{1}^{\beta}(0, u(0))}+\int_{0}^{t} g_{2}^{\beta}(s, u(s)) d s\right]
$$
$$
\leq f_{1}^{\alpha}(t, u(t))\left[\frac{u_{2}^{\alpha}(0)}{f_{1}^{\alpha}(0, u(0))}+\int_{0}^{t} g_{2}^{\alpha}(s, u(s)) d s\right]=u_{2}^{\alpha}(t)
$$
Ce qui prouve que $\left[u_{1}^{\beta}(t), u_{2}^{\beta}(t)\right] \subseteq\left[u_{1}^{\alpha}(t), u_{2}^{\alpha}(t)\right]$.

(3) :Étant donné une suite non décroissante $\left\{\alpha_{i}\right\}$ dans $(0,1]$ telle que $\alpha_{i} \uparrow \alpha \in(0,1]$, nous prouvons que $\left[u_{1}^{\alpha}(t), u_{2}^{\alpha}(t)\right]=\bigcap_{i=1}^{\infty}\left[u_{1}^{\alpha_{i}}(t), u_{2}^{\alpha_{i}}(t)\right] .$

En effet, par le théorème de convergence dominée,
$$
\lim _{\alpha_{i} \uparrow \alpha} \int_{0}^{t} g_{1}^{\alpha_{i}}(s, u(s))=\int_{0}^{t} \lim _{\alpha_{i} \uparrow \alpha} g_{1}^{\alpha_{i}}(s, u(s)) d s=\int_{0}^{t} g_{1}^{\alpha}(s, u(s)) d s
$$
Et donc,
$$
\begin{aligned}
\lim _{\alpha_{i} \uparrow \alpha} u_{1}^{\alpha_{i}}(t) &=\lim _{\alpha_{i} \uparrow \alpha}\left(f_{2}^{\alpha_{i}}(t, u(t))\left[\frac{u_{1}^{\alpha_{i}}(0)}{f_{2}^{\alpha_{i}}(0, u(0))}+\int_{0}^{t} g_{1}^{\alpha_{i}}(s, u(s)) d s\right]\right) \\
&=\lim _{\alpha_{i} \uparrow \alpha}\left(f_{2}^{\alpha}(t, u(t))\left[\frac{u_{1}^{\alpha}(0)}{f_{2}^{\alpha}(0, u(0))}+\int_{0}^{t} g_{1}^{\alpha}(s, u(s)) d s\right]\right)=u_{1}^{\alpha}(t)
\end{aligned}
$$
D'où, $u(t) \in \mathbb{E}=\mathbb{R}_{\mathcal{F}}$.
\item[\color{blue}{Cas II:}]  Par Eq. (\ref{3a}), nous avons les deux EDH´s suivantes avec des conditions initiales
\begin{equation}
\left\{\begin{array}{l}
\left.\frac{d}{d t}\left[\frac{u_{2}^{\mu}(t)}{f_{1}^{\prime \prime}(t, u(t))}\right]=g_{1}^{\alpha}\left(t, u_{(} t\right)\right) \\
u_{2}^{\alpha}(0)=u_{02}^{\alpha}
\end{array}\right.
\label{c}
\end{equation}
et
\begin{equation}
\left\{\begin{array}{l}
\frac{d}{d t}\left[\frac{u_{1}^{x}(t)}{f_{2}^{\prime(t, u(t))}}\right]=g_{2}^{\alpha}(t, u(t)) \\
u_{1}^{\alpha}(0)=u_{01}^{\alpha}
\end{array}\right.
\label{cc}
\end{equation}
La solution aux problèmes $(\ref{c})-(\ref{cc})$ sont respectivement
$$
\begin{aligned}
u_{2}^{\alpha}(t) &=f_{1}^{\alpha}(t, u(t))\left[\frac{u_{2}^{\alpha}(0)}{f_{1}^{\alpha}(0, u(0))}+\int_{0}^{t} g_{1}^{\alpha}(s, u(s)) d s\right] \\
u_{1}^{\alpha}(t) &=f_{2}^{\alpha}(t, u(t))\left[\frac{u_{1}^{\alpha}(0)}{f_{2}^{\alpha}(0, u(0))}+\int_{0}^{t} g_{2}^{\alpha}(s, u(s)) d s\right]
\end{aligned}
$$
En appliquant l'étape 1 et nous considérons la situation où $0 \notin\left[z(0)+\int_{0}^{t} \widetilde{g}(s, u(s)) d s\right]^{\alpha}$,
\begin{equation}
\frac{f_{2}^{\alpha}(t, u(t))}{z_{1}(0)+\int_{0}^{t} g_{1}^{\alpha}(s, u(s)) d s} \leq \frac{f_{1}^{\alpha}(t, u(t))}{z_{2}(0)+\int_{0}^{t} g_{2}^{\alpha}(s, u(s)) d s}
\label{ccc}
\end{equation}
C'est à dire., $\left(u_{1}^{\alpha}(t) \leq u_{2}^{\alpha}(t)\right)$ de même en appliquant le théorème 3 (section 1.7), les détails du cas I sont analogues, et si la situation (\ref{ccc}) ne tient pas, c'est-à-dire, $\left(u_{2}^{\alpha}(t) \leq u_{1}^{\alpha}(t)\right)$, puis par théorème 3 (section 1.7), $ u(t)$ n'est pas une solution floue de (\ref{q2}).
\end{itemize}
\end{proof}
\begin{conclusion}
Nous avons étudié ce problème en applicant les concepts de division généralisée pour les nombres flous. La division G apliquée ici est un concept de division très général, étant également applicable dans la pratique. Développement de la théorie des équations différentielles hybrides avec condition floue impliquant leurs coupes de niveau compactes et convexes. La prochaine étape dans la direction de recherche proposée ici est d'étudier les équations différentielles fractionnaires floues hybrides avec la division G et leurs applications.
\end{conclusion}
\section{Applications des EDF´s}
Cette section présente certaines des applications proposées des EDF´s dans divers domaines scientifiques. Le point à souligner est que malgré tant de propositions, il n'y a pas de rapport sur les expérimentations réelles des applications des EDF´s. L'une des applications importantes des EDF´s proposées ces dernières années est tombée dans le champ de la théorie du contrôle. À cet égard, le contrôle optimal d'un système dynamique linéaire flou basé sur la différentiabilité gH et la différentiabilité SGH a été étudié dans \cite{Ref49}, \cite{Ref50}, \cite{Ref51}. La conception d'un contrôle de rétroaction optimal pour réguler un système dynamique linéaire flou avec une application proposée sur un Boeing 747 a été présentée dans \cite{Ref52} où les EDF´s ont été considérées sous le concept de $g r$-différentiabilité. le problème du contrôle optimal en temps flou par l'utilisation de la $g r$-différentiabilité a été étudié. Une analyse approfondie de la stabilité des systèmes dynamiques linéaires flous sous la notion de $g r$-différentiabilité a été rapportée dans \cite{Ref53}. Les critères de performance des systèmes dynamiques linéaires flous du second ordre et le contrôle de suivi flou des systèmes dynamiques linéaires flous sous le concept de $g r$-différentiabilité ont été étudiés, respectivement, dans \cite{Ref54}, \cite{Ref55}. A l'aide du concept de dérivée $g r$-Caputo, le problème flou du régulateur quadratique fractionnaire a été étudié dans \cite{Ref56} où une commande de rétroaction floue a été conçue. Pour une étude plus approfondie sur les applications des EDF´s d'ordre fractionnaire dans le contrôle et la contrôlabilité optimaux, se référer à \cite{Ref57}.

En effet, les applications des EDF´s ne se limitent pas à la théorie du contrôle, mais elles ont été proposées dans l'étude des biomathématiques, du diabète sucré, de la pression du liquide céphalo-rachidien, des processus hydrologiques, mouvement de l'eau dans une colonne horizontale , modèle épidémique, modèle de croissance démographique, loi de refroidissement de Newton, circuits électriques, et le cerveau tumeur. De plus, certaines des autres applications proposées des EDF´s peuvent être trouvées dans \cite{Ref58} pour un modèle prédateur-proie, dans \cite{Ref59} pour un modèle d'inventaire de production, dans \cite{Ref60} pour un investissement économique et dans \cite{Ref61} pour des feuilles de palmier à huile. Pour d'autres applications des EDF´s, se référer à \cite{Ref62}-\cite{Ref63}.

A l'aide d'une définition de la dérivée du premier ordre d'une fonction floue, les dérivées d'ordre supérieur de la fonction floue peuvent être déterminées, les EDF´s d'ordre supérieur sous diverses notions de dérivées floues ont été examinées dans la littérature. De plus, les équations aux dérivées partielles floues (EDPF) sous certains concepts de dérivées floues ont été étudiées dans de nombreuses études. Les EDF´s et EDPF´s d'ordre supérieur ont été exclus de cette enquête. En outre, il existe d'autres types d'équations différentielles incertaines telles que les équations différentielles floues aléatoires, les équations différentielles intégro floues et les équations intégro-différentielles fractionnaires floues.

Les équations différentielles à valeur d'intervalle peuvent être considérées comme un cas particulier d'équations différentielles floues où l'incertitude est considérée comme un intervalle dans lequel chaque membre a un degré complet d'appartenance à l'intervalle. De telles équations différentielles ont été traitées dans la littérature.
\chapter*{\begin{center}Conclusion et perspective \end{center}}
\addcontentsline{toc}{chapter}{Conclusion et perspective}
Nous avons, dans les pages qui précédent, présenter les sous-ensembles flous, les nombres flous et leur arithmétique. On a tenté d’introduire  les notions fondamentales de dérivation d’ordre fractionnaire. Nous avons présenté ses différentes définitions avec les approches de Riemann-Liouville, de Caputo, de Hilfer, de Grünwald-Letnikov et de Hukuhara, les équations différentielles floues d´ordre entier et d´ordre fractionnaires.

Nous avons étudié un problème fractionnaire flou qui se manifiste dans l´équation différentielle hybride, où nous donnons la forme explicite de la solution.  

Nous avons insisté sur les applications des EDF´s dans la théorie de contrôle et aussi dans les différents domaines scientifiques. 

Nous prévoyons dans le futur, en se basant sur la théorie des sous-ensembles floues et
les mêmes techniques mathématiques utilisés dans cette étude :
\begin{itemize}
\item[•] Étudier l'algèbre floue $(\mathcal{P},+, \odot, .)$ et donner ses propriétés,
\item[•] Étudier les conditions d'approximation d´un nombre flou,
\item[•] Donner un sens à la multiplication de deux nombres flous qui nous fait résoudre certaines équations différentielles aux valeurs initiales incertaines,
\item[•] Donner un sens à la différentiabilité  floue dans l'algèbre floue.
\item[•] Étudier les équations différentielles fractionnaires floues hybrides avec la division G et leurs applications.
\end{itemize}

\end{document}